\documentclass[12pt]{amsart}

\usepackage{amsmath,amsthm,amscd,amsfonts,wasysym,amssymb,epic,eepic,bbm,tikz-cd,mathtools,mathdots, comment,enumitem,fancyhdr,mathrsfs,multicol}
\usepackage[pagebackref,colorlinks=true,linkcolor=blue,citecolor=blue]{hyperref}
\usepackage[noabbrev]{cleveref}
\usepackage{MnSymbol}
\usepackage[all,cmtip]{xy}
\usepackage[normalem]{ulem} 
\usepackage[new]{old-arrows}

\allowdisplaybreaks
\newtheorem{thm}{Theorem}[section]
\newtheorem{cor}[thm]{Corollary}
\newtheorem{conj}[thm]{Conjecture}
\newtheorem{lem}[thm]{Lemma}

\newtheorem{prop}[thm]{Proposition}

\theoremstyle{definition}
\newtheorem{rem}[thm]{Remark}
\newtheorem{example}[thm]{Example}

\newcounter{remarkscounter}

\numberwithin{equation}{section}

\newcommand{\GL}{\mathrm{GL}}

\newcommand{\SL}{\mathrm{SL}}

\newcommand{\ZZ}{\mathbb{Z}}

\newcommand{\Ind}{\mathrm{Ind}}
\newcommand{\QQ}{\mathbb{Q}}

\newcommand{\lto}{\longrightarrow}

\newcommand{\OO}{\mathcal{O}}
\newcommand{\CC}{\mathbb{C}}
\newcommand{\RR}{\mathbb{R}}
\newcommand{\GG}{\mathbb{G}}

\newcommand{\Hom}{\mathrm{Hom}}
\newcommand{\End}{\mathrm{End}}
\newcommand{\tr}{\mathrm{tr}}

\DeclareMathOperator{\ind}{ind}

\newcommand{\quash}[1]{}

\theoremstyle{definition}
\newtheorem{defn}[thm]{Definition}

\newenvironment{psmatrix}
  {\left(\begin{smallmatrix}}
  {\end{smallmatrix}\right)}

\numberwithin{equation}{subsection}

\newcommand{\one}{\mathbbm{1}}

\newcommand{\norm}[1]{\left\lVert#1\right\rVert}
\newcommand{\Temp}{\mathrm{Temp}}
\renewcommand{\Im}{\mathop{\mathrm{Im}}}

\newcommand{\HP}{\mathop{\mathrm{HP}}\nolimits}

\newcommand{\uth}{\textcolor{red}{UTH}}

\allowdisplaybreaks

\begin{document}

\title{The $\rho$-Fourier transform}

\begin{abstract}
Let $G$ be a reductive group over a local field $F$ and let $\rho:{}^LG \to \GL_{V_{\rho}}(\CC)$ be a representation of its $L$-group satisfying suitable assumptions.
Braverman, Kazhdan and Ng\^o conjectured that one has a $\rho$-Fourier transform satisfying certain desiderata.  We give a uniform construction of the Fourier transform for arbitrary local fields.  This complements work of L.~Lafforgue in the non-Archimedean case, and proves a portion of the conjectures of Braverman, Kazhdan and Ng\^o.  Our methods are spectral in nature.
\end{abstract}

\author{Jayce R. Getz}
\address{Department of Mathematics\\
Duke University\\
Durham, NC 27708}
\email{jayce.getz@duke.edu}

\author{Armando Gutiérrez Terradillos}
\address{Department of Mathematics\\
Aarhus University\\
Aarhus, Denmark}
\email{armangute@math.au.dk}

\author{Farid Hosseinijafari}
\address{Department of Mathematics\\
Duke University\\
Durham, NC 27708}
\email{farid.hj@math.duke.edu}

\author{Aaron Slipper}
\address{Department of Mathematics\\
Duke University\\
Durham, NC 27708}
\email{aaron.slipper@duke.edu}

\author{Guodong Xi}
\address{Department of Mathematics \\ University of Minnesota}
\email{xi000023@umn.edu}

\author{HaoYun Yao}
\address{Department of Mathematics\\
Duke University\\
Durham, NC 27708}
\email{haoyun.yao@duke.edu}

\author{Alan Zhao}
\address{Department of Mathematics \\ Columbia University \\
New York, NY, 10027}
\email{asz2115@columbia.edu}


\thanks{Getz is thankful for partial support provided by NSF grant DMS-2400550,  Guti\'erez Terradillos is supported by a research grant (VIL54509) from VILLUM FONDEN, and Slipper is supported by NSF grant DMS-2231514. Any opinions, findings, and conclusions or recommendations expressed in this material are those of the authors and do not necessarily reflect the views of the National Science Foundation.}
\maketitle

{\small
\begin{multicols}{2}
\tableofcontents
\end{multicols}
}

\section{Introduction}

Let $G$ be a (connected) reductive group over a local field $F$ and let $\rho:{}^LG \to \GL_{V_\rho}(\CC)$ be a representation of its $L$-group.  We assume for the introduction that we are given an isomorphism  $ d:G/G^{\mathrm{der}} \tilde{\,\to\,} \GG_m,$  and that the dual map
$\rho \circ d^\vee:\CC^\times \to \GL_{V_{\rho}}(\CC)$
is the inclusion of $\CC^\times$ into the center.   
  
Let $\mathcal{R}:G(F) \times G(F) \to \mathrm{Aut}(L^2(G(F)))$ denote the regular representation.  Let $\mathcal{C}(G(F)) < L^2(G(F))$ be the Harish-Chandra space.  
The space $\mathcal{C}(G(F))$ is  usually referred to as the Harish-Chandra Schwartz space, but we omit ``Schwartz" for brevity and to distinguish it from other spaces we consider.  It is well-known that for $f \in \mathcal{C}(G(F))$ and tempered irreducible admissible representations $\pi$ of $G(F)$ the operator $\pi(f)$ is defined.
Let $\psi:F \to \CC^\times$ be a nontrivial character.  For any function $h:G(F) \to \CC,$ write $h^\vee(g):=h(g^{-1}).$

Our first main theorem is the following:
\begin{thm} \label{thm:main} Assume the desiderata on the local Langlands correspondence explained in \S \ref{sec:LLC}.
 There is a continuous linear operator
$\mathcal{F}_{\rho,\psi}:\mathcal{C}(G(F))\to \mathcal{C}(G(F))$
that satisfies 
\begin{align} \label{operator:bound}
\pi(\mathcal{F}_{\rho,\psi} (f)^\vee) = \gamma(\tfrac{1}{2},\pi,\rho,\psi)\pi(f).
\end{align}
for all tempered irreducible representations $\pi$ of $G(F).$
Here the $\gamma$-factor is defined using the local Langlands correspondence as in \eqref{temp:factors}. Moreover, $\mathcal{F}_{\rho,\psi}$ extends to a unitary operator on $L^2(G(F))$ that satisfies 
\begin{align*}
    \mathcal{F}_{\rho,\psi} \circ \mathcal{R}(g_1,g_2)&=\mathcal{R}(g_2,g_1) \circ \mathcal{F}_{\rho,\psi},\\
    \mathcal{F}_{\rho,\psi} \circ \mathcal{F}_{\rho,\overline{\psi}}&=\mathrm{id}_{\mathcal{C}(G(F))}.
\end{align*} 
\end{thm}
\noindent Since $\psi$ is usually fixed, we often drop it from notation, e.g.~$\mathcal{F}_{\rho}:=\mathcal{F}_{\rho,\psi}.$  We refer to $\mathcal{F}_{\rho}$ as the \textbf{$\rho$-Fourier transform}, or simply the Fourier transform.

    The desiderata on the local Langlands correspondence we require are rather weak, and are known in many cases as we explain in \S \ref{sec:LLC}.  For example, if one restricts attention to  $G=\GL_n$ then everything we require about the local Langlands correspondence is known.

Theorem \ref{thm:main} provides the Fourier transform conjectured to exist by Braverman-Kazhdan and Ng\^o \cite{BK-lifting,NgoSums,Ngo:Hankel}.  A nice summary is given in \cite{Luo:Ngo}. There are two points one must keep in mind when comparing our work and loc.~cit.~  First, we will always normalize the Fourier transform so that it is a unitary operator on $L^2(G(F)),$ not the space $L^2(G(F),|d|^{n_{\rho}}dg)$ considered in loc.~cit.  Here $n_{\rho} \in \QQ.$ One needs only twist functions and transforms by an appropriate power of $|d|$ to move between the two normalizations.

\begin{example} 
Let
\begin{align} \label{lin:FT} \begin{split}
\mathcal{F}:\mathcal{S}(M_n(F)) &\lto \mathcal{S}(M_n(F))\\
f &\longmapsto \left(X \mapsto \int_{M_n(F)}f(Y)\psi(\tr(XY))dY\right) \end{split}
\end{align}
be the Fourier transform on $M_n(F).$
When $G=\GL_n$ and $\rho$ is the standard representation, we will show in \S \ref{sec:GJ} that 
$\mathcal{S}(M_n(F))|\det|^{n/2} <\mathcal{C}(G(F))$ and   
$$
\mathcal{F}_{\rho}(f)=|\det|^{n/2}\mathcal{F}(f|\det|^{-n/2}).
$$
This normalization was considered by Stein in the case of $\GL_n(\CC)$ \cite[(1.1)]{Stein:SLnC}.
\end{example}

The second point one must keep in mind is that the Harish-Chandra space never appears in the conjectures of Braverman-Kazhdan and Ng\^o.  Instead, their conjectures assert the existence of a subspace $\mathcal{S}_{\rho}<C^\infty(G(F)) \cap L^2(G(F))$ that is preserved under $\mathcal{F}_{\rho}$ and satisfies additional desiderata \cite[Conjecture 2.1.1]{Luo:Ngo}.   One of the contributions of the current paper is to observe that for the purposes of constructing $\mathcal{F}_{\rho},$ one can work uniformly with $\mathcal{C}(G(F)),$ which is independent of $\rho.$  The problem of constructing $\mathcal{S}_{\rho}$ can then be reformulated as the problem of constructing subspaces of $\mathcal{C}(G(F))$ (which do depend on $\rho$) that are preserved by $\mathcal{F}_{\rho}$ and satisfy the desiderata.

Now let us explain the additional conjectures that we address in the current paper. Let $\mathcal{C}(\pi)$ denote the space of smooth matrix coefficients of $\pi.$  For $f \in C^\infty(G(F))$ and $c \in \mathcal{C}(\pi)$, let
\begin{align*}
    Z(s,f,c)= \int_{G(F)} f(g)|d(g)|^sc(g) dg,
\end{align*}
whenever this function converges absolutely.  We refer to $Z(s,f,c)$ as a \textbf{local zeta integral}.  
For $a<b$, let $V_{[a,b]}:=\{s\in\mathbb{C}: a\leq \mathrm{Re}(s)\leq b\}.$

\begin{conj}[Braverman-Kazhdan, Ng\^o] \label{conj:Sch:des}
There is a subspace $\mathcal{S}_\rho \leq C^\infty(G(F)) \cap L^2(G(F))$ dense in $L^2(G(F))$ and stable under $\mathcal{R}$ and $\mathcal{F}_{\rho}$ such that for all $f \in \mathcal{S}_{\rho}$ and tempered representations $\pi$ of $G(F):$
\begin{enumerate}[label=($\mathcal{S}$\text{\arabic*}), ref=$\mathcal{S}$\text{\arabic*}]
\item \label{Zetas} For all $c \in \mathcal{C}(\pi)$ the integral $Z(s,f,c)$ converges for $\mathrm{Re}(s)$ sufficiently large and 
\begin{align*}
\frac{Z(s,f,c)}{L(\tfrac{1}{2}+s,\pi,\rho)}
\end{align*}
is holomorphic as a function of $s.$  Moreover
$$
Z(-s,\mathcal{F}_{\rho}(f),c^\vee)=\gamma(\tfrac{1}{2}+s,\pi,\rho,\psi)Z(s,f,c)
$$
as meromorphic functions in $s.$
\quash{\item \label{BKN:nonarch:GCD} For $F$ non-Archimedean, the set $\{Z(s,f,\phi): \phi\in \mathcal{S}_\rho,\,\phi \in\mathcal{C}(\pi)\}$ is nonzero fractional ideal of $\mathbb{C}[q^{\pm s}]$ generated by $L(s,\pi,\rho).$}
    \item\label{BKN:arch:BEV}  For Archimedean $F,$ if $p\in\mathbb{C}[s]$ is chosen so that $s\mapsto p(s)L(\tfrac{1}{2}+s,\pi,\rho)$ has no pole on $V_{[a,b]}$, then $s\mapsto p(s)Z(s,f,c)$ is bounded on $V_{[a,b]}$.
\item \label{BKN:basic} For  non-Archimedean $F,$ if $\psi$ is unramified and $K \leq G(F)$ is a hyperspecial subgroup, then there is a function $b_{\rho} \in \mathcal{S}_{\rho}^{K \times K}$ such that $\mathcal{F}_{\rho}(b_{\rho})=b_{\rho},$ $Z(s,b_{\rho},c)$ is nonzero only if $\pi$ has a nonzero vector fixed by $K,$ and $Z(s,b_{\rho},c^\circ)=L(\tfrac{1}{2}+s,\pi,\rho)$ if $\pi$ has a nonzero vector fixed by $K.$  Here $c^\circ \in \mathcal{C}(\pi)^{K \times K}$ is the zonal spherical function.      
\end{enumerate}
\end{conj}

Up to the twist by $|d|$ explained earlier, this conjecture is part of \cite[Conjecture 2.1.1]{Luo:Ngo}.  To compare the conjectures, we point out that for technical reasons we only consider tempered representations in this paper, whereas in loc.~cit.~the authors consider all irreducible admissible representations. Moreover in loc.~cit.~the authors ask for an interpretation of $\mathcal{F}_{\rho}$ as a convolution satisfying certain properties. We omit a discussion of this in the introduction (but see \S \ref{ssec:kernel} below).

We prove Conjecture \ref{conj:Sch:des} under mild assumptions:

\begin{thm} \label{main:thm:intro} 
If the desiderata on the local Langlands correspondence in \S \ref{sec:LLC} are valid and $\rho$ is tempered then Conjecture \ref{conj:Sch:des} holds.
\end{thm}
\noindent 
Here we say $\rho$ is tempered if the image of $W_F$ under $\rho$ is bounded.

\begin{rem}
In the non-Archimedean case, L.~Lafforgue proved analogues of Theorem \ref{thm:main} and \ref{main:thm:intro} in \cite{LafforgueJJM,Lafforgue:transfer} via similar methods.  We discuss the relationship between our results and his at the end of the introduction.
\end{rem}

Before outlining the argument we discuss related literature.  There are several papers on geometrization of unramified or Iwahori-ramified $\rho$-Schwartz spaces in the equal characteristic case \cite{BNS,BBFK}, and a paper \cite{Luo:Ngo} constructing Fourier transforms for symmetric power representations of ${}^L\mathrm{GL}_2$ for most non-Archimedean fields.  It would be interesting to 
generalize these geometric approaches.  

In contrast to these previous works, our approach is direct and has the virtue that it proceeds essentially uniformly for all types of ramification and all places of a number field.
The key contribution of this paper is to prove that the Harish-Chandra space and the corresponding Plancherel formula are sufficient to construct Fourier transforms.  

The original motivation for constructing these Fourier transforms is to prove the corresponding global Poisson summation formula.   Our modest contribution to the global picture is that one can now state  the Poisson summation formula precisely without relying on a string of local conjectures.  

Let us simultaneously outline the argument and the sections of this paper.
We work with more general $G$ and $\rho$ than those considered above in the body of the paper.
In \S \ref{sec:prelim} we collect various preliminaries, culminating with a description of the matrical Harish-Chandra Plancherel formula in \S \ref{sec:HPmap}.  We then turn to a statement of a coarse version of the local Langlands correspondence in \S \ref{sec:LLC}.    This may be of interest to experts. Along the way we prove the compatibility of two potentially different definitions of $\gamma$-factors in the case of general linear groups (see Proposition \ref{prop:GLn}).  We also prove uniform bounds on $\gamma$-factors in the Archimedean case in \S \ref{LLCa:gamma:bound}.  These are important because we treat functions that are not necessarily finite under a maximal compact subgroup when $F$ is Archimedean.  Despite the technicality of these local results on $\gamma$-factors, they are of independent interest and will be helpful in other circumstances.   

We construct the Fourier transform in \S \ref{sec:FourieronHCSchwartz} and prove Theorem \ref{thm:main}.  Essentially we define the Fourier transform using the Harish-Chandra Plancherel theorem by declaring that  \eqref{operator:bound} is valid.   We construct the kernel function for the Fourier transform spectrally in \S \ref{ssec:kernel}, define the basic function in \S \ref{Section:Basic:Function}, and prove several estimates for later use in \S \ref{ssec:growth}.  

In \S \ref{sec:asymp} we construct the \textbf{asymptotic $\rho$-Schwartz space} $\mathcal{S}_{\rho}^{\mathrm{as}}(G(F)) < \mathcal{C}(G(F)).$  Briefly, it is the inverse image under the Harish-Chandra Plancherel isomorphism of meromorphic sections with poles no worse than the Langlands $L$-functions attached to $\rho.$  We refer to $\mathcal{S}_{\rho}^{\mathrm{as}}(G(F))$ as the asymptotic $\rho$-Schwartz space due to the fact that it is formally similar to the asymptotic Hecke algebra \cite{BK:asymp}, as pointed out to us by Beuzart-Plessis.  The asymptotic Hecke algebra satisfies all of the desiderata of Conjecture \ref{conj:Sch:des}.  We prove this in \S \ref{BKN:statement}.

The space $\mathcal{S}_{\mathrm{\rho}}^{\mathrm{as}}(G(F))$ is still not the ``correct" Schwartz space.  A clue to the correct definition is given by Example \ref{BP:example}.  It demonstrates that elements of $\mathcal{S}_{\rho}^{\mathrm{as}}(G(F))$ need not lie in the space of almost compactly supported smooth functions $C^{\infty}_{ac}(G(F))$ (see Definition \ref{defn:ac}).  On the other hand one expects the Schwartz space to consist of almost compactly supported smooth functions.  Thus when $F$ is non-Archimedean we define the \textbf{$\rho$-Schwartz space}
\begin{align*}
\mathcal{S}_{\rho}(G(F)):=\mathcal{S}_{\rho}^{\mathrm{as}}(G(F)) \cap C_{ac}^\infty(G(F)).
\end{align*}
One way to think about this definition is that $\mathcal{S}_{\rho}^{\mathrm{as}}(G(F))$ gives us spectral control of the Schwartz space and $C_{ac}^\infty(G(F))$ gives geometric control.  After a previous version of this paper was posted to the arXiv, our original work on $\mathcal{S}_{\rho}(G(F))$ was greatly extended in the follow up paper \cite{H:Y} by one of us and C-H.~Hsu.   We have decided to omit our original work on $\mathcal{S}_{\rho}(G(F)),$ and will instead  briefly indicate some of what they prove about $\mathcal{S}_{\rho}(G(F))$ in \S \ref{sec:rho:Sch}.  

Now we can be more precise about the relationship between our work in the non-Archimedean case and the work of L.~Lafforgue.  Roughly, Lafforgue constructs the Fourier transform by axiomatizing what one requires of $L$, $\varepsilon,$ and $\gamma$ factors, and then normalizing so that these factors agree with those of Langlands on representations whose supercuspidal support is equal to their discrete support \cite[\S 6]{Lafforgue:transfer}.  This has the virtue of making his results completely unconditional, but has the drawback that a priori his local factors may not be compatible with the local Langlands correspondence. More precisely, what one requires is \eqref{LLC:gamma:compat} below.
We prove \eqref{LLC:gamma:compat} when $G$ is a product of  general linear groups for arbitrary $\rho$ in Proposition \ref{prop:GLn}.  This makes our results unconditional in this setting.

\section*{Acknowledgements}
We thank R.~Beuzart-Plessis for supplying Example \ref{BP:example}, W.~Casselman and T.~Kaletha for help with Lemma \ref{lem:unram:func}, and P.~Delorme for answering a question regarding \cite{Delorme:constant}. Getz thanks D.~Kazhdan for pointing out two inaccuracies in a previous draft and for encouragement.  He also thanks H.~Hahn for her constant encouragement and help with editing.
This paper was written under the auspices of the Duke Research Scholars program organized by Getz and funded by NSF RTG grant DMS-2231514.

\section{Preliminaries}
\label{sec:prelim}

\subsection{Analytic number theory notation}
Let $X$ be a set. For functions $f,g:X \to \RR_{\geq 0}$, we write $
f \ll g$ or  $f(x) = O(g(x))$
if there exists a constant $C>0$ such that $f(x)\leq Cg(x)$ for all $x\in X$. We say $g$ dominates $f$ if $f\ll g$. Occasionally we write $f\ll_?g$ if the constant depends only on $?$ (up to irrelevant dependencies).
 We write $f \asymp g$ if $f \ll g$ and $g \ll f.$ By contrast, we say $f$ is bounded by $g$ if $f(x) \leq  g(x)$ for all $x \in X.$ 

For a real vector space $W$ and a subset $X\subseteq W,$ let
\begin{align}\label{cylinder}
    V_X := \{x+iw\mid x\in X,\,w\in W\} \subseteq W_\mathbb{C}
\end{align}
be the cylinder over $X.$

\subsection{Groups, characters and measures} 
Let $F$ be a local field of characteristic $0$ with normalized norm $|\cdot|.$  When $F$ is non-Archimedean, we let $\OO_F$ be its ring of integers. Let  $\varpi \in F^\times$ denote a choice of uniformizer, and set $q:=|\varpi|^{-1}.$

Let $G$ be a (connected) reductive group over $F$ and let $A_0 \leq G$ be a maximal split torus.  We denote by $A_G$ the maximal split torus in $Z_G.$
 
We set
\begin{align*}
    \mathfrak{a}_G:=\mathrm{Hom}_{\mathbb{Z}}(X^*(G),\mathbb{R}), \quad \mathfrak{a}_G^* = X^*(G)\otimes_{\mathbb{Z}}\mathbb{R}.
\end{align*} 
Define a homomorphism $H_G:G(F)\rightarrow \mathfrak{a}_{G}$ by
\begin{align} \label{HG}
    \langle \chi,H_G(g)\rangle = \log |\chi(g)|
\end{align}
for  $(g,\chi)\in G(F)\times X^*(G)$. Let
\begin{align*}
    G(F)^1 &:= \ker H_G = \{g\in G(F): |\chi(g)|=1\text{ for all }\chi\in X^*(G)\},\\
    \Lambda_G &:= \Hom_{\mathbf{TopGp}}(G(F)/G(F)^1,\mathbb{C}^\times).
\end{align*}   Thus $\Lambda_G$ is the set of unramified quasi-characters of $G(F).$
There is a surjective homomorphism 
\begin{align}\label{a->>UrChar} \begin{split}
    \mathfrak{a}_{G\CC}^* &\longrightarrow \Lambda_G\\
    \chi &\longmapsto (g \mapsto e^{\langle \chi,H_G(g) \rangle}). \end{split}
\end{align}
Let
\begin{align*}
    \mathop{\mathrm{Re}}\Lambda_G &:= \Hom_{\mathbf{TopGp}}(G(F)/G(F)^1,\mathbb{R}_{>0}),\\
    \mathop{\mathrm{Im}}\Lambda_G &:= \Hom_{\mathbf{TopGp}}(G(F)/G(F)^1,S^1).
\end{align*}
Here $S^1:=\{z \in \CC : |z|=1\}$ is the unit circle.
The restriction of \eqref{a->>UrChar} to $\mathfrak{a}_G^*$ induces an isomorphism onto $\mathop{\mathrm{Re}}\Lambda_G,$ and the restriction of \eqref{a->>UrChar} to $i\mathfrak{a}_G^*$ surjects onto  $\mathop{\mathrm{Im}}\Lambda_G.$  We let
\begin{align}
i\mathfrak{a}_{G}^\vee=\ker\left(i\mathfrak{a}_{G}^*\to\Im\Lambda_G\right)=\ker\left(\mathfrak{a}_{G\CC}^* \to \Lambda_G\right).
\end{align}
Thus $i\mathfrak{a}_{G}^\vee$ is trivial when $F$ is Archimedean, and is a full rank lattice in $i\mathfrak{a}_G^*$ when $F$ is non-Archimedean.
We say a function on $\Lambda_G$ is holomorphic if its pullback along \eqref{a->>UrChar} is holomorphic.

 When $F$ is non-Archimedean,
 there is a canonical identification of $\Lambda_G$ with the complex points of an algebraic torus \cite[\S V.2.4]{Renard}.
  The Lie algebra of this torus is $\mathfrak{a}_{G\CC}^*,$ and the map \eqref{a->>UrChar} is just the exponential map.   We obtain an isomorphism
  \begin{align} \label{alg:torus}
      \mathfrak{a}_{G \CC}^*/i\mathfrak{a}_{G}^\vee\,\tilde{\lto}\, \Lambda_G.
  \end{align}
Thus our earlier definition of a holomorphic function on $\Lambda_G$ coincides with the usual notion of a holomorphic function on a complex torus.
  
 For later use, we make explicit the following fact \cite[\S V.2.4]{Renard}:
\begin{lem}\label{chi(g):polynomial} Assume $F$ is non-Archimedean. For $g\in G(F)$, the map 
\begin{align*}
        \Lambda_G& \lto \mathbb{C}\\
        \chi & \longmapsto \chi(g)
    \end{align*}
is algebraic. \qed
\end{lem}

Following \cite[\S I.1]{Waldspurger:plancherel}, a parabolic subgroup $P$ of $G$ is called semi-standard if it contains $A_0$.  We denote by $P^{\mathrm{op}}$ the opposite parabolic with respect to $A_0.$ A semi-standard Levi subgroup $M$ is the unique Levi subgroup of a semi-standard parabolic $P$ that contains $A_0.$
 For a semi-standard Levi subgroup $M,$ let $\mathcal{P}(M)$ denote the set of all semi-standard parabolic subgroups with Levi subgroup $M.$ Let $\mathcal{M}$ denote a set of representatives for the set of conjugacy classes of semi-standard Levi subgroups. For any $P\in\mathcal{P}(M)$, by abuse of notation we write $H_M:P(F)\to\mathfrak{a}_M$ for the pullback of $H_M$ along $P(F)\to M(F)$.

 When $F$ is non-Archimedean, we normalize measures as in \cite[p.240]{Waldspurger:plancherel}. Namely, let $K$ be the special subgroup associated to the maximal split torus $A_0$. For any algebraic subgroup $H$ of $G$, we normalize the left Haar measure on $H(F)$ so that $K\cap H(F)$ has measure $1$. 

Assume $F$ is Archimedean. We normalize the measure following \cite[I, 1.6]{Arthur:HP:realreductive}.  Since $G(F)=\mathrm{Res}_{F/\RR}G(\RR)$, we may view $G(F)$ canonically as a real Lie group.
Choose a maximal compact subgroup $K$ of $G(F).$ 
Let $N$ be the unipotent radical of a minimal semi-standard parabolic subgroup.  Pick Haar measures on $\mathfrak{a}_{A_0}$ and on $\mathrm{Lie}(N)$ (viewed as a real Lie algebra).  We have
\begin{align}
G(F)=e^{\mathfrak{a}_{A_0}}N(F)K.
\end{align}
   We assume $dg=da\;dn\;dk$,  where the Haar measure $da$ (resp.~$dn$) is the pushforward of the Haar measure on $\mathfrak{a}_{A_0}$ (resp.~$\mathrm{Lie}(N)$) under the exponential map, and $\mathrm{meas}_{dk}(K)=1.$

For a split torus $A$, equip $\mathop{\mathrm{Im}}\Lambda_A$ with a Haar measure as follows. Take the counting measure on $i\mathfrak{a}_{A,F}^\vee$, and choose the Haar measure on $i\mathfrak{a}^*_{A}$ so that the quotient measure on $i\mathfrak{a}^*_{A}/2\pi i \kappa X^*(A)$ has volume $1$. Here
\begin{align*}
    \kappa = 
    \begin{cases}
        (\log q)^{-1} &\text{if $F$ is non-Archimedean,} \\
        2 &\text{if $F=\mathbb{R}$}.
    \end{cases}
\end{align*}
Give $\mathop{\mathrm{Im}}\Lambda_A$ the Haar measure so that the isomorphism  $i\mathfrak{a}^*_{A}/i\mathfrak{a}_{A,F}^\vee\cong \mathop{\mathrm{Im}}\Lambda_A$ is measure preserving. When $F$ is non-Archimedean this simply means $\mathop{\mathrm{Im}}\Lambda_A$ has volume $1$. 

The restriction map $\mathop{\mathrm{Im}}\Lambda_{G}\to \mathop{\mathrm{Im}}\Lambda_{A_G}$ is surjective with finite kernel. Equip $\mathop{\mathrm{Im}}\Lambda_{G}$ with the pullback of the Haar measure on $\mathop{\mathrm{Im}}\Lambda_{A_G}$ that we have just chosen.

\subsection{The tempered dual} \label{ssec:temp} Denote by 
\begin{align*}
    \Pi_2(G)\subseteq \mathrm{Temp}(G)
\end{align*}
the set of isomorphism classes of irreducible square integrable (resp. irreducible tempered) representations of $G(F)$.  Here and throughout this paper we work in the category of continuous unitary representations of $G(F)$ on a Hilbert space, and a representation of $G(F)$ is square integrable if its matrix coefficients are square integrable modulo center.

The group of unramified characters $\Lambda_G$ acts naturally on the set $\Pi(G)$ of isomorphism classes of irreducible admissible representations of $G(F).$ For $\pi\in \Pi(G),$ the \textbf{unramified twists} of $\pi$ are the representations
\begin{align*}
    \pi_{\chi}:=\pi \otimes\chi
\end{align*}
 for $\chi\in\Lambda_G.$
By pulling back along \eqref{a->>UrChar} we obtain an action of $\mathfrak{a}_{G\CC}^*:$  
\begin{align} \label{twists} \begin{split}
    \Pi(G) \times \mathfrak{a}_{G\CC}^* &\lto \Pi(G)\\
    (\pi,\lambda) &\longmapsto \pi_{\lambda}:=\pi \otimes e^{\langle \lambda, H_G(\cdot)\rangle}. \end{split}
\end{align}
An orbit under this action is sometimes called an inertial orbit.
Let 
\begin{align}
i\mathfrak{a}_{G,\pi}^\vee \leq i\mathfrak{a}_G^*
\end{align}
denote the stabilizer of $\pi$; we have $i\mathfrak{a}_{G}^\vee\leq i\mathfrak{a}_{G,\pi}^\vee$. 

Let $\Pi_2(G)_\mathbb{C}\subseteq\Pi(G)$ denote the set of all unramified twists of $\Pi_2(G)$. For each $\pi\in\Pi_2(G)$, the inertial orbit $\mathfrak{a}_{G\CC}^*.\pi$ may be identified with the complex manifold $\mathfrak{a}_{G\CC}^*/i\mathfrak{a}_{G,\pi}^\vee$. Thus the space $\Pi_2(G)_\mathbb{C}$ is naturally a disconnected complex manifold whose connected components are inertial orbits.  The unitary representations inside the inertial orbit of $\pi$ may be identified with the real manifold
$i\mathfrak{a}_G^*/i\mathfrak{a}_{G,\pi}^\vee.$
When $F$ is non-Archimedean, the inertial orbits are moreover complex algebraic tori.

Let $\mathrm{Temp}_{\mathrm{Ind}}(G)$ denote the set of isomorphism classes of tempered representations of $G(F)$ of the form $\mathrm{Ind}_P^G(\sigma)$, where $M\in\mathcal{M}$, $P\in\mathcal{P}(M)$ and $\sigma\in\Pi_2(M)$. The set $\mathrm{Temp}_{\mathrm{Ind}}(G)$ is a quotient of 
\begin{align*}
    \widetilde{\mathrm{Temp}}_{\mathrm{Ind}}(G) = \bigsqcup_{M\in\mathcal{M}}\Pi_2(M).
\end{align*}
Equip $\widetilde{\mathrm{Temp}}_{\mathrm{Ind}}(G)$ with the natural structure of a smooth manifold. For later use we let
\begin{align*}
    \widetilde{\mathrm{Temp}}_{\mathrm{Ind}}(G)_\mathbb{C} = \bigsqcup_{M\in\mathcal{M}}\Pi_2(M)_\mathbb{C}
\end{align*}
and equip it with its natural structure as a complex manifold. When $F$ is non-Archimedean, it is the $\CC$-points of a complex algebraic scheme, each of its connected components being a complex torus. Hence it makes sense to talk about polynomial and rational functions on $\widetilde{\mathrm{Temp}}_{\mathrm{Ind}}(G)_\mathbb{C}$.

We recall the Harish-Chandra canonical measure $d\sigma$ on $\Pi_2(G)$. For an orbit $\mathcal{O}$ in $\Pi_2(G)$ under the $\mathop{\mathrm{Im}}\Lambda_{G}$-action, the restriction of $d\sigma$ to $\mathcal{O}$ is the pushforward of the Haar measure on $\mathop{\mathrm{Im}}\Lambda_{G}$ along the action map $\mathop{\mathrm{Im}}\Lambda_{G}\to \mathcal{O}\subseteq \Pi_2(G)$. This allows us to define a regular Borel measure $d\pi$ on $\Temp_\Ind(G)$ by means of 
\begin{align*}
    \int_{\Temp_\Ind(G)}\varphi(\pi)d\pi:=\sum_{M\in\mathcal{M}}c_{G,M}\int_{\Pi_2(M)} \varphi(M,\sigma) \deg(\sigma)j(\sigma)^{-1} d\sigma,
\end{align*}
for any $\varphi\in C_c(\Temp_\Ind(G))$. Here $c_{G,M}$ is certain positive number depending on $G$ and $M$, $\deg(\sigma)$ is the formal degree of $\sigma$, and $j(\sigma)^{-1}$ is the Harish-Chandra $\mu$-function; see \cite[\S 2.4, 2.6]{BP:local:GGP} and the references therein.  

\begin{lem}\label{pi->pivee:preservemeasure} The map  $(M,\sigma)\mapsto (M,\sigma^\vee)$ defines an involutive homeomorphism on $\mathrm{Temp}_{\mathrm{Ind}}(G)$ that preserves the measure $d\pi$.
\end{lem}

\begin{proof} It is clear from the definition of the formal degree that $\deg(\sigma)=\deg(\sigma^\vee)$. Moreover,  $j(\sigma) = j(\sigma^\vee)$ by \cite[p.286 (2)]{Waldspurger:plancherel}.  Strictly speaking, in loc. cit.~only the non-Archimedean case is discussed. The proof used \cite[p.283 (11)]{Waldspurger:plancherel}. Its Archimedean counterpart remains true by \cite[Lemma 10.5.6]{Wallach:RGII}.
\end{proof}

To ease notation, for parabolic subgroups $P$ of $G$ with Levi subgroup $M$ and admissible representations $\sigma$ of $M(F)$, write
$$
I_P^G(\sigma):=\mathrm{Ind}_{P}^G(\sigma)
$$
for the unitarily normalized induction.  
By the Langlands classification, every irreducible tempered representation $\pi$ is a subrepresentation of some $I_P^G(\sigma)$ with $M\in\mathcal{M}$, $P\in\mathcal{P}(M)$, and $\sigma\in\Pi_2(M)$ (c.f. \cite[Proposition III.4.I]{Waldspurger:plancherel} and \cite[Proposition 5.2.5, \S 5.4]{WallachRG1}). It is unique in the sense that $\pi\mapsto (M,\sigma)$ determines a well-defined map
\begin{align}\label{Temp->TempInd}
    \mathrm{Temp}(G)\lto \mathrm{Temp}_{\mathrm{Ind}}(G).
\end{align}
We mention that each $I_P^G(\sigma)$ has finite length and decomposes as a direct sum of irreducible tempered representations. 

We will need the following fact from \cite[Corollary 5.1]{Pouslen:smoothInd} in the sequel:
\begin{lem}\label{Ind:evaluation:continuous} Assume $F$ is Archimedean. Let $P$ be a parabolic subgroup of $G$ with Levi subgroup $M$, and $\sigma$ be a unitary representation of $M$. Let $X_1,\ldots,X_n$ be a basis for $\mathrm{Lie}(\mathrm{Res}_{F/\mathbb{R}}G)$. For compact sets $\mathcal{K} \subset G(F)$ there exists a constant $C_{\mathcal{K}}>0$ such that
\begin{align*}
    \norm{f(g)} \leq C_{\mathcal{K}} \sum_{\substack{(\alpha_1,\ldots,\alpha_n)\in\mathbb{Z}^n_{\geq 0}\\\alpha_1+\cdots+\alpha_n \leq n}} \norm{I_P^G(\sigma)(X_1^{\alpha_1}\cdots X_n^{\alpha_n})f}
\end{align*}
for $(f,x)\in I_P^G(\sigma)^{\mathrm{sm}} \times \mathcal{K}$. In particular, the map $I_P^G(\sigma)^{\mathrm{sm}}\to \sigma^{\mathrm{sm}}$ given by $f\mapsto f(1)$ is continuous.
\qed
\end{lem}
\noindent Here the superscript $\mathrm{sm}$ denotes the subspace of smooth vectors, endowed with its usual Fr\'echet topology.

 Let $M\in\mathcal{M}$ and $P,\,P'\in\mathcal{P}(M).$  Moreover, let $\sigma$ be an admissible representation of $M(F)$ of finite length.  For $\lambda \in \mathfrak{a}_{M\CC}^*$ we denote by
\begin{align}
J_{P\mid P'}(\sigma_{\lambda}):I_{P'}^G(\sigma_{\lambda})^{\mathrm{sm}}\lto I_P^G(\sigma_{\lambda})^{\mathrm{sm}}
\end{align}
the usual $G(F)$-intertwining operator  \cite[Th\'eor\`eme IV.1.1]{Waldspurger:plancherel} \cite[10.1.11]{Wallach:RGII}.

\subsection{Functions on groups} Let $G$ be a connected reductive group over $F.$ Let $L^2(G(F))$ denote the Hilbert space of square integrable functions on $G$ with respect to a choice of Haar measure. Let
\begin{align}\label{RuG}
\begin{split}
    \mathcal{R}:L^2(G(F)) \times G(F) \times G(F) &\lto L^2(G(F))\\
(\varphi,(g_1,g_2))& \longmapsto (x \mapsto \varphi(g_1^{-1}xg_2))
\end{split}
\end{align}
be the action map.  We use the same notation for any $G(F) \times G(F)$-invariant space of functions on $G(F).$

When $F$ is non-Archimedean, let $\mathcal{S}(G(F)) = C_c^\infty(G(F))$. When $F$ is Archimedean, let $\mathcal{S}(G(F))$ be the Schwartz space of rapidly decreasing functions defined in \cite{Casselman:SchwartzSpace}; we recall the definition in \eqref{S(G(F)):Casselman:Frechet} below. It is a smooth Fr\'echet representation of $G(F)\times G(F)$ of moderate growth.

In the rest of this subsection, we recall the definition of the Harish-Chandra space $\mathcal{C}(G(F))$. 
Choose a closed immersion $\iota:G \to \GL_n$ and a norm $\norm{\,\cdot\,}_{M_n(F)}$ on the vector space of $n \times n$ matrices $M_n(F).$
Define a norm and a log-norm on $G(F)$ by
\begin{align} \label{norms} \begin{split}
\norm{g}:&=\max\left\{\norm{\iota(g)}_{M_n(F)},\norm{\iota(g)^{-t}}_{M_n(F)}\right\},\\
    \sigma_G(g):&=1+\log(\max\{1,\norm{g}\}). 
    \end{split}
\end{align}
We will only use $\norm{g}$ and $\sigma_G(g)$ to define growth rates of functions.  The spaces of functions we obtain in this manner are independent of the choice of embedding $\iota$ and norm.
 For convenience, we assume  that the norm is $K$-biinvariant.   For continuous $f:G(F) \to \CC$ and $d \in \RR$, set
\begin{align}\label{pd}
    p_d(f):=\sup_{g \in G(F)}|f(g)|\Xi^G(g)^{-1}\sigma_{G}(g)^d.
\end{align}
Here $\Xi^G$ is the Harish-Chandra $\Xi$-function \cite[\S1.5]{BP:local:GGP}.

Suppose $F$ is non-Archimedean and $K' \leq G(F)$ is a compact open subgroup. Let 
\begin{align*}
    \mathcal{C}(G(F)/\!/K')&=\{f \in C^\infty(G(F) /\!/ K'):  p_d(f)<\infty \textrm{ for all }d>0\},\\
    \mathcal{C}(G(F)):&=\bigcup_{K'}\mathcal{C}(G(F)/\!/K').
\end{align*}
The seminorms $(p_d)_{d>0}$ endow $\mathcal{C}(G(F)/\!/K')$ with a  Fr\'echet topology, and we equip $\mathcal{C}(G(F))$ with the locally convex colimit topology.  Thus, $\mathcal{C}(G(F))$ is a smooth LF-space representation of $G(F)\times G(F)$.

If $F$ is Archimedean, let $\mathfrak{g}:=\mathop{\mathrm{Lie}}(\mathrm{Res}_{F/\RR}G)$ and $U(\mathfrak{g})$ be the universal enveloping algebra.   The action of $G(F) \times G(F)$ on $C^\infty(G(F))$ induces an action of $U(\mathfrak{g})\times U(\mathfrak{g})$ on functions on $C^\infty(G(F)).$ For $d>0$ and $u,v \in U(\mathfrak{g})$ set
$$
p_{u,v,d}(f):=p_d(\mathcal{R}(u,v)f).
$$
Let
\begin{align*}
    \mathcal{C}(G(F)) := \left\{f\in C^\infty(G(F)): p_{u,v,d}(f)<\infty\text{ for all }d>0,\,u,v\in U(\mathfrak{g})\right\}.
\end{align*}
The seminorms $p_{u,v,d}$ make $\mathcal{C}(G(F))$ a nuclear Fr\'echet space. By \cite[Theorem 7.8.1]{WallachRG1}, it is a smooth Fr\'echet representation of $G(F)\times G(F)$.  

For this paragraph we continue to assume that $F$ is Archimedean.  It is clear that $C_c^\infty(G(F)) \leq \mathcal{S}(G(F))$ and the inclusion is continuous.  
Using standard estimates for $\Xi^G$ \cite[Proposition 1.5.1(i)]{BP:local:GGP} one checks that $\Xi^G(g)^{-1}\sigma_G(g)^d$ is dominated by a polynomial function of $g \in G(F).$ This implies that there are continuous inclusions
$$
C_c^\infty(G(F)) \;\leq \;\mathcal{S}(G(F))\; \leq \;\mathcal{C}(G(F)).
$$

For $f_1,f_2 \in\mathcal{C}(G(F))$, the convolution
\begin{align*}
    f_1\ast f_2(g):=\int_{G(F)} f_1(h)f_2(h^{-1}g) dh
\end{align*}
is absolutely convergent and $f_1 \ast f_2 \in\mathcal{C}(G(F))$; see \cite[\S7.1]{WallachRG1} and \cite[Lemme III.6.1]{Waldspurger:plancherel}. For $f\in \mathcal{C}(G(F))$, let
\begin{align*}
    f^*(g):=\overline{f(g^{-1})}.
\end{align*}
The map $f \mapsto f^*$ defines a topological involution on $\mathcal{C}(G(F))$ since $\sigma_G(g) \asymp \sigma_G(g^{-1})$ and $\Xi^G(g^{-1}) = \Xi^G(g).$

\subsection{Harish-Chandra Plancherel formula}\label{sec:HPmap} 
For a Hilbert representation $H$ of $G(F)$, let 
\begin{align}\label{HS:identification}
    \mathcal{HS}(H)=H^\vee \,\mathop{\widehat{\otimes}}\,H
\end{align}
denote the space of Hilbert-Schmidt operators on the Hilbert space $H$,  where $\mathop{\widehat{\otimes}}$ is the completed Hilbert space tensor product.  This space is naturally a Hilbert representation of $G(F)\times G(F)$.

Assume $(M,\sigma),(M',\sigma') \in \widetilde{\Temp}_{\mathrm{Ind}}(G)$ such that $I_P^G(\sigma)\cong I_{P'}^G(\sigma')$. The theory of intertwining operators provides  compatible families of equivariant isomorphisms
\begin{align}\label{HS:isom:intertwining}
    \mathcal{HS}(I_{P}^G(\sigma)) \cong \mathcal{HS}(I_{P'}^G(\sigma'))
\end{align}
The Archimedean case is discussed in \cite[\S II.(2.2)]{Arthur:HP:realreductive} and the non-Archimedean case is discussed in \cite[\S V.3]{Waldspurger:plancherel}. This allows us to define $\mathcal{HS}(\pi)$
unambiguously for $\pi\in\Temp_{\Ind}(G).$

Let $\mathcal{HS}$ (resp. $\mathcal{HS}^{\mathrm{sm}}$) denote the usual bundle over $\widetilde{\Temp}_\Ind(G)_\CC$ whose fiber over $\pi$ is $\mathcal{HS}(\pi)$ (resp. $\mathcal{HS}^{\mathrm{sm}}(\pi)$). 
For $\lambda\in \Lambda_M$, one has a $K$-equivariant isomorphism $I_P^G(\sigma_{\lambda})\vert_K \cong I^{K}_{P(F)\cap K}(\sigma\vert_{M(F)\cap K})$. This realizes $\mathcal{HS}$ and $\mathcal{HS}^{\mathrm{sm}}$ as (infinite dimensional) vector bundles over the measurable space $\widetilde{\Temp}_{\Ind}(G)_{\CC}$. 
We denote the pullbacks of $\mathcal{HS}$ and $\mathcal{HS}^{\mathrm{sm}}$ along $\widetilde{\mathrm{Temp}}_{\mathrm{Ind}}(G) \to \widetilde{\mathrm{Temp}}_{\mathrm{Ind}}(G)_\CC$ by the same symbols.  The bundles $\mathcal{HS}$ and $\mathcal{HS}^{\mathrm{sm}}$ over $\widetilde{\mathrm{Temp}}_{\mathrm{Ind}}(G)$ descend to vector bundles over $\mathrm{Temp}_{\mathrm{Ind}}(G)$ by \eqref{HS:isom:intertwining}.  We continue to denote these bundles by $\mathcal{HS}$ and $\mathcal{HS}^{\mathrm{sm}}.$

Put
\begin{align*}
    C^\infty&(\Temp_\Ind(G)) \\
    &= \begin{cases}
        \Gamma(\Temp_\Ind(G),\mathcal{HS}^{\mathrm{sm}})\cap C^\infty(\widetilde{\Temp}_{\Ind}(G),\mathcal{HS}^{\mathrm{sm}}) &\text{if $F$ is non-Archimedean},  \\
        \Gamma(\Temp_\Ind(G),\mathcal{HS})\cap C^\infty(\widetilde{\Temp}_{\Ind}(G),\mathcal{HS}) &\text{if $F$ is Archimedean}.
    \end{cases}
\end{align*}
For a section $T$ and $(M,\sigma)\in\widetilde{\Temp}_\Ind(G)$, we write $T(M,\sigma)=T(I_P^G(\sigma))$. 

The space $C^\infty(\Temp_{\Ind}(G))$ is naturally an algebra under pointwise composition \cite[Proposition A.3.1.(v)]{BP:local:GGP}, and for a section $T\in C^\infty(\Temp_\Ind(G))$, taking the adjoint pointwise defines another section $T^*$. For $g,h\in G(F)$ and $T\in C^\infty(\Temp_\Ind(G))$, we define the section $\mathcal{R}(g,h)T$ by
\begin{align*}
    \mathcal{R}(g,h)T(\pi) := \pi(g)\circ T \circ \pi(h)^{-1}.
\end{align*}
This defines an abstract representation $(\mathcal{R},C^\infty(\Temp_\Ind(G)))$ of $G(F)\times G(F)$.

For $M\in\mathcal{M}$ and $\sigma \in \Pi_2(M)$ define \begin{align} \label{Nsig}
    N(\sigma)=1+|\chi_{\sigma}|
\end{align}
as in \cite[\S 2.2]{BP:local:GGP}.  Here $\chi_{\sigma}$ is the infinitesimal character of $\sigma$ and the norm is defined as in loc.~cit.
Define 
\begin{align} \label{CTemp}
    \mathcal{C}(\Temp_{\Ind}(G)) \leq C^\infty(\Temp_\Ind(G))
\end{align}
to be the subspace of sections $T$ such that
\begin{enumerate}
    \item[(nA)] if $F$ is non-Archimedean, then $T\in \Gamma_c(\Temp_\Ind(G),\mathcal{HS}^{\mathrm{sm}})$; 
    \item[(A)] if $F$ is Archimedean, then for any $M\in\mathcal{M},$ any parabolic subgroup $P \in\mathcal{P}(M)$, and any invariant differential operator $D$ on $\Im\Lambda_M$, one has that
    \begin{align*}
        \sup_{\sigma\in \Pi_2(M)} \norm{I_P^G(\sigma)(u) \circ DT(M,\sigma) \circ I_P^G(\sigma)(v)}_{\mathrm{op}} N(\sigma)^k<\infty
    \end{align*}
    for all $k\geq 0$ and $u,v\in U(\mathrm{Lie}(K))$. Here $D$ is viewed as a differential operator on $\widetilde{\mathrm{Temp}}_{\mathrm{Ind}}(G)$ in the  obvious way.
\end{enumerate}
Here $\Gamma_c$ denotes the space of compactly supported continuous sections and $\norm{\cdot}_{\mathrm{op}}$ denotes the operator norm.
If $F$ is Archimedean, the topology on $\mathcal{C}(\mathrm{Temp}_{\mathrm{Ind}}(G))$ is defined by the seminorms in (A). When $F$ is non-Archimedean, it is topologized as in \cite[(2.6.3)]{BP:local:GGP}. 

It is not obvious that the definition in (A) is independent of the choice of $K.$  Thus we state and prove the following

\begin{lem}\label{HP:arch:bound} Let $F$ be Archimedean.  One has that $T\in \mathcal{C}(\Temp_{\Ind}(G))$ if and only if \begin{align*}
        \sup_{\sigma\in \Pi_2(M)} \norm{I_P^G(\sigma)(u) \circ DT(M,\sigma)\circ I_P^G(\sigma)(v)}_{\mathrm{op}} N(\sigma)^k<\infty
\end{align*}
for all $u,v\in U(\mathfrak{g})$ and all invariant differential operators $D$ on $\Im\Lambda_M$, where all unexplained notation is as in (A). 
\end{lem}
\begin{proof} The ``if'' direction is clear.  Thus assume $T \in \mathcal{C}(\mathrm{Temp}_{\mathrm{Ind}}(G)).$  
Combining \cite[Lemma 4.4.4.8]{Warner:semisimple:I} and \cite[p.211]{KnappSS}, it suffices to prove the bound in the statement of the lemma when $u,v$ are powers of $C_\mathfrak{g} - 2C_{\mathrm{Lie}(K)},$  where $C_{\mathfrak{g}}$ and $C_{\mathrm{Lie}(K)}$ are the Casimir elements of $\mathfrak{g}$ and $\mathrm{Lie}(K)$ (with respect to some symmetric bilinear form). By assumption it then suffices to prove the bound when $u,v$ are powers of $C_\mathfrak{g}$. The infinitesimal character of $I_P^G(\sigma)$ is the infinitesimal character $\chi_\sigma$ of $\sigma.$   Hence $C_\mathfrak{g}$ acts on the space of $I_P^G(\sigma)$ by a scalar bounded above by $N(\sigma)$. We deduce the desired bound for $T.$
\end{proof}

\begin{lem}\label{HP:Ginvariant} The space $\mathcal{C}(\Temp_\Ind(G))$ is a subalgebra of $C^\infty(\Temp_\Ind(G))$ invariant under the $\mathcal{R}$-action.
\end{lem}
\begin{proof} 
When $F$ is Archimedean, the $G(F)\times G(F)$-invariance follows from the identity $\pi(X)\pi(g) = \pi(g)\pi(\mathrm{Ad}(g^{-1})X)$, for $g\in G(F)$ and $X\in U(\mathfrak{g})$, and Lemma \ref{HP:arch:bound}.  All other claims are clear.
\end{proof}

The matrical Paley-Wiener theorem, a strong form of the Harish-Chandra Plancherel formula, is the following Theorem; see \cite{Waldspurger:plancherel} and \cite{Arthur:HP:realreductive} for the proofs. We follow the formulation in \cite[Theorem 2.6.1]{BP:local:GGP}.

\begin{thm}\label{HCPlan:thm1}
The map 
\begin{align*}
\mathrm{HP}_G:\mathcal{C}(G) &\lto \mathcal{C}(\Temp_{\Ind}(G))\\
f &\longmapsto (\pi \mapsto \pi(f))
\end{align*}
is a topological isomorphism. Its inverse sends a section $T\in \mathcal{C}(\Temp_{\Ind}(G))$ to
\begin{align*}
    \mathrm{HP}^{-1}_G(T)(g) := \int_{\Temp_{\Ind}(G)} \tr(\pi(g^{-1})\circ T(\pi)) d\pi.
\end{align*}
\end{thm}
\noindent The integral on the right is absolutely convergent by \cite[II, Theorem 6.3]{Arthur:HP:realreductive}.  We usually drop $G$ from notation, writing $\mathrm{HP}:=\mathrm{HP}_G.$

\begin{cor}\label{HCPlan:enhanced}  For $f\in\mathcal{C}(G)$, one has the inner product formula
\begin{align*}
    \int_{G(F)} |f(x)|^2 dx  = \int_{\mathrm{Temp}_{\mathrm{Ind}}(G)}\mathrm{tr}\,(\pi(f)^*\circ \pi(f)) d\pi.
\end{align*}
\begin{proof} Apply Theorem \ref{HCPlan:thm1} to $f^*\ast f$.\quash{
\begin{align*}
    \int_{G(F)} |f(x)|^2 dx = f^*\ast f(e) &= \int_{\mathrm{Temp}_{\mathrm{Ind}}(G)}\mathrm{tr}\,(\pi(f^*\ast f)) d\pi\\
    &=\int_{\mathrm{Temp}_{\mathrm{Ind}}(G)}\mathrm{tr}\,(\pi(f)^*\circ \pi(f)) d\pi.
\end{align*}}
\end{proof}
\end{cor}

For $g,h\in G(F)$ and $f\in \mathcal{C}(G(F))$, one has that $\pi(\mathcal{R}(g,h)f) = \pi(g)\circ \pi(f)\circ \pi(h^{-1}),$ so that
\begin{align} \label{HP:equi}
    \mathrm{HP}(\mathcal{R}(g,h)f) = \mathcal{R}(g,h) \mathrm{HP}(f).
\end{align}
This implies the map $\mathrm{HP}$ is $G(F)\times G(F)$-equivariant and $(\mathcal{R},\mathcal{C}(\Temp_\Ind(G)))$ is a continuous representation.

For a Hilbert space $H$ and $S\in\mathcal{HS}(H)$, let 
\begin{align} \label{Transp}
S^\intercal\in \mathcal{HS}(H^\vee)
\end{align}
denote the usual dual operator, which is given by $\phi\mapsto \phi\circ S.$ We choose to not write $S^\vee$ to distinguish from other $\vee$ that have already appeared.  We point out that $S^\intercal$ is again Hilbert-Schmidt and 
\begin{align} \label{FT:lem:easy:2}
    \mathrm{tr}(S^*S) = \mathrm{tr}((S^\intercal)^*S^\intercal).
\end{align}

We record some observations for future use:

\begin{lem}\label{FT:lem:easy:1} For $f\in \mathcal{C}(G(F))$ and $(\pi,V)$ a tempered representation of $G(F)$, one has that
\begin{align*}
    \pi^\vee(f^\vee) = \pi(f)^\intercal
\end{align*}
as operators on $V^\vee$. \qed
\end{lem}

\quash{\begin{proof} Let $v\in \pi$ and $w\in \pi^\vee$. Then
\begin{align*}
    \langle v,\pi(f)^\intercal w\rangle = \langle \pi(f)v,w\rangle &= \int_{G(F)} f(x)\langle \pi(x)v,w\rangle dx \\
    &= \int_{G(F)} f(x)\langle v,\pi^\vee(x^{-1})w\rangle dx \\
    &= \int_{G(F)} f^\vee(x)\langle v,\pi^\vee(x)w\rangle dx = \langle v,\pi^\vee(f^\vee) w\rangle.
\end{align*}
\end{proof}  }

\begin{lem}\label{HCPlan:transpose} For $T\in C^\infty(\Temp_{\Ind}(G))$, define a section $T^\intercal$ by
\begin{align*}
    T^\intercal(\pi) = T(\pi^\vee)^\intercal \in \mathcal{HS}(\pi).
\end{align*}
Then $T^\intercal\in C^\infty(\Temp_{\Ind}(G))$. If $T\in \mathcal{C}(\Temp_{\Ind}(G))$ then $T^\intercal\in \mathcal{C}(\Temp_{\Ind}(G))$. Moreover, if $f\in \mathcal{C}(G)$, then $\mathrm{HP}(f)^\intercal = \mathrm{HP}(f^\vee)$. \hfill\qedsymbol
\end{lem}
\quash{
\begin{proof} 
 The first assertion is clear, as $\Im\Lambda^M\ni \chi\mapsto T(M,\sigma^\vee\otimes\chi)\in\End_{\mathrm{sm}}(\Ind_{P\cap K}^K(\sigma^\vee\vert_{P\cap K}))$ and $\Im\Lambda^M\ni \chi\mapsto \chi^{-1}\in \Im\Lambda^M$ are smooth. If $T$ has compact support, then so does $T^\intercal$. When $F$ is Archimedean, the bound for $T^\intercal(M,\sigma)$ follows from the facts that $S\mapsto S^\intercal$ is norm-preserving, $N(\sigma)=N(\sigma^\vee)$ and 
\begin{align*}
    ((\Ind_P^G(\sigma))(u)D(T^\intercal(M,\sigma))(\Ind_P^G(\sigma))(v))^\intercal &= (\Ind_P^G(\sigma))(v)^\intercal (D(T(M,\sigma^\vee)^\intercal))^\intercal ((\Ind_P^G(\sigma))(u))^\intercal\\
    &=(\Ind_P^G(\sigma)^\vee)(v) DT(M,\sigma^\vee) (\Ind_P^G(\sigma)^\vee)(u)
\end{align*}
for $u,v\in U(\mathrm{Lie}(K)_\mathbb{C})$. The last statement follows from the following lemma.

\end{proof}}

\subsection{Constant terms}

Let $M\in\mathcal{M}$ and $P=MN\in\mathcal{P}(M)$. For $f\in \mathcal{C}(G(F))$, define the \textbf{constant term $f^P:M(F)\to\mathbb{C}$ of $f$ along $P$} by
\begin{align*}
    f^P(m) := \delta_P^{1/2
    }(m)\int_{N(F)} f(mn) dn.
\end{align*}
By \cite[III.6]{Waldspurger:plancherel} and \cite[\S7.2.1]{WallachRG1}, the integral is absolutely convergent and $f^P\in \mathcal{C}(M(F))$. Moreover,
\begin{align*}
    (\cdot)^P:\mathcal{C}(G(F))&\lto\mathcal{C}(M(F))
\end{align*}
is continuous.
\begin{lem}\label{constant:map:transitive} Let $P=MN\geq P'=M'N'$ be semi-standard parabolic subgroups of $G.$ For $f\in \mathcal{C}(G(F))$, one has that
$f^{P'} = (f^P)^{M\cap P'}.$ \qed
\end{lem}

\begin{lem}\label{constant:map:twist:equivariant} One has that
\begin{align*}
    (\mathcal{R}(m_1,m_2)f)^{P}=\delta^{1/2}_P(m_1)\delta^{1/2}_P(m_2)\mathcal{R}(m_1,m_2)f^P
\end{align*}
for $(m_1,m_2,f)\in M(F)^2 \times \mathcal{C}(G(F)).$ \qed
\end{lem}
Let $P=MN.$ For $(M',\sigma')\in \widetilde{\Temp}_\Ind(G)$ with $M'\leq M,$ there is a canonical isomorphism
\begin{align*}
    \iota_{P',P}:I_{P'(F)\cap K}^K(\sigma'\vert_{M'(F) \cap K})\tilde{\lto} I_{P(F)\cap K}^{K}\left(I_{ P'(F) \cap M(F) \cap K}^{ M(F) \cap K}(\sigma'\vert_{M'(F)\cap K})\right).
\end{align*}
This induces an $M(F)$-equivariant map 
\begin{align*}
    \mathrm{ev}_{P',P}:\mathcal{HS}(I_{P'(F)\cap K}^K(\sigma'\vert_{M'(F) \cap K}))^{\mathrm{sm}}&\lto \mathcal{HS}(I_{ P'(F) \cap M(F) \cap K}^{ M(F) \cap K}(\sigma'\vert_{M'(F)\cap K}))^{\mathrm{sm}}\\
    T&\longmapsto ((\iota_{P',P}\otimes \iota_{P',P})T)(1,1).
\end{align*}
Here we are using the identification \eqref{HS:identification}.  This in turn yields a map between sections
\begin{align}\label{constant:map:spectral}
\begin{split}    
    \mathcal{C}(\Temp_\Ind(G))&\lto \mathcal{C}(\Temp_\Ind(M))\\
    T&\longmapsto T^{P}:=\mathrm{ev}_{P',P}(T(M',\sigma')).
\end{split}
\end{align}

\begin{lem}\label{CTemp:HS:op:equiv:bound} Assume that $F$ is Archimedean. Let $T\in\mathcal{C}(\Temp_\Ind(G))$ and $(M,\sigma)\in\widetilde{\Temp}_\Ind(G)$. For $u,v\in U(\mathfrak{g})$ the operator $I_P^G(\sigma)(u)\circ T(M,\sigma)\circ I_P^G(\sigma)(v)$ is Hilbert-Schmidt. Moreover there exists $Z\in Z(\mathrm{Lie}(K))$ such that 
\begin{align*}
    \norm{I_P^G(\sigma)(u)\circ T(M,\sigma)\circ I_P^G(\sigma)(v)}_{\mathrm{HS}} \ll \norm{ I_P^G(\sigma)(Zu)\circ T(M,\sigma)\circ I_P^G(\sigma)(vZ)}_{\mathrm{op}}
\end{align*}
for all $\sigma \in \Pi_2(M)$ and $u,v \in U(\mathfrak{g}).$
\end{lem}
\begin{proof}  We follow the proof of \cite[Lemme 11]{Delorme:PaleyWiener}. Let $\widehat{K}$ denote the unitary dual of $K$. For $\delta\in\widehat{K}$ and $\sigma\in\Pi_2(M)$, let $p_\delta$ denote the orthogonal projection of $I_P^G(\sigma)$ onto its $\delta$-isotypic component. Then $\mathrm{id} = \sum_{\delta} p_\delta$ converges pointwise with respect to the Hilbert space topology  on $I_P^G(\sigma).$ 

Let $C_{\mathrm{Lie}(K)}\in Z(\mathrm{Lie}(K))$ be the element used in the proof of Lemma \ref{HP:arch:bound}, and let $\Omega = 1-C_{\mathrm{Lie}(K)}$. For each $\delta\in \widehat{K}$ the element $\Omega$ acts on the space of $\delta$ by a scalar $c(\delta)\geq 1$. Moreover, by \cite[Lemma 4.4.2.3]{Warner:semisimple:I} one has
\begin{align*}
    \sum_{\delta\in \widehat{K}} c(\delta)^{-k}(\dim\delta)^{1/2} < \infty
\end{align*}
for some large $k\geq 1$; we choose one such $k$ and let $Z := \Omega^k\in Z(\mathrm{Lie}(K))$. Then
\begin{align*}
&\norm{I_P^G(\sigma)(\Omega^{k}u)\circ T(M,\sigma)\circ I_P^G(\sigma)(v\Omega^{k})}_{\mathrm{op}} \\&\geq
    \norm{p_\delta\circ I_P^G(\sigma)(\Omega^{k}u)\circ T(M,\sigma)\circ I_P^G(\sigma)(v\Omega^{k})\circ p_\gamma}_{\mathrm{op}}\\
    &\gg_k c(\delta)^{k}c(\gamma)^k\norm{p_\delta\circ I_P^G(\sigma)(u)\circ T(M,\sigma)\circ I_P^G(\sigma)(v)\circ p_\gamma}_{\mathrm{op}}\\
    &\geq c(\delta)^{k}c(\gamma)^k(\dim\delta)^{-1/2}(\dim\gamma)^{-1/2}\norm{p_\delta\circ I_P^G(\sigma)(u)\circ T(M,\sigma)\circ I_P^G(\sigma)(v)\circ p_\gamma}_{\mathrm{HS}}.
\end{align*} 
Here in the last equality we used \cite[(4.5)]{Delorme:PaleyWiener}. The previous two bounds imply that
\begin{align} \label{FS}
    \sum_{(\delta,\gamma) \in \widehat{K}^2 }p_\delta\circ I_P^G(\sigma)(u)\circ T(M,\sigma)\circ I_P^G(\sigma)(v)\circ p_\gamma
\end{align}
defines an element in $\mathcal{HS}(I_P^G(\sigma))$ whose Hilbert-Schmidt norm is dominated by $$
\norm{I_P^G(\sigma)(\Omega^ku)\circ T(M,\sigma)\circ I_P^G(\sigma)(v\Omega^k)}_{\mathrm{op}}.
$$
By the remark in the first paragraph, \eqref{FS} is equal to $I_P^G(\sigma)(u)\circ T(M,\sigma)\circ I_P^G(\sigma)(v)$. 
\end{proof}

\begin{lem}\label{constant:map:spectral:welldefined} The map \eqref{constant:map:spectral} is well-defined.
\end{lem}
\begin{proof} The assertion is immediate when $F$ is non-Archimedean. Thus assume $F$ is Archimedean. Let $T\in\mathcal{C}(\Temp_\Ind(G))$. Let us retain the notation above \eqref{constant:map:spectral}. Let $u,v\in U(\mathrm{Lie}(M'(F)\cap K))$ and $D$ be an invariant differential operator on $\Im\Lambda_{M'}$. One has
\begin{align*}
    &\norm{I_{P'\cap M}^M(\sigma)(u)\circ D\mathrm{ev}_{P',P}(T(M',\sigma'))\circ I_{P'\cap M}^M(\sigma)(v)}_{\mathrm{op}}\\
    &\leq \norm{I_{P'\cap M}^M(\sigma)(u)\circ D\mathrm{ev}_{P',P}(T(M',\sigma'))\circ I_{P'\cap M}^M(\sigma)(v)}_{\mathrm{HS}}\\
    &=\norm{\mathrm{ev}_{P',P}\left(I_{P'}^G(\sigma)(u)\circ DT(M',\sigma')\circ I_{P'}^G(\sigma)(v)\right)}_{\mathrm{HS}}
\end{align*}

Let $X_1,\ldots,X_n\in\mathfrak{g}$ be a basis. By Lemma \ref{Ind:evaluation:continuous}, 
\begin{align*}
    &\norm{\mathrm{ev}_{P',P}\left(I_{P'}^G(\sigma)(u)\circ DT(M',\sigma')\circ I_{P'}^G(\sigma)(v)\right)}_{\mathrm{HS}}\\
    &\ll \sum_{\substack{\alpha,\beta\in \mathbb{Z}_{\geq 0}^n\\|\alpha|,|\beta|\leq n}}\norm{I_{P'}^G(\sigma)(X_1^{\alpha_1}\cdots X_n^{\alpha_n}u)\circ DT(M',\sigma')\circ I_{P'}^G(\sigma)(vX_1^{\beta_1}\cdots X_n^{\beta_n})}_{\mathrm{HS}}.
\end{align*}
By Lemma \ref{CTemp:HS:op:equiv:bound}, we can find $Z\in Z(\mathrm{Lie}(K))$ such that the above is dominated by
\begin{align*}
    \sum_{\substack{\alpha,\beta\in \mathbb{Z}_{\geq 0}^n\\|\alpha|,|\beta|\leq n}}\norm{I_{P'}^G(\sigma)(ZX_1^{\alpha_1}\cdots X_n^{\alpha_n}u)\circ DT(M',\sigma')\circ I_{P'}^G(\sigma)(vX_1^{\beta_1}\cdots X_n^{\beta_n}Z)}_{\mathrm{op}}
\end{align*}
The assertion now follows from Lemma \ref{HP:arch:bound}.
\end{proof}

\begin{lem}\label{constant:map:HP:nice} The diagram
\begin{center}
    \begin{tikzcd}
        \mathcal{C}(G(F))\arrow[d,"(-)^{P}"]\arrow[r,"\HP_G"]&\mathcal{C}(\Temp_\Ind(G))\arrow[d,"(-)^{P}"]\\
        \mathcal{C}(M(F))\arrow[r,"\HP_M"]&\mathcal{C}(\Temp_\Ind(M))
    \end{tikzcd}
\end{center}
commutes.
\end{lem}

\begin{proof} 
Explicitly, we are claiming that
\begin{align*}
    \mathrm{ev}_{P',P}(\HP_G(f)(M',\sigma')) = \HP_M(f^P)(M',\sigma').
\end{align*}
By transitivity of induction, we have 
\begin{align*}
    \mathrm{ev}_{P'\cap M,M}\circ \mathrm{ev}_{P',P} = \mathrm{ev}_{P',G}:\mathcal{HS}\left(I_{P(F) \cap K}^K(\sigma'|_{M'(F) \cap K})\right) \lto\mathcal{HS}(\sigma'|_{M'(F) \cap K})
\end{align*}
Thus it suffices to show
\begin{align*}
    \mathrm{ev}_{P',G}(\HP_G(\mathcal{R}(m_1,m_2)f)(M',\sigma')) = \mathrm{ev}_{P'\cap M,M}(\HP_M((\mathcal{R}(m_1,m_2)f)^P)(M',\sigma')).
\end{align*}
for all $m_1,m_2 \in M(F) \cap K.$ 

By \cite[Lemme VII.1.2]{Waldspurger:plancherel}, for $(M',\sigma')\in \widetilde{\Temp}_\Ind(G(F))$, one has
\begin{align*}
    \mathrm{ev}_{P',G}(\HP_G(f)(M',\sigma'))& = \HP_{M'}((f)^{P'})(M',\sigma')\\
    \mathrm{ev}_{P'\cap M,M}(\HP_M(f^P)(M',\sigma')) &= \HP_{M'}((f^P)^{P'\cap M})(M',\sigma').
\end{align*}
Strictly speaking, loc.~cit.~only treats the non-Archimedean case, but the proof is valid in the Archimedean case as well. In view of Lemma \ref{constant:map:transitive} this is the desired identity when $m_1=m_2=1.$  The identity for general $m_1,m_2$ now follows from   Lemma \ref{constant:map:twist:equivariant}.
\end{proof}

\section{The local Langlands correspondence} \label{sec:LLC}

We pin down a precise, weak version of the local Langlands correspondence for our purposes. For a  reductive group $G$ over a local field $F,$ we take $^LG = \widehat{G}(\mathbb{C})\rtimes W_F,$ the Weil form of the $L$-group.  If $M \leq G$ is a Levi subgroup then we obtain an $L$-map ${}^LM \to {}^LG$ in the standard manner \cite[\S 7.4]{Getz:Hahn}.

By a \textbf{representation of $^LG$}, we mean a continuous group homomorphism $\rho:{}^LG\to \GL_{V_\rho}(\mathbb{C})$ such that  $\rho\vert_{\widehat{G}(\mathbb{C})}$ is algebraic.  Here $V_{\rho}$ is a finite-dimensional $\CC$-vector space.

Let 
\begin{align*}
    W_F'=\begin{cases}
    W_F\times\SL_2(\mathbb{C}) &\text{ if $F$ is non-Archimedean,} \\
    W_F &\text{ if $F$ is Archimedean}. 
\end{cases}
\end{align*} 
An $L$-parameter $\phi:W_F'\to{}^LG$ is called \textbf{tempered} if the projection of $\phi(W_F)$ to $\widehat{G}(\mathbb{C})$ has a bounded image; i.e. is contained in a compact subset in the usual Hausdorff topology.  It is \textbf{discrete} if its image is not contained in any proper Levi subgroup of ${}^LG.$  Two $L$-parameters are equivalent if they are conjugate by elements in $\widehat{G}(\mathbb{C})$. Let 
\begin{align*}
    \Phi_2(G)\subseteq \Phi_t(G) \subseteq \Phi(G)
\end{align*}
be the set of equivalence classes of discrete $L$-parameters, tempered $L$-parameters, and $L$-parameters, respectively.

Let $|\cdot|:W_F \to \RR_{>0}$ be the usual norm, defined as in \cite[\S 1.4.6]{Tate_NT}.  We extend it to $W_F'$ by pullback along the projection $W_F' \to W_F.$

For 
\begin{align} \label{lambda:2:ways}
    \lambda \in \mathfrak{a}_{G\CC}^*=X^*(G) \otimes \CC \leq X_*((G/G^{\mathrm{der}})^{\wedge})_\mathbb{C} = X_*(Z_{\widehat{G}}^\circ)_\mathbb{C}
\end{align}
we can form the $L$-parameter 
\begin{align} \label{lambda} \begin{split}
|\cdot|^{\lambda}:W_F' &\lto {}^LG\\
x &\longmapsto |x|^{\lambda}. \end{split}
\end{align}
There is an action of $\mathfrak{a}_{G\CC}^*$ on $\Phi(G)$ given by
\begin{align} \label{ia:act} \begin{split}
\Phi(G) \times \mathfrak{a}_{G\CC}^* &\lto \Phi(G)\\
(\phi,\lambda) &\longmapsto \phi_{\lambda}:=\phi|\cdot|^{\lambda} .\end{split}
\end{align}
The restriction of the action to $i\mathfrak{a}_{G}^*$ preserves $\Phi_t(G)$ and $\Phi_2(G).$

By a \textbf{local Langlands correspondence}, we mean a map
\begin{align*}
    \mathrm{LL}_G:\Pi(G) \lto \Phi(G),
\end{align*}
satisfying a list of properties.
We make the list as short as possible to indicate precisely what we require.
\begin{enumerate}[label=(LLC\text{\arabic*}), ref=LLC\text{\arabic*}]
    \item\label{LLC:LCT} If $G$ is torus, $\mathrm{LL}_{G}$ is given by class field theory as in \cite[\S9]{Borel:Corvallis}.

    \item\label{LLC:GL} If $G=\GL_{N}$ when $F$ is non-Archimedean, $\mathrm{LL}_G$ is the correspondence constructed by Harris-Taylor and Henniart \cite{HT:LLC,Henniart:preuve}. 

    \item \label{LLC:real} When $F$ is Archimedean $\mathrm{LL}_G$ is the correspondence constructed by Langlands \cite{Langlands:ArchLLC}.

\item\label{LLC:temp} One has that $\mathrm{LL}_G^{-1}(\Phi_t(G))=\Temp(G).$
    \item\label{LLC:UT} 
    The map $\mathrm{LL}_G$ is $\mathfrak{a}_{G\CC}^*$-equivariant.   
    \item\label{LLC:Cg} If $\rho:{}^LG\to\GL_{V_\rho}(\mathbb{C})$ is a representation, then $\rho\circ \mathrm{LL}_G(\pi^\vee)\cong (\rho\circ\mathrm{LL}_G(\pi))^\vee$.  
    \item\label{LLC:para} Let $(M,\sigma)\in\widetilde{\Temp}_\Ind(G)$ and $\pi$ an irreducible subquotient of $I_P^G(\sigma)$. Then 
    $$\mathrm{LL}_G(\pi) = ({}^LM\to {}^LG)\circ \mathrm{LL}_M(\sigma).
    $$
    
\end{enumerate}

\begin{rem}

\begin{enumerate}
    \item In the Archimedean setting the local Langlands correspondence constructed by Langlands \cite{Langlands:ArchLLC} satisfies all the conditions above; see \cite[\S 11]{Borel:Corvallis} for a survey.
    
    \item Due to the work of Fargues and Scholze \cite[Theorem I.9.6]{FS}, one has a map $\mathrm{LL}_G$ satisfying \eqref{LLC:LCT} and \eqref{LLC:UT} when $F$ is non-Archimedean.  It also satisfies an analogue of \eqref{LLC:GL}, but only for supercuspidal representations.  The Fargues-Scholze map does not give information about the $\SL_2(\CC)$-factor on the Galois side.  
    \item If $G=\GL_n$ then the local Langlands correspondence constructed by Harris-Taylor and Henniart satisfies \eqref{LLC:LCT} (when $G=\GL_1$), \eqref{LLC:GL}, \eqref{LLC:temp}, \eqref{LLC:UT}, \eqref{LLC:Cg} and \eqref{LLC:para}.
    \item If one is only interested in constructing Schwartz spaces for a particular group $G,$ then our argument only requires
a local Langlands correspondence satisfying the conditions  \eqref{LLC:LCT}, \eqref{LLC:GL}, \eqref{LLC:real}, \eqref{LLC:temp}, \eqref{LLC:UT}, \eqref{LLC:Cg} and \eqref{LLC:para}
for all Levi subgroups of $G.$ 
\end{enumerate}
\end{rem}

\noindent We henceforth assume
\eqref{LLC:LCT}, \eqref{LLC:GL}, \eqref{LLC:real}, \eqref{LLC:temp}, \eqref{LLC:UT}, \eqref{LLC:Cg} and \eqref{LLC:para}.

The local Langlands correspondence for $\GL_n$ in \eqref{LLC:GL} 
tells us that $L$-functions, $\varepsilon$-factors, and $\gamma$-factors defined either from the automorphic or Galois theoretic perspective agree \cite[\S 12.4]{Getz:Hahn}.  In loc.~cit.~this is stated in terms of Rankin-Selberg $L$-functions, which agree with Godement-Jacquet $L$-functions when both are defined; see \cite[Theorem 5.1]{Jacquet:GLn:Corvallis} for the Archimedean case and \cite[\S 5]{JPSS:Conv} for the non-Archimedean case.

We say that a representation
\begin{align} \label{tempered}
    \rho:{}^LG\lto \GL_{V_\rho}(\mathbb{C})
\end{align}
is \textbf{tempered} if the image of $W_F$ is bounded.  We point out that this implies that $\rho$ is semisimple.
The local factors attached to a tempered representation $\rho$ are defined as follows. Fix a nontrivial character $\psi:F\to\mathbb{C}^\times$. Consider the sequence of maps
\[\begin{tikzcd}
    \mathrm{Temp}(G)\arrow[r,"\mathrm{LL}_G"]&\Phi_t(G)\arrow[r,"\rho "]&\Phi_t(\GL_n).
\end{tikzcd}\]
For $\pi\in\mathrm{Temp}(G)$, put
\begin{align} \label{temp:factors} \begin{split}
    L(s,\pi,\rho) &:= L(s,\rho\circ\mathrm{LL}_G(\pi)),\\
    \varepsilon(s,\pi,\rho,\psi) &:= \varepsilon(s,\rho\circ\mathrm{LL}_G(\pi),\psi),\\
    \gamma(s,\pi,\rho,\psi) &:= \gamma(s,\rho\circ\mathrm{LL}_G(\pi),\psi).\end{split}
\end{align}
\quash{
\item\label{LLC:discrete} Let $(M,\sigma) \in \widetilde{\mathrm{Temp}}_{\ind}(G),$ $P \in \mathcal{P}(M),$ and assume that $\pi \in \Pi_2(G)$ is a subquotient of $I_P^G(\sigma_{\lambda})$ where $\lambda \in \mathfrak{a}_M^*.$  Then 
    $\mathrm{LL}(\pi)(w,g)=\mathrm{LL}(\sigma)(w)\psi_{\sigma,\lambda}(g),$ where $\psi_{\sigma,\lambda}:W_F \times \SL_2(\CC) \lto {}^LG$ is defined as in \cite[\S 6.2]{HeiermanN:Unip}.
    }

We require the following compatibility assumption:
\begin{enumerate}[label=(LLC\text{\arabic*}), ref=LLC\text{\arabic*}]\setcounter{enumi}{7}
\item\label{LLC:gamma:compat}
Assume $F$ is non-Archimedean.  Let $M\in \mathcal{M}$, $P\in\mathcal{P}(M)$, and let $\sigma$ be a unitary supercuspidal representation of $M(F)$ and $\lambda_0\in \mathfrak{a}^*_{M}$. If $\pi\in\Pi_2(G)$ is a quotient of $I_P^G(\sigma_{\lambda_0})$, then the meromorphic function $\gamma(\tfrac{1}{2},\sigma_\lambda,\rho\vert_{^LM})$ is analytic at $\lambda=\lambda_0$ and
\begin{align*}
    \gamma(\tfrac{1}{2},\pi,\rho,\psi) = \gamma(\tfrac{1}{2},\sigma_{\lambda},\rho\vert_{^LM},\psi)|_{\lambda=\lambda_0}.
\end{align*}
\end{enumerate}

In the following we show that $\GL_n$ satisfies \eqref{LLC:gamma:compat}. Let $a \in \ZZ_{\geq 0}$ and let $\phi:W_F \to \GL_{n/a}(\CC)$ be an irreducible representation.  
Let
\begin{align} \label{Delta} \begin{split}
    \Delta:W_F &\lto W_F \times \SL_2(\CC)\\
    w &\longmapsto \left(w,\begin{psmatrix} |w|^{1/2} & \\ & |w|^{-1/2}\end{psmatrix}\right). \end{split}
\end{align}
Then we have $L$-parameters
\begin{align*}
\phi \otimes \mathrm{Sym}^{a-1}:W_F \times \SL_2(\CC) \lto \GL_n(\CC),\\
\phi \circ \mathrm{Sym}^{a-1} \circ \Delta=\bigoplus_{i=0}^{a-1} \phi |\cdot|^{(a-1)/2-i}:W_F \lto \GL_n(\CC).
\end{align*}
We extend the second homomorphism trivially to $\SL_2(\CC)$ to regard it as an $L$-parameter.

\begin{lem} \label{lem:L:param:discrete}
Assume $\phi$ is tempered.  Then the function $ \gamma\left(\tfrac{1}{2},\bigoplus_{i=0}^{a-1} \phi |\cdot|^{\lambda_i+(a-1)/2-i} ,\psi\right)$
is holomorphic in a neighborhood of $\lambda=0$ and 
$$
\gamma(\tfrac{1}{2},\phi \otimes \mathrm{Sym}^{a-1},\psi)=\gamma\left(\tfrac{1}{2},\bigoplus_{i=0}^{a-1} \phi |\cdot|^{(a-1)/2-i} ,\psi\right).
$$
\end{lem}

\begin{proof} 
Let $\mathcal{I} <W_F$ be the inertial group and let $\mathrm{Fr} \in W_F$ be a Frobenius element.  Then using \cite[(11)]{Gross:Reeder},
\begin{align*}
\gamma\left(\tfrac{1}{2},\phi \otimes \mathrm{Sym}^{a-1},\psi \right)=\frac{\varepsilon(\tfrac{1}{2},\phi,\psi)^a\det (-\phi(\mathrm{Fr})| (\CC^{n/a})^{\mathcal{I}})^{a-1}
\det(I-q^{-a/2}\phi^\vee(\mathrm{Fr})| (\CC^{n/a})^{\mathcal{I}})^{-1}
}{\det(I-q^{-a/2}\phi(\mathrm{Fr})| (\CC^{n/a})^{\mathcal{I}})^{-1}}.
\end{align*}
On the other hand 
\begin{align*}
&\gamma\left(\tfrac{1}{2},\bigoplus_{i=0}^{a-1} \phi |\cdot|^{(a-1)/2+i},\psi\right)\\
&=\prod_{i=0}^{a-1}\frac{\varepsilon\left(\tfrac{1}{2},\phi,\psi\right) \det(I-q^{-a/2+i}\phi^\vee(\mathrm{Fr})|(\CC^{n/a})^{\mathcal{I}})^{-1}}{\det(I-q^{-a/2+i}\phi(\mathrm{Fr})|(\CC^{n/a})^{\mathcal{I}})^{-1}}\\
&=\varepsilon\left(\tfrac{1}{2},\phi,\psi \right)^a\frac{\det(I-q^{-a/2}\phi^\vee(\mathrm{Fr})|(\CC^{n/a})^{\mathcal{I}})^{-1}}{\det(I-q^{-a/2}\phi(\mathrm{Fr})|(\CC^{n/a})^{\mathcal{I}})^{-1}} 
\frac{\prod_{i=1}^{a-1}\det(I-q^{-a/2+i}\phi^\vee(\mathrm{Fr})|(\CC^{n/a})^{\mathcal{I}})^{-1}}{\prod_{i=1}^{a-1}\det(I-q^{-a/2+i}\phi(\mathrm{Fr})|(\CC^{n/a})^{\mathcal{I}})^{-1}}.
\end{align*}
To complete the proof we observe that 
\begin{align*} 
&\frac{\prod_{i=1}^{a-1}\det(I-q^{-a/2+i}\phi^\vee(\mathrm{Fr})|(\CC^{n/a})^{\mathcal{I}})^{-1}}{\prod_{i=1}^{a-1}\det(I-q^{-a/2+i}\phi(\mathrm{Fr})|(\CC^{n/a})^{\mathcal{I}})^{-1}}\\&=\frac{\prod_{i=1}^{a-1}\det(-q^{a/2-i}\phi(\mathrm{Fr})|(\CC^{n/a})^{\mathcal{I}})\det(I-q^{a/2-i}\phi(\mathrm{Fr})|(\CC^{n/a})^{\mathcal{I}})^{-1}}{\prod_{i=1}^{a-1}\det(I-q^{-a/2+i}\phi(\mathrm{Fr})|(\CC^{n/a})^{\mathcal{I}})^{-1}}\\
&=\det (-\phi(\mathrm{Fr})|(\CC^{n/a})^{\mathcal{I}})^{a-1}
\end{align*}
Here we have changed variables $i \mapsto a-i$ in the numerator.
\end{proof}

\begin{prop} \label{prop:GLn} 
If $G=\GL_n$  then  \eqref{LLC:gamma:compat} is true.
\end{prop}

\begin{proof}
Let $a|n$ and let $M$ be the semi-standard Levi subgroup isomorphic to $\GL_{n/a}^a.$  Let $\sigma$ be a unitary supercuspidal representation of $\GL_{n/a}(F).$  Then $\sigma^{\otimes a}$ is a unitary supercuspidal representation of $\GL_{n/a}^a(F).$  Let $\lambda_a \in \mathfrak{a}_{M}^*$ be the point such that 
$$
e^{\left\langle\lambda_a , H_M\begin{psmatrix} m_1 & & \\ & \ddots & \\ & & m_a \end{psmatrix}\right\rangle}=\prod_{i=1}^a|\det m_i|^{i-1-(a-1)/2}.
$$
Then the induced representation $I_P^{\GL_n}(\sigma^{\otimes a},\lambda_a)$ has a unique irreducible quotient $Q(\sigma^{\otimes a},\lambda_a)$ which is a discrete series representation.  Moreover all unitary discrete series representations of $\GL_n(F)$ are isomorphic to $Q(\sigma^{\otimes a},\lambda_a)$ for some $\sigma$ and $a|n.$  One can consult \cite[\S 8.4]{Getz:Hahn} for references for these facts.

Let $\phi:W_F \to \GL_{n/a}(\CC)$ be the $L$-parameter of $\sigma.$  Then the $L$-parameter of $Q(\sigma^{\otimes a},\lambda_a)$ is $\phi \otimes \mathrm{Sym}^{a-1}$ \cite[\S 10]{Zelevinsky:IndII}.  Therefore
\begin{align} \label{decomp}
\rho(\phi \otimes \mathrm{Sym}^{a-1})=\bigoplus_{i=1}^k \phi_i \otimes \mathrm{Sym}^{a_i-1}
\end{align}
for some representations $\phi_i$ of $W_F$ and $a_i \in \ZZ_{\geq 1}.$

Observe that
$$
\phi^{\oplus a}_{\lambda_a}=\phi \circ \Delta:W_F \lto {}^LM
$$
is the $L$-parameter of $\sigma^{\otimes a}_{\lambda_a}.$
Moreover \eqref{decomp} implies that 
\begin{align}
    \rho|_{{}^LM}(\phi^{\oplus a}_{\lambda_a})=\bigoplus_{i=1}^k \bigoplus_{j=0}^{a_i-1}\phi_i \otimes |\cdot|^{(a_i-1)/2-j}.
\end{align}
Thus the proposition follows from Lemma \ref{lem:L:param:discrete}.  
\end{proof}

\begin{prop} \label{prop:gamma:compat:temp} Assume \eqref{LLC:gamma:compat}. 
Let $M\in \mathcal{M}$ and $P\in\mathcal{P}(M).$  Let $\sigma$ be a unitary supercuspidal representation of $M(F)$ and let $\lambda_0\in \mathfrak{a}^*_{M}$. If $\pi\in \mathrm{Temp}(G)$ is a subquotient of $I_P^G(\sigma_{\lambda_0})$, then
\begin{align*}
    \gamma(\tfrac{1}{2},\pi,\rho,\psi) = \gamma(\tfrac{1}{2},\sigma_\lambda,\rho\vert_{^LM},\psi)|_{\lambda=\lambda_0}.
\end{align*}
In particular, the meromorphic function $\gamma(\tfrac{1}{2},\sigma_\lambda,\rho\vert_{^LM},\psi)$ is analytic at $\lambda=\lambda_0.$
\end{prop}
\begin{proof}Let $(M',\sigma')\in\widetilde{\Temp}_\Ind(G)$ be chosen so that $\pi$ is a subrepresentation of $I_{P'}^G(\sigma')$. Then $\gamma(s,\pi,\rho) = \gamma(s,\sigma',\rho\vert_{^LM'})$. By \cite[Corollaire VI.2.1]{Renard} we can find a semi-standard Levi subgroup $M''$ of $G$ contained in $M'$, $P''\in \mathcal{P}(M'')$, a unitary supercuspidal representation $\sigma''$ of $M''(F),$ and $\lambda\in \mathfrak{a}^*_{M''\CC}$ such that $\sigma'$ is a subrepresentation of $\Ind_{P''\cap M'}^{M'}(\sigma''_{\lambda''})$. Applying \eqref{LLC:gamma:compat} and taking duals to switch quotients and subrepresentations, we have $\gamma(s,\sigma',\rho\vert_{^LM'}) = \gamma(s,\sigma'',\rho\vert_{^LM''})$. By transitivity of induction, $\pi$ is a subrepresentation of $I_{P''}^G(\sigma'_{\lambda''})$.  This implies that the pairs $(M'',\sigma''_{\lambda''})$ and $(M,\sigma_{\lambda_0})$ are associate \cite[Th\'eor\`eme VI.5.4]{Renard}.  We deduce that $\gamma(\tfrac{1}{2},\sigma
_{\lambda''},\rho|_{{}^LM''},\psi)=\gamma(\tfrac{1}{2},\sigma_{\lambda_0},\rho|_{{}^LM},\psi).$ 
\end{proof}

For the rest of this section we record several consequences of the local Langlands correspondence for later use.

\subsection{Local factors as functions on $\Temp_\Ind(G)$}\label{LLC:gamma:invariant}
Let $(M,\sigma)\in\widetilde{\Temp}_\Ind(G)$ and let $\pi$ be an irreducible subquotient of $I_P^G(\sigma)$. By \eqref{LLC:para}
\begin{align*}
    L(s,\pi,\rho) = L(s,\sigma,\rho\vert_{^LM}) \quad\textrm{and}\quad\varepsilon(s,\pi,\rho,\psi)=\varepsilon(s,\sigma,\rho\vert_{^LM},\psi).
\end{align*}
Suppose that $(M,\sigma'),(M,\sigma) \in \widetilde{\Temp}_{\Ind}(G)$ map to the same point in $\Temp_\Ind(G)$. Then $I_P^G(\sigma)$ and $I_P^G(\sigma')$ have the same semisimplification. In particular, if $\pi$ is a common constituent, then
\begin{align*}
    L(s,\sigma',\rho\vert_{^LM}) = L(s,\pi,\rho) = L(s,\sigma,\rho\vert_{^LM}) 
\end{align*}
and similarly for the $\varepsilon$-factors.
This implies that 
\begin{align}\label{loc:factors}\begin{split}
    L_{\rho}(M,\sigma):&= L(\tfrac{1}{2},\sigma,\rho\vert_{^LM}),\\
    \varepsilon_{\rho,\psi}(M,\sigma):&= \varepsilon(\tfrac{1}{2},\sigma,\rho\vert_{^LM},\psi),\\
    \gamma_{\rho,\psi}(M,\sigma):&= \gamma(\tfrac{1}{2},\sigma,\rho\vert_{^LM},\psi)\end{split}
\end{align}
are well-defined functions on $\Temp_\Ind(G).$  We usually drop $\psi$ from notation for simplicity.  Writing $\pi=I_P^G(\sigma)$ with $P\in\mathcal{P}(M)$, we set
\begin{align*}
    L_\rho(\pi) := L_\rho(M,\sigma) = L(\tfrac{1}{2},\sigma,\rho\vert_{^LM}),
\end{align*}
and define the remaining local factors similarly.  By pullback along the map \eqref{Temp->TempInd}, we obtain functions
 $$
 L_{\rho},\varepsilon_{\rho},\gamma_{\rho}:\Temp(G) \lto \CC.
 $$
For tempered $\pi$ we have $L_{\rho}(\pi)=L(\tfrac{1}{2},\pi,\rho),$  $\varepsilon_{\rho}(\pi)=\varepsilon(\tfrac{1}{2},\pi,\rho,\psi),$ and $\gamma_\rho(\pi)=\gamma(\tfrac{1}{2},\pi,\rho,\psi).$

\begin{lem}\label{LLC:gamma:rationality:1} The composition
\begin{align*}
    \widetilde{\mathrm{Temp}}_{\mathrm{Ind}}(G)\lto \mathrm{Temp}_{\mathrm{Ind}}(G)\stackrel{\gamma_{\rho}}{\lto}\mathbb{C}
\end{align*}
is a smooth function and extends to a rational (resp. meromorphic) function on $\widetilde{\mathrm{Temp}}_{\mathrm{Ind}}(G)_\mathbb{C}$ when $F$ is non-Archimedean (resp. Archimedean). The same holds if $\gamma_{\rho}$ is replaced by $L_{\rho}$ or $\varepsilon_{\rho}$. In the case of $\varepsilon_{\rho}$, the extension is a polynomial (resp. entire) function when $F$ is non-Archimedean (resp. Archimedean). Moreover, the extensions of $L_{\rho}$ and $\gamma_{\rho}$ are holomorphic in an (analytic) open neighborhood of $\widetilde{\mathrm{Temp}}_{\mathrm{Ind}}(G)$. 

The inverse $L_\rho^{-1}:(M,\sigma)\mapsto L(\tfrac{1}{2},\sigma,\rho|_{{}^LM})^{-1}$ is well-defined and extends to a polynomial (resp. entire) function $\widetilde{\mathrm{Temp}}_{\mathrm{Ind}}(G)_\mathbb{C}\to\mathbb{C}$ when $F$ is non-Archimedean (resp. Archimedean).
\end{lem}
\begin{proof}
Let $M$ be a semi-standard Levi subgroup of $G.$ 
Since $\widehat{M}(\mathbb{C})$ is reductive and $\rho$ is tempered, the restriction $\rho\vert_{{^L}\!M}:{}^LM\to\GL_n(\mathbb{C})$ is semisimple.  On the other hand, the functions $L_{\rho},L_{\rho}^{-1},\;\gamma_{\rho},$ and $\varepsilon_{\rho}$ are additive.  Indeed, this is obvious for $L_{\rho}$ and $L_{\rho}^{-1},$ it is stated in \cite[(3.4.2)]{Tate_NT} for $\varepsilon_{\rho}$, and then follows for $\gamma_{\rho}.$  

By the additivity observation above to prove the assertions in the lemma regarding $\gamma_{\rho}$ it suffices to study the function
\begin{align*}
\mathfrak{a}_{M\CC}^*/i\mathfrak{a}_{M,\sigma}^\vee&\lto \CC\\
    \lambda &\longmapsto \gamma(\tfrac{1}{2},\sigma_{\lambda},\rho,\psi)
\end{align*}
for irreducible $\rho.$ Since $i\mathfrak{a}_{M,\sigma}^\vee \subseteq i\mathfrak{a}_M^*$ is a full rank lattice contained in $i\mathfrak{a}_{M,\sigma}^\vee$ if $F$ is non-Archimedean, it suffices to show this function pulls back to a rational function (resp. meromorphic) function on $ \mathfrak{a}_{M\CC}^*/i\mathfrak{a}_{M}^\vee \cong \Lambda_M$. By irreducibility of $\rho$, the image of $\rho\circ |\cdot|^{\lambda}$ lies in the center $Z_{\GL_n}(\mathbb{C})=\mathbb{C}^\times$, so $\rho \circ |\cdot|^{\lambda}=|\cdot|^{\mu(\lambda)}$
for some homomorphism $\mu:X^*(M) \to \ZZ$. By \eqref{LLC:GL} and \eqref{LLC:UT},
we deduce that 
\begin{align*}
    \gamma(\tfrac{1}{2},\sigma_{\lambda},\rho,\psi)=\gamma(\tfrac{1}{2}+\mu(\lambda),\mathrm{LL}_{\GL_n}^{-1}(\rho \circ\mathrm{LL}_M(\sigma)),\psi)
\end{align*}
and similarly for $L_{\rho},$ $L_{\rho}^{-1}$ and $\varepsilon_{\rho}.$  
The claims in the lemma are now standard; see \cite{Jacquet:GLn:Corvallis} and the reference therein, for example.
\end{proof}

\begin{lem} \label{lem:chis}
Let $\pi \in \Temp(G)$ and $\sigma=\mathrm{LL}^{-1}_{\GL_n}(\rho \circ \mathrm{LL}_{G}(\pi)).$  Let $\chi_{\sigma}$ be the central character of $\sigma.$  
One has that
\begin{align*}
\varepsilon(s,\pi,\rho,\psi)\varepsilon(1-s,\pi^\vee,\rho,\psi) &= \chi_\sigma(-1),\\
\gamma(s,\pi,\rho,\psi)\gamma(1-s,\pi^\vee,\rho,\psi)&=\chi_{\sigma}(-1).
\end{align*}
\end{lem}
\begin{proof}
When $G=\GL_n$ and $\rho$ is the standard representation the first identity follows from the local Langlands correspondence and \cite[1.3.10]{Jacquet:GLn:Corvallis}, and the second identity follows from the first.  One reduces the general case to this one using
\eqref{LLC:Cg}.
\end{proof}

\begin{lem}\label{LLC:gammafactor:unitarity:2} For $\pi\in\mathrm{Temp}(G)$, one has that
\begin{align*}
    \gamma(\tfrac{1}{2},\pi^\vee,\rho,\overline{\psi}) &= \overline{\gamma(\tfrac{1}{2},\pi,\rho,\psi)},\\
    |\gamma(\tfrac{1}{2},\pi,\rho,\psi)|& = 1.
\end{align*}

\end{lem}
\begin{proof} 
As in the previous lemma, using \eqref{LLC:Cg}, the argument reduces to the special case where $G=\GL_n$ and $\rho$ is the standard representation.  We henceforth assume we are in this special case.

For any $f\in\mathcal{S}(M_n(F))$, recall the local zeta integral of Godement-Jacquet:
\begin{align*}
    Z(s,f,c) = \int_{\GL_n(F)} f(x)c(x)|\det x|^{s+\frac{n-1}{2}} dx,
\end{align*}
where $c$ is a matrix coefficient  of $\pi$. Let $c^\vee(x):=c(x^{-1})$, and $\mathcal{F}_\psi$ be the usual Fourier transform against the additive character $\psi$:
\begin{align*}
    \mathcal{F}_\psi (f)(y) = \int_{M_n(F)} f(x)\psi(\mathrm{tr}(xy))dx.
\end{align*}
Here we use the Haar measure that is self-dual with respect to $\psi.$
Then the functional equation
\begin{align*}
    Z(1-s,\mathcal{F}_\psi (f),c^\vee) = \gamma(s,\pi,\psi)Z(s,f,c)
\end{align*}
holds \cite[Theorems 3.3 and 8.7]{GodementJacquetBook}.  Strictly speaking, in the Archimedean case, Godement and Jacquet work with a subspace of $\mathcal{S}(M_n(F)),$ but the argument that follows is insensitive to the difference.
For $s \in \RR$
\begin{align*}
Z(1-s,\mathcal{F}_{\overline{\psi}}(\overline{f}),\overline{c^\vee})=    \overline{Z(1-s,\mathcal{F}_\psi (f),c^\vee)} = \overline{\gamma(s,\pi,\psi)}\overline{Z(s,f,c)}=\overline{\gamma(s,\pi,\psi)}Z(s,\overline{f},\overline{c}).
\end{align*}
Since $\overline{c}$ is a matrix coefficient for $\pi^\vee$, upon comparing the last two displayed equations, we obtain the first identity.

 For the second identity, using the first equality, \cite[(3.4.4)]{Tate_NT} and Lemma \ref{lem:chis} we have
\begin{align*}
    |\gamma(\tfrac{1}{2},\pi,\psi)|^2 = \gamma(\tfrac{1}{2},\pi,\psi)\gamma(\tfrac{1}{2},\pi^\vee,\overline{\psi}) = \chi_\pi(-1)\gamma(\tfrac{1}{2},\pi,\psi)\gamma(\tfrac{1}{2},\pi^\vee,\psi) = \chi_\pi(-1)^2 = 1.
\end{align*}
\end{proof}

\subsection{Poles of $L$-functions}\label{LLC:Lfactor:pole} In this subsection, we assume $F$ is an Archimedean local field.  By applying Weil restriction of scalars, we may assume $F=\RR$ without loss of generality.

In the Archimedean setting, $L$-functions have infinitely many poles. We will require some control on these poles that is uniform over all irreducible square integrable representations of $G(F).$  This control is formalized in Theorem \ref{LLC:Lfactor:pole:thm} below.

Recall the following theorem from \cite{Langlands:ArchLLC}, \cite[\S 11]{Borel:Corvallis}:
\begin{thm} \label{thm:square}
One has that $\mathrm{LL}^{-1}_{G}(\Phi_2(G))=\Pi_2(G).$ \qed
\end{thm}

We write $W_{\RR}=\CC^\times \cup j\CC^\times$, and identify the subgroup $\CC^\times$ with $W_{\CC}$ as usual \cite[\S 1.4.3]{Tate_NT}.  Thus $j^2=-1$ and $jzj^{-1}=\overline{z}$ for $z\in W_\mathbb{C}$. By \cite[Theorem 1.3]{AdamsVoganContragredient}, we have the following:
\begin{thm} \label{thm:contra}
Let $C_{W_\mathbb{R}}\in\mathrm{Aut}\,W_{\mathbb{R}}$ be the involution defined by $C_{W_\mathbb{R}}(z) = z^{-1}$ for $z\in W_\mathbb{C}$ and $C_{W_\mathbb{R}}(j)=j$. Then $\mathrm{LL}_G(\pi)\circ C_{W_\mathbb{R}} = \mathrm{LL}_G(\pi^\vee)$.  \qed
\end{thm}

Recall the cylinder $V_{[-\beta,\beta]}$ defined in \eqref{cylinder}.   The main theorem of this subsection follows:
\begin{thm}\label{LLC:Lfactor:pole:thm} There exists
\begin{itemize}
\item a set $\mathcal{D}$ of representatives of $\Pi_2(G)/i\mathfrak{a}_G^*$, stable under contragredients, and
\item finite subsets $\mathcal{P}(G,\beta) \subset \CC[s]$ for $\beta>0$ 
\end{itemize}
satisfying the following property:
for $\sigma\in\mathcal{D}$ there exists a $p\in \mathcal{P}(G,\beta)$ such that $p(s)L(s,\sigma,\rho)$ has neither pole nor zero in $V_{[-\beta,\beta]}$. 
\end{thm}

\noindent The key point in this theorem is that the finite set $\mathcal{P}(G,\beta)$ does not depend on $\sigma.$  

The proof of Theorem \ref{LLC:Lfactor:pole:thm} will occupy the majority of this subsection.  
In view of \eqref{LLC:UT} and Theorem \ref{thm:contra}, it suffices to show that there exists
\begin{itemize}
\item a set $\mathcal{D}$ of representatives of $\Phi_2(G)/i\mathfrak{a}_{G}^*$ invariant under precomposing with $C_{W_{\RR}}$, and
\item finite subsets $\mathcal{P}(G,\beta) \subset \CC[s]$ for $\beta>0$
\end{itemize}
satisfying the following property: for $\phi \in \mathcal{D}$ there exists a $p \in \mathcal{P}(G,\beta)$ such that  $p(s)L(s,\rho \circ \phi)$ has no pole and no zero in $V_{[-\beta,\beta]}.$ 

The following lemma can be extracted from \cite[\S 3]{Shelstad:realgroup}:

\begin{lem} \label{lem:finite:torus}
If $\Phi_2(G)$ is nontrivial, then there is a torus $T \leq G$ with $(T \cap G^{\mathrm{der}})^{\circ}$  anisotropic and an $L$-map $\iota:{}^LT \to {}^LG$ 
such that 
$
\Phi_2(G) \subseteq \iota_{*}\Phi_t(T).
$ \qed
\end{lem}
Here 
$\iota_{*}:\Phi_t(T) \to \Phi_t(G)$ is the pushforward.

\begin{lem} \label{lem:red:torus}
    To prove Theorem \ref{LLC:Lfactor:pole:thm}, it suffices to consider the case where $G$ is a torus $T.$
\end{lem}
\begin{proof}
We may assume $\Phi_2(G)$ is nontrivial.  Let $\iota:{}^LT \to {}^LG$ be the $L$-map of Lemma \ref{lem:finite:torus}.
Restriction to $X^*(T)$ induces a morphism
$i\mathfrak{a}_G^* \to i\mathfrak{a}_T^*.
$
Thus we obtain an action of $i\mathfrak{a}_G^*$ on $\Phi_t(T)$ via \eqref{ia:act}, and $\iota_*$ is $i\mathfrak{a}_G^*$-equivariant. 
Since the torus $T$ is anisotropic modulo $Z_G,$ the torus $A_G$ is the maximal split subtorus of $T$. Thus the morphism $i\mathfrak{a}_G^* \to i\mathfrak{a}_T^*$ is an isomorphism.

For $\phi\in\Phi_2(G)$ we have $\phi = \iota_*\phi'$ for some $\phi'\in\Phi_t(T)$ by Lemma \ref{lem:finite:torus}.  Thus
$L(s,\phi,\rho)=L(s,\phi',\rho\circ \iota).$
The representation $\rho \circ \iota$ is again tempered.  This completes the reduction.
\end{proof}

\begin{lem}\label{LT:Weilform:repn} All irreducible representations of $^LT$ are either
\begin{enumerate}
    \item quasi-characters $\chi:{}^LT \to \CC^\times$ such that $\chi|_{\widehat{T}(\CC)}$ is algebraic and  $\chi(t,z) = \chi(jtj^{-1},\overline{z})$ for $(t,z)\in\widehat{T}(\mathbb{C})\times W_\mathbb{C}$, or
    \item $2$-dimensional induced representations $\mathrm{Ind}_{\widehat{T}(\mathbb{C})\times W_\mathbb{C}}^{{}^LT}(\chi_0)$ where $\chi_0:\widehat{T}(\mathbb{C})\times W_\mathbb{C} \to \CC^\times$ is a quasi-character with $\chi_0|_{\widehat{T}(\CC)}$ algebraic such that $\chi_0(t,z)\neq \chi_0(jtj^{-1},\overline{z})$ for some $(t,z)\in\widehat{T}(\mathbb{C})\times W_\mathbb{C}$.
\end{enumerate}
\end{lem}
\begin{proof} \quash{Let $\rho:{}^LT\to\GL_{V_\rho}(\mathbb{C})$ be a tempered representation. 
Since $\rho$ is tempered $\rho|_{W_{\CC}}$ is semisimple.  
Decompose $\rho$ into $W_\mathbb{C}$-eigenspaces. Each eigenspace is $\widehat{T}(\mathbb{C})$-invariant. Hence it can be decomposed into irreducible $\widehat{T}(\mathbb{C})$-invariant subspaces. Thus $\rho|_{\widehat{T}(\mathbb{C})\times W_\mathbb{C}}$ is semisimple.  Since $[{}^LT:\widehat{T}(\mathbb{C})\times W_\mathbb{C}]$ is finite and we are considering complex representations, this implies $\rho$ itself is semisimple by a standard argument (see the proof of \cite[Lemma 2.7]{Bushnell:Henniart}).  }
Let $\rho:{}^LT \to \GL_{V_{\rho}}(\CC)$ be irreducible. The restriction $\rho|_{\widehat{T}(\CC) \times W_{\CC}}$ is a direct sum of $1$-dimensional representations.  Let $\chi_0:\widehat{T}(\CC) \times W_{\CC} \to \CC^\times$ be one of these representations; in particular $\chi_0|_{\widehat{T}(\CC)}$ is algebraic. By Frobenius reciprocity, we then have a nonzero intertwining map
$$
\rho \lto \mathrm{Ind}_{\widehat{T}(\CC) \times W_{\CC}}^{{}^LT}(\chi_0).
$$
The representation on the right is two dimensional.  It is reducible if and only if $\chi_0$ is fixed under $W_{\RR}/W_{\CC},$ and hence extends to a homomorphism $\chi$ as in (1).    
\end{proof}    

To prove Theorem \ref{LLC:Lfactor:pole:thm}, by Lemmas \ref{lem:red:torus}  we may assume that $G=T$ is a torus and $\rho$ is irreducible. Moreover, by \cite[Theorem 2]{Casselman:realtori}, we may take
\begin{align}\label{LLC:realtorus:nice}
    T = \mathbb{G}_{m}^a \times (\mathrm{Res}_{\mathbb{C}/\mathbb{R}}\mathbb{G}_{m,\mathbb{C}})^b \times \mathrm{U}_{1}^c
\end{align}
for some $a,b,c\in\mathbb{Z}_{\geq 0}$, where  $\mathrm{U}_1(R):=\{x \in R \otimes_{\RR}\CC : x\overline{x}=1\}$ for $\RR-$algebras $R$. Thus
\begin{align}
    {}^LT=((\CC^\times)^a \times (\CC^\times \times \CC^\times)^b \times (\CC^\times)^c) \rtimes W_{\RR}.
\end{align}
We write elements of this group as $(\underline{t},\underline{(z_1,z_2)},\underline{u}) \rtimes w$, where $w \in W_{\RR}.$ 
One has that
$$
((1,1,1) \rtimes j)((\underline{t},\underline{(z_1,z_2)},\underline{u}) \rtimes w)((1,1,1) \rtimes j^{-1})=(\underline{t},\underline{(z_2,z_1)},\underline{u}^{-1}) \rtimes jwj^{-1}.
$$

\begin{lem} \label{lem:1d}
The $1$-dimensional representations of ${}^LT$ are given by the homomorphisms $\rho_{\ell,\ell',s_0,e}$ indexed by $(\ell,\ell',s_0,e) \in \ZZ^a \times \ZZ^b \times \CC \times \{0,1\}$ defined by stipulating that
\begin{align*}
\rho_{\ell,\ell',s_0,e}((\underline{t},\underline{(z_1,z_2)},\underline{u}) \rtimes z) &=|z|^{s_0}\prod_{i=1}^a t_i^{\ell_i}\prod_{i=1}^b z_{1i}^{\ell_i'}z_{2i}^{\ell_i'},\\
\rho_{\ell,\ell',s_0,e}((1,1,1) \rtimes j)&=(-1)^e.
\end{align*}
The representation $\rho_{\ell,\ell',s_0,e}$ is tempered if and only if $s_0 \in i\RR.$
\end{lem}
\begin{proof} 
A general algebraic homorphism $\widehat{T}(\CC) \to \CC^\times$ is of the form 
$$
(\underline{t},\underline{(z_1,z_2)},\underline{u}) \longmapsto \prod_{i=1}^at_i^{\ell_i}\prod_{i=1}^bz_{1i}^{\ell_i'}z_{2i}^{\ell_i''}\prod_{i=1}^cu_i^{\ell_i'''}
$$
for $(\ell,\ell',\ell'',\ell''') \in \ZZ^a \times \ZZ^b \times \ZZ^b \times \ZZ^c.$
By Lemma \ref{LT:Weilform:repn},  we must have $\ell_i'=\ell_i''$ for all $i$ and $\ell'''=0.$
\end{proof}

Again, applying Lemma \ref{LT:Weilform:repn} and a similar argument, we obtain the following:

\begin{lem} \label{lem:2d} The $2$-dimensional irreducible representations of $^LT$ are given by the induction of the homomorphisms $\rho_{\ell,\ell',\ell'',\ell''',s_0,m}:\widehat{T}(\CC) \times W_\CC \to \CC^\times$ indexed by $(\ell,\ell',\ell'',\ell''', s_0,m) \in \ZZ^a \times \ZZ^b \times \ZZ^b \times \ZZ^c \times \CC \times \ZZ$ given by 
$$
\rho_{\ell,\ell',\ell'',\ell''',s_0,m}((\underline{t},\underline{(z_1,z_2)},\underline{u}) \rtimes z)=|z|_\mathbb{C}^{s_0}(z/|z|^{1/2}_\mathbb{C})^{m}\prod_{i=1}^a t_i^{\ell_i}\prod_{i=1}^b z_{1i}^{\ell_i'}z_{2i}^{\ell_i''}\prod_{i=1}^c u_i^{\ell_i'''},
$$
where $\ell_i' \neq \ell_i''$ for some $i,$ $\ell_j''' \neq 0$ for some $j,$ or $m \neq 0$. The representation is tempered if and only if $s_0\in i\RR$. \qed
\end{lem}

 Consider the following $L$-parameters:
\[
\begin{array}{l|l}
\hline
T &  W_\RR \lto {}^LT \\
\hline
\GG_m & \chi_\epsilon(z) = 1 \rtimes z, \quad \chi_\epsilon(j) = (-1)^\epsilon \rtimes j \text{ for}\;\epsilon \in \{0,1\}. \\
\mathrm{Res}_{\CC/\RR}\GG_{m,\CC} & 
\chi_k(z) = \Bigl( \bigl(\tfrac{z}{(z\overline{z})^{1/2}}\bigr)^k, \bigl(\tfrac{z}{(z\overline{z})^{1/2}}\bigr)^{-k} \Bigr)\rtimes z, \quad  \chi_k(j) = \Bigl((-1)^k,1\Bigr)\rtimes j\text{ for }\; k \in \ZZ .\\
\mathrm{U}_1 & \chi_k(z) = \bigl(\tfrac{z}{(z\overline{z})^{1/2}}\bigr)^k \rtimes z, \quad
\chi_k(j) = 1\rtimes j \text{ for }\; k \in \ZZ. \\
\hline
\end{array}
\]
\quash{
}
The following lemma can be extracted from \cite[\S 6]{Casselman:realtori}: 
\begin{lem} \label{lem:reps} Let $T \in \{\GG_m,\;\mathrm{Res}_{\CC/\RR}\GG_{m,\CC},\;\mathrm{U}_1\}.$ The $L$-parameters in the table above provide a complete set of representatives for $\Phi_t(T)/i\mathfrak{a}_{T}^*.$  
\qed
\end{lem}

\begin{proof}[Proof of Theorem \ref{LLC:Lfactor:pole:thm}] 
We translated the theorem to a statement about $L$-parameters above the statement of Lemma \ref{lem:finite:torus}.
By Lemma \ref{lem:red:torus}, it suffices to consider the case where $G$ is a torus.  We may also assume $\rho$ is irreducible. 

Let $\mathrm{p}:({}^L\GG_m)^a \times ({}^L\mathrm{Res}_{\CC/\RR}\GG_{m\CC})^b \times ({}^L\mathrm{U}_1)^c \to \widehat{T}(\CC)$ be the set-theoretic projection; it is not a homomorphism.
Let
\begin{align}\label{LLCa:pole:D}
    \mathcal{D} := \left\{\phi_{\underline{\epsilon},\underline{k},\underline{k'}}: (\underline{\epsilon},\underline{k},\underline{k}') \in \{0,1\}^a \times \mathbb{Z}^b \times \mathbb{Z}^c\right\},
\end{align}
where 
\begin{align*}
\phi_{\underline{\epsilon},\underline{k},\underline{k'}}(w)= \mathrm{p}\left( (\chi_{\epsilon_1}(w),\ldots,\chi_{\epsilon_a}(w),\chi_{k_1}(w),\ldots,\chi_{k_b}(w),\chi_{k'_1}(w),\ldots,\chi_{k'_c}(w))\right) \rtimes w
\end{align*}
for $w \in W_\RR.$
This is a set of representatives for $\Phi_t(T)/i\mathfrak{a}_T^*$ by Lemma \ref{lem:reps}.
The set $\mathcal{D}$ is stable under precomposing with $C_{W_\mathbb{R}}$ defined as in Theorem \ref{thm:contra}.   

In view of Lemma \ref{LT:Weilform:repn}, the only possibilities for irreducible representations of ${}^LT$ are listed in Lemma \ref{lem:1d} and Lemma \ref{lem:2d}.  
 Assume $\rho = \rho_{\ell,\ell',s_0,e}$ is one dimensional. Then for all $\phi \in \mathcal{D},$
\begin{align} \label{expl:1d}
    \rho\circ \phi(z)=|z|_{\CC}^{s_0}, \quad \rho\circ \phi(j)=(-1)^{e_{\rho,\phi}}
\end{align}
for some $e_{\rho,\phi}\in\{0,1\}$, so that $ L(s,\phi,\rho)= \pi^{-\frac{s+s_0+e_{\rho,\phi}}{2}} \Gamma\left(\frac{s+s_0+e_{\rho,\phi}}{2}\right)$. Take
\begin{align*}
    p_0(s) := \prod_{n\in [-\beta,0]\cap 2\mathbb{Z}} (s+s_0-n) \quad \textrm{and} \quad
    p_1(s) := \prod_{n\in [-\beta,0]\cap (1+2\mathbb{Z})} (s+s_0-n).
\end{align*}
Then $p_{e_{\rho,\phi}}(s)L(s,\phi,\rho)$
has no pole and is nonvanishing in $V_{[-\beta,\beta]}$ for all $\phi\in\mathcal{D}.$ In this case the finite set $\mathcal{P}(G,\beta) = \{p_0,p_1\}$ satisfies the assertion of Theorem \ref{LLC:Lfactor:pole:thm}.

Now, let $\rho_0=\rho_{\ell,\ell',\ell'',\ell''',s_0,m}$ and consider $\left(\mathrm{Ind}_{\widehat{T}(\CC) \times W_{\CC}}^{{}^LT} \rho_0\right) \circ \phi$ for $\phi \in \mathcal{D}.$  One checks that
$$
\left(\mathrm{Ind}_{\widehat{T}(\CC) \times W_{\CC}}^{{}^LT} \rho_{0}\right) \circ \phi=\mathrm{Ind}_{W_{\CC}}^{W_{\RR}} \left(\rho_{0} \circ \phi|_{W_{\CC}}\right).
$$
Since $L$-factors are invariant under induction we deduce that
\begin{align} \label{expl:2d} \begin{split}
L(s,\phi_{\underline{\epsilon},\underline{k},\underline{k'}},\mathrm{Ind}_{\widehat{T}(\CC) \times W_{\CC}}^{{}^LT} \rho_{0})&=L(s,\rho_0 \circ \phi_{\underline{\epsilon},\underline{k},\underline{k'}}|_{W_\CC})
\\&=2(2\pi)^{-s-s_0-f_{\rho,\phi}}\Gamma\left(s+s_0+ f_{\rho,\phi}\right) \end{split}
\end{align}
where 
\begin{align}\label{rhophi:exp:C}
    f_{\rho,\phi}:=\frac{1}{2}\left|m+\sum\limits_{i=1}^bk_i(\ell'_i-\ell''_i)+\sum\limits_{i=1}^ck'_i\ell'''_i\right|.
\end{align}
For $k\in\tfrac{1}{2}\mathbb{Z}_{\geq 0}$, set
\begin{align*}
    p_k(s) := \prod_{n\in [-\beta,-k]\cap \mathbb{Z}}(s+s_0-n).
\end{align*}
For all but finitely many $k\in \tfrac{1}{2}\mathbb{Z}_{\geq 0}$, we have $p_k(s)\equiv 1$. By construction we see
\begin{align*}
    p_{f_{\rho,\psi}}(s)L(s,\phi_{\underline{\epsilon},\underline{k},\underline{k'}},\mathrm{Ind}_{\widehat{T}(\CC) \times W_{\CC}}^{{}^LT} \rho_{0})
\end{align*}
has no pole and is nonvanishing in $V_{[-\beta,\beta]}$ for all $\phi\in\mathcal{D}.$ In this case the finite set $\mathcal{P}(G,\beta) = \{p_k\mid k\in\tfrac{1}{2}\mathbb{Z}_{\geq 0}\}$ satisfies the assertion of Theorem \ref{LLC:Lfactor:pole:thm} for all $\phi \in \mathcal{D}$.
\end{proof}

\begin{cor}\label{LLC:Lfactor:pole:cor}  There exists 
\begin{itemize}
    \item a set $\mathcal{D}$ of representatives of $\Pi_2(G)/i\mathfrak{a}_{G}^*$  stable under contragredient, and
    \item finite sets of polynomials $\mathcal{P}(G,\mathcal{K})$ on $\mathfrak{a}_{G\CC}^*$ for each compact subset $\mathcal{K} \subset \mathfrak{a}_G^*$
\end{itemize}
    such that following holds: for all $\sigma \in \mathcal{D}$ there exists a $p \in \mathcal{P}(G,\mathcal{K})$ such that $p(\lambda) L(\tfrac{1}{2},\sigma_{\lambda},\rho)$
has neither pole nor zero on the cylinder $V_{\mathcal{K}}.$ 
\end{cor}

Here $V_{\mathcal{K}}$ is defined as in \eqref{cylinder}.

\begin{proof} We may assume $\rho$ is irreducible. Retain the notation in the proof of Lemma \ref{LLC:gamma:rationality:1} (view $M$ therein as $G$ here). Then
\begin{align*}
    L(\tfrac{1}{2},\sigma_{\lambda},\rho) = L(\tfrac{1}{2}+\mu(\lambda),\sigma,\rho).
\end{align*} 
Take $\beta>0$ such that $\{\frac{1}{2}+\mu(\lambda)\mid \lambda\in\mathcal{K}\}\subseteq V_{[-\beta,\beta]}.$  Let $\mathcal{D}$ and  $\mathcal{P}(G,\beta)\subseteq \mathbb{C}[s]$ be as in Theorem \ref{LLC:Lfactor:pole:thm}. Then $\mathcal{P}(G,\mathcal{K}):=\{\lambda\mapsto p(\frac{1}{2}+\mu(\lambda))\mid p\in \mathcal{P}(G,\beta)\}$ is a finite set of polynomials on $\mathfrak{a}_{G\CC}^*$ satisfying the statement of the corollary. 
\end{proof}

\subsection{Bounds for $\gamma$-factors}\label{LLCa:gamma:bound} We continue to take $F=\mathbb{R}.$  We assume $G$ admits a maximal torus $T$ such that $T/A_G$ is anisotropic, and that we have a tempered representation
$$
\rho:{}^LG \lto \GL_{V_\rho}(\CC).
$$
A well-known theorem of Harish-Chandra implies that this is equivalent to the statement that $\Pi_2(G) \neq \emptyset.$  For a parameter $\phi:W_\mathbb{R}\to {}^LG$, the restriction $\phi\vert_{W_\mathbb{C}}$  is $\widehat{G}(\mathbb{C})$-conjugate to $z\mapsto (z^\lambda\overline{z}^{\mu},z)$ for some $\lambda,\mu\in X_*(\widehat{T})_\mathbb{C}$ such that $\lambda-\mu\in X_*(\widehat{T})$. Let $|\cdot|$ be a Weyl group invariant norm on $X_*(\widehat{T})_\mathbb{C}$ and set
\begin{align*}
    N(\phi) := 1 + |\lambda|.
\end{align*}
Recall the definition of $N(\sigma)$ from \eqref{Nsig}.  
\begin{lem} \label{lem:N:same}
If $|\cdot|$ is normalized appropriately then $N(\sigma)=N(\mathrm{LL}_G(\sigma))$  for all $\sigma \in \Pi_2(G).$ 
\end{lem}
\begin{proof}
    This follows from the behavior of the local Langlands correspondence with respect to infinitesimal characters, see \cite[\S 2.3]{Buzzard:Gee}.    
\end{proof}
We point out that $N(\phi)$ and $N(\sigma)$ depend on a choice of norm.  These choices are irrelevant for our purposes since we will only use these quantities to study coarse growth rates of functions.

\begin{lem} \label{lem:norm:rho} One has that $N(\rho\circ \phi) \ll_{\rho} N(\phi)$ for $\phi \in \Phi_t(G).$ \qed
\end{lem}

In the remainder of this subsection, we prove some bounds for later use.  The following is \cite[\S III.5, Lemma 3 and Lemma 5]{Moreno:Advanced:analytic:number:theory}: 

\begin{lem}\label{Moreno:gamma:bound}  Let $d\in\{1,2\}$. For $Q \geq 0$ and $-\tfrac{1}{2}\leq\mathrm{Re}(s)\leq\tfrac{1}{2}$ one has that
\begin{align*}
\left|\frac{\Gamma\left(\frac{d}{2}\left(Q+1-s\right)\right)}{\Gamma\left(\frac{d}{2}\left(Q+s\right)\right)} \right| \leq \left(\tfrac{d}{2}\left|Q+1+s\right|\right)^{d/2-d\mathrm{Re}(s)}.
\end{align*} \qed
\end{lem}

\begin{cor}\label{cor:gamma:bound}  Let $d\in\{1,2\},\,Q \in \frac{1}{2}\ZZ_{\geq 0}$ and $n\in\mathbb{Z}_{>0}$. Let $p,p^\vee\in\mathbb{C}[s]$ be the monic polynomials such that $p(s)\Gamma\left(\frac{d}{2}\left(Q+ \tfrac{1}{2} - s\right)\right)$ and  $p^\vee(s)\Gamma\left(\frac{d}{2}\left(Q+\tfrac{1}{2}+s)\right)\right)$ are holomorphic  for $| \mathrm{Re}(s)|\leq n$ and nowhere vanishing on $\CC.$ Then one has that
\begin{align*}
\left|\frac{p(s)\Gamma\left(\frac{d}{2}\left(Q+\tfrac{1}{2} -  s\right)\right)}{p^\vee(s)\Gamma\left(\frac{d}{2}\left(Q+\tfrac{1}{2}+s\right)\right)} \right| \ll_{n} (1+|p(s)|)|Q+1+s+2n|^{nd+2}
\end{align*}
for $|\mathrm{Re}(s)|\leq n$.
\end{cor}
\begin{proof} We adapt the proof of \cite[Lemma 3.3]{Getz:Liu:BK}.  
If $Q<3n+2$ then there are only finitely many possibilities for the  parameters $d, Q$ and the polynomials $p,p^\vee$ in the expression we wish to bound. 
Thus if we assume further that $|\mathrm{Im}(s)|\leq 1$ the bound follows.  
 
 On the other hand if $Q \geq 3n+2$  then $p=p^\vee=1.$    Thus it suffices to show that 
\begin{align*}
    \left|\frac{\Gamma\left(\frac{d}{2}\left(Q+\tfrac{1}{2} -  s\right)\right)}{\Gamma\left(\frac{d}{2}\left(Q+\tfrac{1}{2}+s\right)\right)}\right| \ll_{n} |Q+1+s+2n|^{nd+2}
\end{align*}
provided either $|\mathrm{Im}(s)|\geq 1$ or $Q \geq 3n+2.$

Assume that $0\leq \mathrm{Re}(s)\leq n$. For any $N\in\mathbb{Z}_{> 0}$, since $\Gamma(z+1)=z\Gamma(z),$ one has that
\begin{align*}
    \Gamma\left(\tfrac{d}{2}(Q+\tfrac{1}{2}+s)\right) &= \Gamma\left(\tfrac{d}{2}(Q+\tfrac{1}{2}+s)-N\right)\displaystyle\prod_{k=1}^{N}\left(\tfrac{d}{2}(Q+\tfrac{1}{2}+s)-k\right)\\
    \Gamma\left(\tfrac{d}{2}\left(Q+\tfrac{1}{2}-s\right)\right) &= \dfrac{\Gamma\left(\tfrac{d}{2}\left(Q+\tfrac{1}{2}-s\right) + N\right)}{\displaystyle\prod_{k=0}^{N-1} \left(\tfrac{d}{2}\left(Q+\tfrac{1}{2}-s\right)+k\right)}.
\end{align*}
Therefore we have that
\begin{align*}
    &\left|\frac{\Gamma\left(\tfrac{d}{2}\left(Q+\tfrac{1}{2} -  s\right)\right)}{\Gamma\left(\tfrac{d}{2}\left(Q+\tfrac{1}{2}+s\right)\right)}\right| 
    &= \left|\prod_{k=0}^{N-1} \left(\tfrac{d}{2}\left(Q+\tfrac{1}{2}-s\right)+k\right)\prod_{k=1}^{N}\left(\tfrac{d}{2}(Q+\tfrac{1}{2}+s)-k\right)\right|^{-1}\left|\dfrac{\Gamma\left(\tfrac{d}{2}\left(Q+\tfrac{1}{2}-s\right) + N\right)}{\Gamma\left(\tfrac{d}{2}(Q+\tfrac{1}{2}+s)-N\right)}\right|.
\end{align*}
Assume $N \leq n+1.$  Since either $|\mathrm{Im}(s)| \geq 1$ or $Q\geq 3n+2$
the above is bounded by a constant depending on $n$ times
\begin{align} \label{before:d:split}\left|\dfrac{\Gamma\left(\tfrac{d}{2}\left(Q+\tfrac{1}{2}-s\right) + N\right)}{\Gamma\left(\tfrac{d}{2}(Q+\tfrac{1}{2}+s)-N\right)}\right|
    &=\left|\dfrac{\Gamma\left(\tfrac{d}{2}\left(Q+1-(\tfrac{1}{2}+s-\tfrac{2N}{d})\right)\right)}{\Gamma\left(\tfrac{d}{2}(Q+(\tfrac{1}{2}+s-\tfrac{2N}{d}))\right)}\right|.
\end{align} 
There is an integer $0 \leq N \leq n+1$ so that 
\begin{align}
\tfrac{1}{2}+\mathrm{Re}(s)-\tfrac{2N}{d} \in [-\tfrac{2}{d}+\tfrac{1}{2},\tfrac{1}{2}].
\end{align}
If $d=2$, by Lemma \ref{Moreno:gamma:bound} we have
\begin{align*}
    \left|\frac{\Gamma\left(\left(Q+\tfrac{1}{2}-s\right) +N\right)}{\Gamma\left((Q+\tfrac{1}{2}+s)-N\right)}\right|&\leq |Q+1+\tfrac{1}{2}+s-N|^{2(N-\mathrm{Re}(s))}\\
    &\ll_n |Q+1+\tfrac{1}{2}+s|^{2}.
\end{align*}
Thus we have a sufficient bound on \eqref{before:d:split} when $d=2.$

Now assume $d=1.$  
Let
$$
\beta(s)=\one_{[-2,-1)}(\mathrm{Re}(s)-2N).
$$
Then $\tfrac{1}{2}+s-2N = -\beta(s)+\left(\tfrac{1}{2}+s-2N + \beta(s)\right)$
and
\begin{align*}
    \mathrm{Re}\left(\tfrac{1}{2}+s-2N + \beta(s)\right) \in [-\tfrac{1}{2},\tfrac{1}{2}].
\end{align*} 
 Applying Lemma \ref{Moreno:gamma:bound} again we have
\begin{align*}
   \left|\frac{\Gamma\left(\tfrac{1}{2}\left(Q+\tfrac{1}{2}-s\right) +N\right)}{\Gamma\left(\tfrac{1}{2}(Q+\tfrac{1}{2}+s)-N\right)}\right|
    &=\left|\frac{\Gamma\left(\tfrac{1}{2}\left(Q+\tfrac{1}{2}-s\right) +N\right)}{\Gamma\left(\tfrac{1}{2}(Q+\tfrac{1}{2}+s)-N+\beta(s)\right)}\right| \times \begin{cases}
         \left|\tfrac{1}{2}(Q+\tfrac{1}{2}+s)-N)\right|&\text{ if } \beta(s) = 1  \\
         1&\text{ otherwise}  \end{cases}
   \\
    & \leq \left(\tfrac{1}{2}|Q+\beta(s)+1+\tfrac{1}{2}+s-2N + \beta(s)|\right)^{-\mathrm{Re}(s)+2N-\beta(s)} |\tfrac{1}{2}(Q+\tfrac{1}{2}+s)-N|\\
    &\ll_n |Q+1+s|^2.
\end{align*}
Hence for $0\leq \mathrm{Re}(s)\leq n$ we have that
\begin{align}\label{gamma:bound:1}
    \left|\frac{\Gamma\left(\tfrac{1}{2}\left(Q+\tfrac{1}{2} -  s\right)\right)}{\Gamma\left(\tfrac{1}{2}\left(Q+\tfrac{1}{2}+s\right)\right)}\right| \ll_{n} |Q+1+\tfrac{1}{2}+s|^2
\end{align}
if $Q\geq 3n+2$ or  $|\Im(s)|\geq 1.$

Assume $-n\leq \mathrm{Re}(s)\leq 0$. Set $s':=s+n$, so that $0\leq \mathrm{Re}(s')\leq n$. Then
\begin{align*}
    \left|\frac{\Gamma\left(\tfrac{d}{2}\left(Q+\tfrac{1}{2}-s\right)\right)}{\Gamma\left(\tfrac{d}{2}(Q+\tfrac{1}{2}+s)\right)} \right| &= \left|\frac{\Gamma\left(\tfrac{d}{2}\left(Q+\tfrac{1}{2}+n-s'\right)\right)}{\Gamma\left(\tfrac{d}{2}(Q+\tfrac{1}{2}-n+s')\right)} \right|=\left|\frac{\Gamma\left(\tfrac{d}{2}\left(Q+n+\tfrac{1}{2}-s'\right)\right)}{\Gamma\left(\tfrac{d}{2}(Q+n+\tfrac{1}{2}+s')-nd\right)} \right|\\
    &=\left(\prod_{k=0}^{nd-1}\left|\tfrac{d}{2}(Q+n+\tfrac{1}{2}+s')-k\right|\right) \left|\frac{\Gamma\left(\tfrac{d}{2}\left(Q+n+\tfrac{1}{2}-s'\right)\right)}{\Gamma\left(\tfrac{d}{2}(Q+n+\tfrac{1}{2}+s')\right)} \right|\\
    &\ll_n |Q+n+\tfrac{1}{2}+s'|^{nd} \left|\frac{\Gamma\left(\tfrac{d}{2}\left(Q+n+\tfrac{1}{2}-s'\right)\right)}{\Gamma\left(\tfrac{d}{2}(Q+n+\tfrac{1}{2}+s')\right)} \right|.
\end{align*}
Applying the bound \eqref{gamma:bound:1}, we see
\begin{align*}
    \left|\dfrac{\Gamma\left(\tfrac{d}{2}\left(Q+\tfrac{1}{2}-s\right)\right)}{\Gamma\left(\tfrac{d}{2}(Q+\tfrac{1}{2}+s)\right)} \right| &\ll_{n} |Q+n+\tfrac{1}{2}+s+n|^{nd}|Q+n+1+\tfrac{1}{2}+s+n|^2\\
    &\ll_n |Q+1 + s+2n|^{nd+2}.
\end{align*}
\end{proof}

\begin{lem}\label{LLCa:gamma:bound:torus} Let $T$ be a real torus. There exists a set $\mathcal{D}$ of representatives of $\Phi_t(T)/i\mathfrak{a}_T^*$  stable under precomposition with $C_{W_\mathbb{R}}$ that satisfies the following: for
\begin{enumerate}[label=(\alph*)]
    \item any $\phi \in \mathcal{D},$
    \item \label{com} any compact set $\mathcal{K}\subset\mathfrak{a}_T^*$,
    \item \label{tem} any tempered representation $\rho:{}^LT\to\GL_{V_\rho}(\CC)$,
    \item \label{ppvee} any polynomials $p(\lambda),\,p^\vee(\lambda)$ on $\mathfrak{a}_{T\mathbb{C}}^*$ such that $p(\lambda)L(\tfrac{1}{2},\phi_{\lambda},\rho)$ and $p^\vee(\lambda)L(\tfrac{1}{2},\phi_{\lambda}\circ C_{W_\mathbb{R}},\rho)$ have no poles in $V_{\mathcal{K}}$  and are nonvanishing on $\mathfrak{a}_{T\CC}^*$,
\end{enumerate}
there exist an integer $k\in\mathbb{Z}_{\geq 1}$ and a polynomial $p_0(\lambda)$ on $\mathfrak{a}_{T\CC}^*$ both independent of $\rho, p,p^\vee$ such that
\begin{align*}
   \sup_{\lambda\in V_\mathcal{K}} \left|\dfrac{p(\lambda)\gamma(\tfrac{1}{2},\phi_{\lambda}\circ C_{W_\mathbb{R}},\rho,\psi)}{p^\vee(\lambda)}\right| \ll_{\mathcal{K},p,p^\vee} (1+|p(\lambda)|)\left(N(\phi) + |p_0(\lambda)|\right)^k.
\end{align*}
\end{lem}

\begin{proof}
Arguing as in the last subsection, we may assume $T$ is in the form of \eqref{LLC:realtorus:nice} and take $\mathcal{D}$ as in \eqref{LLCa:pole:D}.  Since $\gamma$-factors are additive, we may assume that $\rho$ is irreducible.  The irreducible representations are $1$ or $2$-dimensional by Lemma \ref{LT:Weilform:repn} and classified in Lemma \ref{lem:1d} and Lemma \ref{lem:2d} above. By \ref{tem} one has $s_0\in i\mathbb{R}$ in the notation of loc.~cit.  

 If $\rho$ is $2$-dimensional  let $f_{\rho,\phi}$ be as in \eqref{rhophi:exp:C}. Then \eqref{expl:1d} and \eqref{expl:2d} imply that 
\begin{align}
    N(\rho \circ \phi)\ll \begin{cases}1+|s_0| &\textrm{ if }\rho \textrm{ is $1$-dimensional,}\\
    1+|s_0| +f_{\rho,\phi} &\textrm{ if }\rho \textrm{ is $2$-dimensional.} \end{cases}
\end{align}
We have that
\begin{align*}
L(\tfrac{1}{2},\phi_{\lambda}\circ C_{W_\mathbb{R}},\rho)&=L(\tfrac{1}{2},(\phi_{\lambda})^{\vee},\rho)=L(\tfrac{1}{2}-\mu(\lambda),\phi^\vee,\rho)
\end{align*}
for some linear functional $\mu:\mathfrak{a}_{T\CC}^* \to \CC.$  Similarly $L(\tfrac{1}{2},\phi_{\lambda},\rho)=L(\tfrac{1}{2}+\mu(\lambda),\phi, \rho).$  The analogous identities hold for the $\varepsilon$-factor.  Applying \eqref{expl:1d} and \eqref{expl:2d} again we see that
\begin{align*}
    \gamma(\tfrac{1}{2},\phi_\lambda\circ C_{W_\RR},\rho,\psi) = \varepsilon(\tfrac{1}{2},\phi_\lambda\circ C_{W_\RR},\rho,\psi) (\pi\dim \rho)^{-(s_0+\mu(\lambda))\dim\rho}\dfrac{\Gamma\left(\frac{\dim\rho}{2}\left(e_{\rho,\phi}+\tfrac{1}{2}+s_0+\mu(\lambda)\right) \right)}{\Gamma\left(\frac{\dim\rho}{2}\left(e_{\rho,\phi}+\tfrac{1}{2}-s_0-\mu(\lambda)\right) \right)}
\end{align*}
where $e_{\rho,\phi} \in \{0,1\}$ in the $1$-dimensional case and $e_{\rho,\phi}=f_{\rho,\phi}$ in the $2$-dimensional case. By Corollary \ref{cor:gamma:bound}, we see that
\begin{align*}
    \left|\dfrac{p(\lambda)\gamma(\tfrac{1}{2},\phi_{\lambda}\circ C_{W_\mathbb{R}},\rho,\psi)}{p^\vee(\lambda)}\right|&\ll_{\mathcal{K},p,p^\vee} (1+|p(\lambda)|)|e_{\rho,\phi}+1-s_0-\mu(\lambda)+\ell|^{\ell}\\
    &\ll (1+|p(\lambda)|)\left(N(\rho\circ\phi) + |\mu(\lambda)|+\ell\right)^\ell
\end{align*}
for some $\ell \in\mathbb{Z}_{\geq1}$. To complete the proof we apply Lemma \ref{lem:norm:rho}.
\end{proof}

\begin{prop} \label{prop:gamma:bounds} There exists a set $\mathcal{D}$ of representatives of $\Pi_2(G)/i\mathfrak{a}_{G}^*$  stable under taking contragredient that satisfies the following: for 
\begin{itemize}
\item any $\sigma \in \mathcal{D},$
    \item any compact set $\mathcal{K} \subseteq\mathfrak{a}_G^*,$
    \item any polynomials $p(\lambda),\,p^\vee(\lambda)$ on $\mathfrak{a}_{G\CC}^*$ such that  $
    p(\lambda)L(\tfrac{1}{2},\sigma_{\lambda},\rho)$ and  
$p^\vee(\lambda)L(\tfrac{1}{2},(\sigma_{\lambda})^\vee,\rho)$ have no poles in $V_{\mathcal{K}}$
 and $p^\vee(\lambda)L(\tfrac{1}{2},(\sigma_{\lambda})^\vee,\rho)$
is nonvanishing on $a_{G\CC}^*,$ 
\end{itemize}
there exist an integer $k\in\mathbb{Z}_{\geq 1}$ and a polynomial $p_0$ on $\mathfrak{a}_{G\CC}^*$ both independent of $\rho,p,p^\vee$  such that 
\begin{align*}
    \sup_{\lambda\in V_{\mathcal{K}}} \left|\dfrac{p(\lambda)\gamma(\tfrac{1}{2},(\sigma_\lambda)^\vee,\rho,\psi)}{p^\vee(\lambda)}\right| \ll_{\mathcal{K},p,p^\vee} (1+|p(\lambda)|)\left(N(\sigma)+|p_0(\lambda)|\right)^k.
\end{align*} 
\end{prop}

\begin{proof}
Arguing as in the proof of Lemma \ref{lem:red:torus}, and using \eqref{LLC:UT}, Theorem \ref{thm:contra} and Lemma \ref{lem:N:same}, the proposition follows from Lemma \ref{LLCa:gamma:bound:torus}.  
\end{proof}

\section{The Fourier transform}
\label{sec:FourieronHCSchwartz}

Let $\psi:F\to\mathbb{C}^\times$ be a nontrivial character of the local field $F,$ and let 
$\rho:{}^LG \to \GL_{V_\rho}(\CC)$ be a tempered representation.  By Lemma \ref{LLC:gamma:rationality:1} we have a well-defined smooth function $
\gamma_{\rho}:\mathrm{Temp}_{\mathrm{Ind}}(G)\to \mathbb{C}.$
We now investigate the multiplication map 
\begin{align} \label{multi} \begin{split}
\gamma_{\rho}\cdot :\mathcal{C}(\mathrm{Temp}_{\mathrm{Ind}}(G))&\lto \mathcal{C}(\mathrm{Temp}_{\mathrm{Ind}}(G))\\
    T &\longmapsto \left( \pi \mapsto \gamma_{\rho}(\pi)T(\pi)\right). \end{split}
\end{align}

For the reader's convenience, we record a version of the Cauchy integral formula.
\begin{lem}\label{FT:defn:lem:-1} Let $\Omega\subseteq \mathbb{C}^n$ be an open subset and let $f:\Omega\to\mathbb{C}$ be a holomorphic function. For $p\in\Omega$ and $r>0$ assume  $\{z\in\mathbb{C}^n:|z-p|\leq r\}\subseteq \Omega.$   Let $(m_1,\ldots,m_n)\in(\mathbb{Z}_{\geq 0})^{n}.$ One has 
\begin{align*}
    \frac{\partial^{m_1+\cdots+m_n}f}{\partial z_1^{m_1}\cdots\partial z_n^{m_n}}(p) = \frac{m_1!\cdots m_n!}{(2\pi i)^{n}}\int\dots \int\frac{f(z)}{(z_1-p_1)^{m_1+1}\cdots(z_n-p_n)^{m_n+1}} dz_1\cdots dz_n 
\end{align*}
where the $i$th integral is over $|z_i-p_i|=r,$ oriented counter-clockwise.
\end{lem}
\begin{proof} When $m_1=\cdots=m_n=0$, this is the usual Cauchy integral formula \cite[Theorem 1.17]{Voison:ComplexAG}. The formula for general $m_1,\ldots,m_n$ follows upon differentiating.
\end{proof}

\begin{lem}\label{FT:defn:lem:0} The map $\gamma_{\rho}\cdot$ is a well-defined continuous linear operator.
\end{lem}
\begin{proof} When $F$ is non-Archimedean the lemma is immediate.

Assume $F$ is Archimedean. To prove that $\gamma_{\rho}\cdot$ is well-defined and continuous, we must prove the continuity of the seminorm on $\mathcal{C}(\mathrm{Temp}_{\mathrm{Ind}}(G))$ given by 
\begin{align*}
   T\longmapsto \sup_{\sigma\in\Pi_2(M)} \norm{I_P^G(\sigma)(u) \circ D(\gamma_{\rho}\cdot T)(M,\sigma) \circ I_P^G(\sigma)(v)}_{\mathrm{op}} N(\sigma)^k
\end{align*}
for $M \in \mathcal{M}$, $D$ an invariant differential operator on $i\mathfrak{a}^*_M,$ $k\in\mathbb{Z}_{\geq 0}$ and $u,v\in U(\mathrm{Lie}(K)_\mathbb{C}).$ By the product rule,  it suffices to show there exists some $L\in\mathbb{Z}_{\geq 0}$ such that 
\begin{align} \label{gamma:esti}
    \left|D\gamma(\tfrac{1}{2},\sigma,\rho\vert_{^LM},\psi)\right| \ll_{D,\rho} N(\sigma)^L
\end{align}
for all $\sigma \in \Pi_2(M).$ We first observe that there is an open neighborhood of $0$ in $i\mathfrak{a}_M$ with compact closure  $\mathcal{K}$ such that \eqref{gamma:esti} holds for $D=\mathrm{id}$ and $\lambda \in V_{\mathcal{K}}$ by Proposition \ref{prop:gamma:bounds}.  
To prove it for general $D$ we use Lemma \ref{FT:defn:lem:-1} to reduce to the case of $D=\mathrm{id}$.
\end{proof}

We remind the reader that $\gamma_{\rho}=\gamma_{\rho,\psi}$ also depends on a choice of $\psi$ (see \eqref{loc:factors}).  
Define $\mathcal{F}_{\rho,\psi}:\mathcal{C}(G(F))\to \mathcal{C}(G(F))$ by 
\begin{align}\label{FT:defn:1}
    \mathcal{F}_{\rho,\psi} := (\cdot)^\vee \circ \mathrm{HP}^{-1}\circ \gamma_{\rho,\psi}\cdot\circ \mathrm{HP}
\end{align}
where $\mathrm{HP}$ is the map in Theorem \ref{HCPlan:thm1}.  By loc.~cit, for $f\in\mathcal{C}(G)$ 
\begin{align}\label{FT:defn:2}
    \mathcal{F}_{\rho,\psi} (f)(g) = \int_{\mathrm{Temp}_{\mathrm{Ind}}(G)} \gamma_{\rho,\psi}(\pi)\tr(\pi(g)\circ\pi(f)) d\pi.
\end{align}
We usually drop the $\psi$ from notation, writing $\mathcal{F}_{\rho}:=\mathcal{F}_{\rho,\psi}.$
By Lemma \ref{FT:defn:lem:0} and Theorem \ref{HCPlan:thm1} $\mathcal{F}_{\rho}$ is a continuous linear operator. By \eqref{HP:equi}  $\mathrm{HP}$ is $G(F)\times G(F)$-equivariant. It follows that 
\begin{align}\label{FT:equivariant:close}
    \mathcal{F}_{\rho} \circ \mathcal{R}(h,g)  = \mathcal{R}(g,h)\circ \mathcal{F}_{\rho}
\end{align}
where $\mathcal{R}$ is the action map \eqref{RuG}.

By Lemma \ref{HCPlan:transpose} one has that
\begin{align}\label{FT:defn:3}
    \mathcal{F}_{\rho} =  \mathrm{HP}^{-1}\circ (\cdot)^\intercal \circ \gamma_{\rho}\cdot\circ  \mathrm{HP}.
\end{align}
In view of this, define $\mathcal{F}^{\mathrm{op}}_{\rho}:\mathcal{C}(\mathrm{Temp}_{\mathrm{Ind}}(G))\to \mathcal{C}(\mathrm{Temp}_{\mathrm{Ind}}(G))$ by
\begin{align}\label{FT:defn:4}
    \mathcal{F}^{\mathrm{op}}_{\rho} := (\cdot)^\intercal \circ \gamma_{\rho} \cdot
\end{align}
This is again a continuous linear operator. The two operators just introduced are related by the commuting square:
\begin{equation}\label{FT:defn:5}
    \begin{tikzcd}
        \mathcal{C}(G(F))\arrow[r,"\mathcal{F}_{\rho}"]\arrow[d,"\mathrm{HP}"']&\mathcal{C}(G(F))\arrow[d,"\mathrm{HP}"]\\
        \mathcal{C}(\mathrm{Temp}_{\mathrm{Ind}}(G))\arrow[r,"\mathcal{F}_{\rho}^{\mathrm{op}}"']&\mathcal{C}(\mathrm{Temp}_{\mathrm{Ind}}(G))
    \end{tikzcd}    
\end{equation}
We shall refer to $\mathcal{F}_{\rho}=\mathcal{F}_{\rho,\psi}$ as the \textbf{Fourier transform} attached to $\rho$ and $\psi$ on the Harish-Chandra space $\mathcal{C}(G(F)).$ 

\begin{lem}\label{FT:lem:important:1} For $f\in\mathcal{C}(G(F))$ and $\pi\in\Temp(G)\cup\Temp_{\Ind}(G)$, one has that
\begin{align*}
    \pi(\mathcal{F}_{\rho} (f)^\vee) = \gamma_{\rho}(\pi)\pi(f)\quad \textrm{and}\quad
    \pi(\mathcal{F}_{\rho} (f)) = \gamma_{\rho}(\pi^\vee)\pi(f^\vee).
\end{align*}
\quash
{a subrepresentation of $I_P^G(\sigma)$ for some $(M,\sigma)\in\widetilde{\mathrm{Temp}}_{\Ind}(G)$, one has
\begin{align*}
    \pi((\mathcal{F}_{\rho} f)^\vee) &= \gamma(\tfrac{1}{2},\sigma,\rho)\pi(f)\\
    \pi(\mathcal{F}_{\rho} f) &= \gamma(\tfrac{1}{2},\sigma^\vee,\rho)\pi(f^\vee).
\end{align*}}
\begin{proof} First assume $\pi\in\Temp_\Ind(G)$, and write $\pi = I_P^G(\sigma)$ for some $(M,\sigma)\in\widetilde{\Temp}_\Ind(G)$ and $P\in\mathcal{P}(M)$. By \eqref{FT:defn:1}, we have
\begin{align*}
    \pi(\mathcal{F}_{\rho} (f)^\vee) = \mathrm{HP}(\mathcal{F}_{\rho} (f)^\vee)(\pi) = (\gamma_{\rho}\mathrm{HP}(f))(\pi) = \gamma_{\rho}(\pi)\pi(f).
\end{align*}
This proves the first equality. By  Lemma \ref{FT:lem:easy:1} and the first equality (with $\pi$ replaced by $\pi^\vee$)
\begin{align*}
    \pi(\mathcal{F}_{\rho} (f))^\intercal = \pi^\vee(\mathcal{F}_{\rho}( f)^\vee) = \gamma_{\rho}(\pi^\vee)\pi^\vee(f).
\end{align*}
The second equality follows from another application of Lemma \ref{FT:lem:easy:1}.

If $\pi\in\Temp(G)$, it is a subrepresentation of $I_P^G(\sigma)$ for some $(M,\sigma)\in\widetilde{\Temp}_\Ind(G)$ and $P\in\mathcal{P}(M)$ (see the discussion near \eqref{Temp->TempInd}), and $\pi(f)$ is the operator obtained by restricting $I_P^G(\sigma)(f)$ to the subspace $\pi$. The claim for $\Temp(G)$ follows.
\end{proof}
\end{lem}

\begin{lem} \label{lem:FT:unit} The Fourier transform $\mathcal{F}_{\rho}:\mathcal{C}(G(F))\to \mathcal{C}(G(F))$ is unitary.
\begin{proof}We first observe that Corollary \ref{HCPlan:enhanced}  implies that 
\begin{align*}
    \norm{\mathcal{F}_{\rho} (f)}^2_{L^2}&= \int_{\mathrm{Temp}_{\mathrm{Ind}}(G)}\mathrm{tr}\,(\pi(\mathcal{F}_{\rho} (f))^*\circ \pi(\mathcal{F}_{\rho} (f))) d\pi\\
    &=\int_{\mathrm{Temp}_{\mathrm{Ind}}(G)}|\gamma(\tfrac{1}{2},\pi^\vee,\rho,\psi)|^2\mathrm{tr}\,(\pi(f^\vee)^*\circ \pi(f^\vee)) d\pi 
 \quad \textrm{(by Lemma \ref{FT:lem:important:1})}\\
&=\int_{\mathrm{Temp}_{\mathrm{Ind}}(G)}\mathrm{tr}\,(\pi(f^\vee)^*\circ \pi(f^\vee)) d\pi \quad \textrm{(by Lemma \ref{LLC:gammafactor:unitarity:2}).}
\end{align*}
Changing variables $\pi \mapsto \pi^\vee$ and using Lemma \ref{pi->pivee:preservemeasure} this is 
\begin{align*}    \int_{\mathrm{Temp}_{\mathrm{Ind}}(G)}&\mathrm{tr}\,(\pi^\vee(f^\vee)^*\circ \pi^\vee(f^\vee)) d\pi\\
     &=\int_{\mathrm{Temp}_{\mathrm{Ind}}(G)}\mathrm{tr}\,((\pi(f)^\intercal)^*\circ \pi(f)^\intercal) d\pi\quad \text{(by Lemma \ref{FT:lem:easy:1})}\\
   &=\int_{\mathrm{Temp}_{\mathrm{Ind}}(G)}\mathrm{tr}\,(\pi(f)^*\circ \pi(f)) d\pi \quad \textrm{(by \eqref{FT:lem:easy:2})}\\
    &=\norm{f}^2_{L^2} \quad \textrm{(by Corollary \ref{HCPlan:enhanced}).}
\end{align*}
\end{proof}
\end{lem}

\begin{lem}\label{lem:FT:inv} One has that $\mathcal{F}_{\rho,\overline{\psi}}\circ \mathcal{F}_{\rho,\psi} = \mathrm{id}_{\mathcal{C}(G(F))}$ and $\mathcal{F}_{\rho,\psi}^4 = \mathrm{id}_{\mathcal{C}(G(F))}$.  In particular $\mathcal{F}_{\rho,\psi}$ is a topological isomorphism.
\end{lem}
\begin{proof} Let $f\in \mathcal{C}(G(F))$. For $\pi \in\Temp_\Ind(G)$ with $\pi = I_P^G(\sigma)$ for some $(M,\sigma)\in\widetilde{\Temp}_\Ind(G)$,  Lemma \ref{FT:lem:important:1} and Lemma \ref{LLC:gammafactor:unitarity:2} imply
\begin{align*}
\pi(\mathcal{F}_{\rho,\overline{\psi}}\circ \mathcal{F}_{\rho,\psi}( f)) &= \gamma(\tfrac{1}{2},\sigma^\vee,\rho\vert_{^L{M}},\overline{\psi})\pi(\mathcal{F}_{\rho,\psi} (f)^\vee)\\
    &=\gamma(\tfrac{1}{2},\sigma^\vee,\rho\vert_{^L{M}},\overline{\psi})\gamma(\tfrac{1}{2},\sigma,\rho\vert_{^LM},\psi) \pi(f)\\
    &=\overline{\gamma(\tfrac{1}{2},\sigma,\rho\vert_{^L{M}},\psi)} \gamma(\tfrac{1}{2},\sigma,\rho\vert_{^LM},\psi)\pi(f)=\pi(f). 
\end{align*}
Similarly, by Lemma \ref{FT:lem:important:1} and Lemma \ref{lem:chis},
\begin{align*} 
\pi(\mathcal{F}_{\rho,\psi}^4 (f)) &= \gamma(\tfrac{1}{2},\sigma^\vee,\rho\vert_{^LM},\psi)\pi(\mathcal{F}_{\rho,\psi}^3(f)^\vee)\\
    &=\gamma(\tfrac{1}{2},\sigma^\vee,\rho\vert_{^LM},\psi)\gamma(\tfrac{1}{2},\sigma,\rho\vert_{^LM},\psi)\pi(\mathcal{F}_{\rho,\psi}^2(f))\\
    &=(\gamma(\tfrac{1}{2},\sigma^\vee,\rho\vert_{^LM},\psi)\gamma(\tfrac{1}{2},\sigma,\rho\vert_{^LM},\psi))^2\pi(f)= \pi(f).
\end{align*}
Since $\pi$ was arbitrary, the lemma now follows from the Plancherel formula (Theorem \ref{HCPlan:thm1}).
\end{proof}

\begin{lem} Let $M\in\mathcal{M}$ and $P\in \mathcal{P}(M)$. The following diagrams are commutative:
\begin{center}
    \begin{tikzcd}
        \mathcal{C}(\Temp_\Ind(G))\arrow[d,"(-)^{P}"]\arrow[rr,"\mathcal{F}_\rho^{\mathrm{op}}"]&&\mathcal{C}(\Temp_\Ind(G))\arrow[d,"(-)^{P}"]\\
        \mathcal{C}(\Temp_\Ind(M))\arrow[rr,"\mathcal{F}_{\rho\vert_{^LM}}^{\mathrm{op}}"']&&\mathcal{C}(\Temp_\Ind(M)),
    \end{tikzcd}
    \begin{tikzcd}
        \mathcal{C}(G(F))\arrow[rr,"\mathcal{F}_\rho"]\arrow[d,"(-)^{P}"]&&\mathcal{C}(G(F))\arrow[d,"(-)^{P}"]\\
        \mathcal{C}(M(F))\arrow[rr,"\mathcal{F}_{\rho\vert_{^LM}}"']&&\mathcal{C}(M(F))
    \end{tikzcd}
\end{center}
\end{lem}
\begin{proof}  The commutativity of the left diagram follows from the fact that $\gamma_\rho\vert_{\widetilde{\Temp}_\Ind(M)} = \gamma_{\rho\vert_{^LM}}$.  The commutativity of the right diagram follows from the commutativity of the left diagram and Lemma \ref{constant:map:HP:nice}.
\end{proof}

\begin{proof}[Proof of Theorem \ref{thm:main}]
    The Fourier transform is constructed in \eqref{FT:defn:1}.
    The identity \eqref{operator:bound} is part of Lemma \ref{FT:lem:important:1}.
    The Fourier transform is unitary by Lemma \ref{lem:FT:unit}, the twisted equivariance property is \eqref{FT:equivariant:close}, and the identity $\mathcal{F}_{\rho,\overline{\psi}} \circ \mathcal{F}_{\rho,\psi}=\mathrm{id}_{\mathcal{C}(G(F))}$ is part of Lemma \ref{lem:FT:inv}.  
\end{proof}

\subsection{The kernel} \label{ssec:kernel} By a \textbf{distribution on $G(F)$} we mean  a continuous linear functional $\ell:\mathcal{C}(G(F))\to\mathbb{C}.$
If $\ell$ is a distribution on $G(F)$ and $f \in \mathcal{C}(G(F))$ define $\ell\ast f:G(F)\to\mathbb{C}$ by the formula
\begin{align*}
    \ell \ast f(x) := \ell(\mathcal{R}(x,1)f^\vee) = \ell((\mathcal{R}(1,x)f)^\vee).
\end{align*}
We say that $\ell$ is \textbf{invariant }if $\ell\circ \mathcal{R}(g,g) = \ell$ for $g\in G(F)$.

\begin{lem}\label{TempDist:inv:end} If $\Phi:\mathcal{C}(G(F))\to\mathcal{C}(G(F))$ is a $G(F)\times G(F)$-equivariant continuous linear operator, then there exists an invariant distribution $\ell$ on $G(F)$ such that $\Phi(f) = \ell \ast f$ for $f\in \mathcal{C}(G(F))$.
\end{lem}
\begin{proof}
The distribution $\ell$ is given by $\ell(f):= \Phi(f^\vee)(1).$
\end{proof}

By \eqref{FT:equivariant:close} we see that 
\begin{align*}
    \mathcal{F}_\rho\circ(-)^\vee:\mathcal{C}(G(F))\lto \mathcal{C}(G(F))
\end{align*}
is a $G(F) \times G(F)$-equivariant continuous linear operator. Therefore Lemma \ref{TempDist:inv:end} implies there exists an invariant distribution
\begin{align*}
    J_\rho:\mathcal{C}(G(F))\lto\mathbb{C}
\end{align*}
satisfying $J_\rho\ast f = \mathcal{F}_\rho(f^\vee)$ for $f\in \mathcal{C}(G(F)).$  Equivalently
\begin{align}\label{FT:kernel:formula}
    \mathcal{F}_\rho(f) = J_\rho\ast f^\vee.
\end{align}
We refer to $J_{\rho}$ as the \textbf{kernel of $\mathcal{F}_{\rho}.$}

For the rest of this subsection assume $F$ is non-Archimedean. For any  compact open subgroup $K'\leq G(F)$, the Fourier transform $\mathcal{F}_\rho$ restricts to an isomorphism
\begin{align*}
    \mathcal{F}_\rho:\mathcal{C}(G(F)/\!/K')\tilde{\lto}\mathcal{C}(G(F)/\!/K').
\end{align*}
Let 
$e_{K'}:=\mathrm{meas}_{dg}(K')^{-1}\one_{K'}.$
For compact open subgroups $K'\leq G(F)$ let $\Gamma_{K'}\in\mathcal{C}(G(F)/\!/K')$ be the unique function such that
\begin{align*}
    \HP(\Gamma_{K'}) = \gamma_\rho\HP(e_{K'}).
\end{align*}

\begin{lem}\label{FT:kernel:onG//K'} Let $K'\leq G(F)$ be a compact open subgroup. For $f\in \mathcal{C}(G(F)/\!/K')$, one has
\begin{align*}
    \mathcal{F}_\rho(f) = \Gamma_{K'}^\vee\ast f^\vee = f^\vee\ast\Gamma_{K'}^\vee.
\end{align*}
\end{lem}
\begin{proof} Since $(f\ast g)^\vee = g^\vee\ast f^\vee$, it suffices to show
$    f\ast \Gamma_{K'}=\mathcal{F}_\rho(f)^\vee = \Gamma_{K'}\ast f.$
One has that
\begin{align*}
    \HP(\mathcal{F}_\rho(f)^\vee) = \gamma_\rho \HP(f).
\end{align*}
Since $f$ is $K'$ bi-invariant, we see $f\ast e_{K'} = e_{K'}\ast f = f.$ It follows that
\begin{align*}
 \HP(f\ast\Gamma_{K'})= \gamma_\rho \HP(f)\HP(e_{K'})=   \gamma_\rho \HP(f) = \gamma_\rho\HP(e_{K'})\HP(f) = \HP(\Gamma_{K'}\ast f).
\end{align*}
\end{proof}
\quash{
The kernel $J_{\rho}$ is represented by a $K\times K$-finite  function in $C_c^\infty(G(F)).$  More precisely:
\begin{lem}\label{FT:kernel:L2} Assume $F$ is non-Archimedean. There is a $w \in C^\infty(G(F))$ such that 
$$
w*f^\vee(1)=J_{\rho}*f^\vee(1)
$$ for all $f \in C_c^\infty(G(F)).$
\end{lem}

\begin{proof} Let $K'\leq G(F)$ be a compact open subgroup; then $\{gK'\}_{g\in G(F)}$ is an open cover of $G(F)$. The problem is local.  Thus it suffices to show that there exists an $w\in C^\infty_u(G(F))$ such that $J_\rho*f^\vee(1) = w*f^\vee(1)$ for all $f\in C^\infty(gK')$. Note that $C^\infty(gK')\subseteq\mathcal{C}(G(F)/\!/(K'\cap gK'g^{-1}))$, so by \eqref{FT:kernel:formula} and Lemma \ref{FT:kernel:onG//K'}  we see that
\begin{align*}
     J_\rho\ast f^\vee(1) = \mathcal{F}_\rho(f)(1) = \Gamma_{K'\cap gK'g^{-1}}^\vee\ast f^\vee(1). 
\end{align*}
Since $\Gamma_{K'\cap gK'g^{-1}}^\vee\in \mathcal{C}(G(F))\subseteq C^\infty(G(F))$ the result follows.
\end{proof}}

\begin{rem}
Lemma \ref{FT:kernel:onG//K'} implies a version of \cite[Conjecture 2.1.1(2)]{Luo:Ngo} in the non-Archimedean case.
\end{rem}

\subsection{The basic function}\label{Section:Basic:Function} 
In the rest of this subsection when $F$ is non-Archimedean we assume $G$ is unramified, $K$ is hyperspecial and $\psi$ is unramified. Let $P_0$ be a minimal parabolic subgroup containing $A_0.$ Thus $A_0$ is a maximal split torus in $P_0$, $M_0 = C_G(A_0)$ is a Levi subgroup of $P_0$, and $A_{M_0}=A_0$. \quash{ Since $A_0$ is maximal, it follows that the quotient $M_0/A_0$ is anisotropic, so $(M_0/A_0)(F)=M_0(F)/A_0(F)$ is compact. In particular, the trivial representation $\mathrm{triv}_{M_0}$ of $M_0(F)$ is square integrable; by abuse of notation }We write $\Im\Lambda_{M_0}$ for the connected component of $\Temp_\Ind(G)$ containing $(M_0,\mathrm{triv}_{M_0})$. When $F$ is non-Archimedean, $G$ is quasi-split, so $M_0$ is a maximal torus in $G.$

\begin{lem} \label{lem:unram:func} For $T\in\mathcal{C}(\mathrm{Temp}_{\mathrm{ind}}(G))^{K\times K}$, one has that $\mathrm{supp}(T)\subseteq\Im\Lambda_{A_0}$.
\end{lem}
\begin{proof} Let $M\in\mathcal{M}$, $P\in\mathcal{P}(M)$ and $\sigma\in \Pi_2(M)$. Put $\pi = I_P^G(\sigma)$. 
It is enough to prove that if $\pi^K \neq 0$ then 
$M=M_0$ and $\sigma \in \Im \Lambda_{M_0}.$
 If $\pi^K \neq 0$ then $\sigma^{K \cap M(F)} \neq 0.$  Using this observation we are reduced to proving that if $\pi \in \Pi_2(G)$ and $\pi^K \neq 0$ then $G=M_0.$
If $F$ is Archimedean this follows from \cite[Theorem 13.8.1]{Wallach:RGII} and if $F$ is non-Archimedean it follows from \cite[Remark, \S 10.4]{Borel:Corvallis}.
\end{proof}

\begin{defn}\label{basicfunct:defn} Assume $F$ is non-Archimedean.  The \textbf{basic function} $b_{\rho}$ is the unique element in $\mathcal{C}(G(F)/\!/K)$ such that
\begin{align}\label{basicfunction:defn}
    \mathrm{HP}(b_{\rho}) = L_{\rho}\mathrm{HP}(\one_K) \in \mathcal{C}(\Temp_\Ind(G))^{K\times K}.
\end{align}
\end{defn}
\noindent Here $L_{\rho}$ is defined as in \eqref{loc:factors}.
\quash{

By Corollary \ref{spherical:inversion}, one then has
\begin{align} \label{basic:inversion:L}
    b_{\rho}(g)=\int_{\Im\Lambda_{M_0}}L\left( \tfrac{1}{2},\chi,\rho|_{{}^LM_0}\right)\mathfrak{S}_{\chi^{-1}}(g)d\chi.
\end{align}
Previously in the literature the basic function has been defined using the spherical Fourier transform \cite{WWLiSat,Luo:basicfunction}. In view of \eqref{basic:inversion:L}, this agrees with our definition.}

\begin{lem}\label{Basic:Function:Fixed:Fourier} If $\psi$ is unramified then $\mathcal{F}_{\rho}(b_{\rho}) = b_{\rho}$.
\end{lem}
\begin{proof} By Lemma \ref{FT:lem:important:1}, for $\chi \in \Im\Lambda_{M_0}$ we have
\begin{align*}
    I_{P_0}^G(\chi)(\mathcal{F}_{\rho}(b_{\rho})^\vee) &= \gamma(\tfrac{1}{2},\chi,\rho\vert_{^LM_0},\psi)I_{P_0}^G(\chi)(b_{\rho})\\
    &=\gamma(\tfrac{1}{2},\chi,\rho\vert_{^LM_0},\psi)L(\tfrac{1}{2},\chi,\rho\vert_{^LM_0},\psi)\mathrm{HP}(\one_{K})(\chi)\\
    &=\varepsilon(\tfrac{1}{2},\chi,\rho\vert_{^LM_0},\psi)L(\tfrac{1}{2},\chi^{-1},\rho\vert_{^LM_0},\psi)\mathrm{HP}(\one_K)(\chi)\\
    &=\varepsilon(\tfrac{1}{2},\chi,\rho\vert_{^LM_0},\psi)I_{P_0}^G(\chi)(b_{\rho}^\vee)\qquad \text{(by \eqref{basicfunction:defn})}.
\end{align*}
\quash{\textcolor{red}{Should we cite Lemma 2.8 too? because   $I_{P_0}^G(\chi)^{\vee}(b_{\rho}) = I_{P_0}^G(\chi)(b_{\rho}^{\vee})^{\intercal}$ by Lemma 2.8, the definition of $I_{P_0}^G(\chi)(b_{\rho}^{\vee})^{\intercal}$ and then (4.2.1) give the result}}Since $\psi$ and $\chi$ are unramified $
    \varepsilon(\tfrac{1}{2},\chi,\rho\vert_{^LM_0},\psi)=1.$  Combining Lemma \ref{lem:unram:func} with the  Plancherel formula (Theorem \ref{HCPlan:thm1}) we deduce that $\mathcal{F}_{\rho}(b_\rho)^\vee=b_{\rho}^\vee$ and hence $\mathcal{F}_{\rho}(b_\rho)=b_{\rho}.$
\end{proof}

\subsection{Polynomial and holomorphic sections} \label{ssec:growth} 
For any semi-standard parabolic $P$ of $G$ we choose a continuous function $\mathbf{a} \times \mathbf{k}:G(F) \to P(F) \times K$ such that for $g \in G(F)$ one has
\begin{align} \label{ak}
  g=\mathbf{a}(g)\mathbf{k}(g).
\end{align}
When $F$ is non-Archimedean we choose $\mathbf{a}$ to be locally constant.

\subsubsection{The non-Archimedean case} Assume $F$ is non-Archimedean. 

\begin{defn} \label{defn:poly0} A section $T \in C^\infty(\Temp_\Ind(G))$ is  \textbf{polynomial} if for all $(M,\sigma)\in\widetilde{\mathrm{Temp}}_{\Ind}(G)$, the map 
\begin{align*}
    \Im\Lambda_M &\lto \mathcal{HS}(I_P^G(\sigma))^{\mathrm{sm}}\\
    \chi &\longmapsto T(M,\sigma\otimes\chi)
\end{align*} 
has image contained in a finite-dimensional space $W$
and the induced map $\Im\Lambda_M\to W$ is the restriction of an algebraic map $\Lambda_M\to W.$
\end{defn}
\noindent Here we give $\Lambda_M$ the structure of the $\CC$-points of a torus over $\CC$ as discussed above \eqref{alg:torus}. \quash{Denote by
\begin{align}\label{defn:polycpt}
    \mathrm{Poly}(\Temp_\Ind(G)) \leq \mathcal{C}(\Temp_\Ind(G))
\end{align}
denote the subspace of polynomial sections with compact supports.}

To clarify the role of the assumptions in the definition we prove the following lemma:
\begin{lem}
If  $T \in C^\infty(\Temp_{\mathrm{Ind}}(G))$ then $T(\Im\Lambda_M)$ is contained in a finite-dimensional subspace.
\end{lem}
\begin{proof}
One has 
\begin{align*}
(I_{P}^G(\sigma) \otimes I_{P}^G(\sigma^\vee))^{\mathrm{sm}}=\mathrm{colim}_{K'}((I_{P}^G(\sigma) \otimes I_{P}^G(\sigma^\vee))^{\mathrm{sm}})^{K' \times K'}
\end{align*}
where the colimit is over compact open subgroups $K' \leq G(F).$
Each of the spaces $((I_{P}^G(\sigma) \otimes I_{P}^G(\sigma^\vee))^{\mathrm{sm}})^{K' \times K'}$ is endowed with the canonical topology on a finite dimensional vector space and $(I_{P}^G(\sigma) \otimes I_{P}^G(\sigma^\vee))^{\mathrm{sm}}$ is given the strict colimit topology.  Bounded subsets with respect to this colimit topology are contained in $((I_{P}^G(\sigma) \otimes I_{P}^G(\sigma^\vee))^{\mathrm{sm}})^{K' \times K'}$ for some $K'$ \cite[\S 6.5]{Schaefer:Wolff}.

On the other hand, by definition of $C^\infty(\mathrm{Temp}_{\mathrm{Ind}}(G))$ and \cite[Proposition A.3.1(ii)]{BP:local:GGP}, for any $T \in C^\infty(\Temp_{\mathrm{Ind}}(G))$, the image $T(\Im\Lambda_M)$ is contained in a bounded subset.   
\end{proof}

The extension of a polynomial section to a morphism from $\Lambda_M $ is unique and we will identify polynomial sections and their extensions when convenient.  Thus a polynomial section may be regarded as a section of $\mathcal{HS}^{\mathrm{sm}} \to \widetilde{\mathrm{Temp}}_{\Ind}(G)_{\CC}.$

For polynomial sections $T \in C^\infty(\Temp_{\Ind}(G))$ and $(\chi ,(M,\sigma))\in \Lambda_G \times \widetilde{\Temp}_{\Ind}(G)$ define
\begin{align}
T_{\chi}(M,\sigma):=T(M,\sigma \otimes \chi|_M).
\end{align}

\begin{lem}\label{polysec:properties} Let $T\in C^\infty(\Temp_{\Ind}(G))$ be a polynomial section.
\begin{enumerate}[label=(\alph*)]
    \item For $\chi\in\Lambda_G$, $T_{\chi}$ is a polynomial section.
    
    \item One has $ T \in \mathcal{C}(\mathrm{Temp}_{\mathrm{Ind}}(G))$ if and only if the support of $T$ is contained in  a finite set of connected components of $\mathrm{Temp}_{\mathrm{Ind}}(G).$
    \item If $T\in \mathcal{C}(\Temp_{\Ind}(G))$, then $T_\chi\in \mathcal{C}(\Temp_{\Ind}(G))$ for all $\chi\in\Lambda_G$.
    \item \label{transpose:nA} The section $T^\intercal$ defined in Lemma \ref{HCPlan:transpose} is a polynomial section.
    \item \label{poly:mult} If $T$ and $S$ are  polynomial sections then $TS$ is  a polynomial section. 
\end{enumerate}\qed
\end{lem}

\begin{lem}\label{polysec:cptsupp} If $f\in C_c^\infty(G(F))$ then  $\mathrm{HP}(f)$ is a polynomial section.
\end{lem}
\begin{proof} \quash{Let $(M,\sigma)\in\widetilde{\mathrm{Temp}}_{\Ind}(G)$ and $f\in C_c^\infty(G(F))$ be fixed.  Let $K'$ be an open compact subgroup by which $f$ is bi-invariant.  Then $I_{P}^G(\sigma\otimes\chi)(f)$ is $K'\times K'$-invariant for all $\chi \in \Lambda_M.$ 
We must show that the map
\begin{align*}
    \Lambda_M &\lto \mathcal{HS}(I_P^G(\sigma))^{K'\times K'}\\
   \chi&\longmapsto I_P^G(\sigma\otimes\chi)(f)
\end{align*}
is algebraic. }
Let $(M,\sigma) \in \Temp_{\Ind}(G)$ and $P \in \mathcal{P}(M).$
A standard argument using $K$-finiteness reduces the proof to showing that for any  $(\varphi,\varphi^\vee)\in I_P^G(\sigma)^{\mathrm{sm}}\times I_P^G(\sigma^\vee)^{\mathrm{sm}}$ and $g\in G(F)$ the map
\begin{align} \label{is:alg} \begin{split}
\Lambda_M &\lto \CC\\
    \chi &\longmapsto \langle I_P^G(\sigma\otimes\chi)(g)\varphi,\varphi^\vee\rangle \end{split}
\end{align}
is algebraic. We have 
\begin{align*}
\langle I_P^G(\sigma\otimes\chi)(g)\varphi,\varphi^\vee\rangle &= \int_K \langle I_P^G(\sigma \otimes \chi)(g)\varphi(k),\varphi^\vee (k)\rangle dk\\
    &=\int_K (\delta_P^{1/2}\chi)(\mathbf{a}(kg))\langle \varphi(\mathbf{k}(kg)),\varphi^\vee(k)\rangle dk.
\end{align*}
As noted after the definition of $\mathbf{a}$ and $\mathbf{k}$ in \eqref{ak}, we can take $\mathbf{a}$ to be locally constant. Hence, by Lemma \ref{chi(g):polynomial}, we conclude that \eqref{is:alg} is algebraic.
\quash{As we remarked after the definition of $\mathbf{a}$ and $\mathbf{k}$ in \eqref{ak}, we can take $\mathbf{a}$ to be locally constant, so the we deduce that \eqref{is:alg} is algebraic from Lemma \ref{chi(g):polynomial}.}
\end{proof}

\subsubsection{Archimedean case} Assume $F$ is Archimedean.
Let $V$ be a quasi-complete locally convex vector space over $\CC.$  A fairly comprehensive discussion of the notion of a holomorphic function $\CC^n \to V$ is given in \cite{Kruse}.  In what follows we will use the fact that holomorphy is equivalent to separate holomorphy (holomorphy in each variable separately).  Moreover holomorphy is equivalent to weak holomorphy. 

\begin{defn} An element $T\in C^\infty(\Temp_{\Ind}(G))$ is called a \textbf{holomorphic} section if for all $(M,\sigma)\in\widetilde{\mathrm{Temp}}_{\Ind}(G)$, the map
\begin{align*}
    \Im\Lambda_M&\lto\mathcal{HS}(I_P^G(\sigma))\\
    \chi&\longmapsto T(M,\sigma\otimes\chi)
\end{align*}
extends to a holomorphic map $\Lambda_M\to \mathcal{HS}(I_P^G(\sigma))$.
\end{defn}
\noindent If $T$ is a holomorphic section we continue to denote by $T$ its extension to 
$\widetilde{\Temp}_\Ind(G)_\mathbb{C}.$  It is a holomorphic section of the bundle $\mathcal{HS}\to \widetilde{\Temp}_\Ind(G)_\mathbb{C}.$

For any holomorphic section $T$ and $\chi \in \Lambda_G$ let
\begin{align}
T_{\chi}(M,\sigma):=T(M,\sigma \otimes \chi|_{M(F)})
\end{align}
for $(M,\sigma)\in\widetilde{\Temp}_\Ind(G)$. 

\begin{lem}\label{holosec:transpose} Let $T\in C^\infty(\Temp_{\Ind}(G))$ be a holomorphic section.
\begin{enumerate}[label=(\alph*)]
    \item If $\chi\in\Lambda_G$ then $T_{\chi}$ is a holomorphic section.
    \item \label{transposearch} The section $T^\intercal$ is a holomorphic section.
\end{enumerate} \qed
\end{lem}

\begin{lem}\label{holosec:composition} If $T,S$ are holomorphic sections then $TS$ is a holomorphic section. 
\end{lem}
\begin{proof} The section $TS$ is the composition
\begin{align*}
    \Im\Lambda_M\to \Im\Lambda_M\times \Im\Lambda_M&\xrightarrow{\ T\times S\ } \mathcal{HS}(I_P^G(\sigma))\times \mathcal{HS}(I_P^G(\sigma))\to \mathcal{HS}(I_P^G(\sigma))\\
   (\chi,\chi')&\longmapsto (T(M,\sigma\otimes\chi), S(M,\sigma\otimes\chi'))
\end{align*}
where the first arrow is the diagonal embedding and the last arrow is composition. 

Using the fact that weak holomorphy is equivalent to holomorphy, we see that the composite of the last two arrows is separately holomorphic.  Since separate holomorphy is equivalent to holomorphy in the setting at hand we deduce that $TS$ is holomorphic.
\end{proof}

Recall the definition of the Fr\'echet algebra $\mathcal{S}(G(F))$ given in \cite{Casselman:SchwartzSpace}. By Weil restriction we may assume $F=\mathbb{R}$. Define the norm $\norm{\,\cdot\,}$ on $G(F)$ as in \eqref{norms}. For $f\in C^\infty(G(F))$ and $n\geq 1,\,u,v\in U(\mathfrak{g})$, let
\begin{align*}
    \norm{f}_{u,v,n} := \sup_{g\in G(F)}\norm{g}^n|\mathcal{R}(u,v)f(g)|.
\end{align*}
Then
\begin{align}\label{S(G(F)):Casselman:Frechet}
    \mathcal{S}(G(F)) := \left\{f\in C^\infty(G(F)): \norm{f}_{u,v,n}<\infty\text{ for all $n\geq 1,\,u,v\in U(\mathfrak{g})$}\right\}.
\end{align}

\begin{lem}\label{holosec:cptsupp} For $f\in \mathcal{S}(G(F))$, the section $\mathrm{HP}(f)$ is a holomorphic section and 
\begin{align*}
        \sup_{\sigma\in\Pi_2(M)/i\mathfrak{a}_M^*}\sup_{\lambda\in V_{\mathcal{K}}} \norm{I_P^G(\sigma_{\lambda})(u)\circ \mathrm{HP}(f)(M,\sigma_{\lambda})\circ I_P^G(\sigma_{\lambda})(v)}_{\mathrm{op}}N(\sigma_{\lambda})^k<\infty
\end{align*}
for all $M\in\mathcal{M},\,u,v\in U(\mathrm{Lie}(K)_\mathbb{C})$, $k\geq 0$ and compact subsets $\mathcal{K} \subset \mathfrak{a}_M^*$. 
\end{lem}
\begin{proof} 
 Let $(M,\sigma)\in\widetilde{\mathrm{Temp}}_{\Ind}(G)$ and $f\in \mathcal{S}(G(F))$. 
 Since the representation $I_P^G(\sigma_{\lambda})^{\mathrm{sm}}$ has moderate growth,
 the operator $\mathrm{HP}(f)(M,\sigma_{\lambda})$ is given by the absolutely convergent integral
\begin{align*}
    I_P^G(\sigma_{\lambda})(f)=\int_{G(F)} f(g) I_P^G(\sigma_{\lambda})(g)dg
\end{align*}
for all $\lambda \in \mathfrak{a}_{M\CC}^*.$

We first prove the bound.  By \cite[(1.25)]{Delorme:PaleyWiener}, we have 
$\norm{I_P^G(\sigma_{\lambda})(g)}_{\mathrm{op}} \leq \norm{g}^{|\mathrm{Re}(\lambda)|}.
$
 Then
\begin{align*}
    \norm{I_P^G(\sigma_{\lambda})(\mathcal{R}(u,v)f)}_{\mathrm{op}} \ll \int_{G(F)} |\mathcal{R}(u,v)f(g)|\norm{g}^{|\mathrm{Re}(\lambda)|} dg \ll \norm{f}_{u,v,k} 
\end{align*}
for some $k\gg_\mathcal{K} 1.$  Now argue as in \cite[Lemme 1]{Delorme:PaleyWiener} to obtain the bound.

We now prove that $\mathrm{HP}(f)$ is a holomorphic section.
The continuous dual of $\mathcal{HS}(I_P^G(\sigma))$ is isomorphic to $\mathcal{HS}(I_P^G(\sigma^\vee)).$ In this space smooth finite rank operators are dense.  Employing \cite[\S2, Remarque 1]{Grothendieck:holomorphic:I}, to prove that $\mathrm{HP}(f)$ is a holomorphic section it suffices to check that for each $\varphi\in I_P^G(\sigma)^{\mathrm{sm}}$ and $\varphi^\vee \in I_P^G(\sigma^\vee)^{\mathrm{sm}}$ the function
\begin{align*}
\mathfrak{a}_M^* &\lto \CC\\
\lambda &\longmapsto  \langle I_P^G(\sigma_\lambda)(f)\varphi,\varphi^\vee\rangle
\end{align*}
is holomorphic.  We have that 
\begin{align*}
    \langle I_P^G(\sigma_{\lambda})(f)\varphi,\varphi^\vee\rangle &= \int_{G(F)} f(g) \int_K \langle I_P^G(\sigma_{\lambda})(g)\varphi(k),\varphi^\vee(k)\rangle dk dg\\
    &=\int_{G(F)} f(g) \int_K e^{\langle \rho_P+\lambda,H_M(\mathbf{a}(g))\rangle}\langle \varphi(\mathbf{k}(kg)),\varphi^\vee(k)\rangle dk dg
\end{align*}
where $\rho_P$ is the half-sum of the positive roots. Differentiating  we deduce the result. 
\end{proof}

\section{The asymptotic $\rho$-Schwartz space} \label{sec:asymp}

In this section we give a definition of a Schwartz space attached to a tempered representation $\rho:{}^LG \to \GL_{V_\rho}(\CC).$ We assume the weak form of the local Langlands conjecture discussed in  \S\ref{sec:LLC}. 
When $F$ is Archimedean, for $M \in \mathcal{M}$ denote by 
\begin{align}
\mathcal{D}(M) \subseteq \Pi_2(M)
\end{align}
a set satisfying the conclusion of   Corollary \ref{LLC:Lfactor:pole:cor} with $(G,\rho)$ replaced by $(M,\rho|_{{}^LM}).$

If $F$ is non-Archimedean let
\begin{align*}
    \mathcal{S}_\rho^{\mathrm{as}}(\mathrm{Temp}_{\mathrm{Ind}}(G)):=\{T \in \mathcal{C}(\mathrm{Temp}_{\mathrm{Ind}}(G)):L_{\rho}^{-1}T \textrm{ is a polynomial section}\}.
\end{align*}
If $F$ is Archimedean, we let $\mathcal{S}_\rho^{\mathrm{as}}(\mathrm{Temp}_{\mathrm{Ind}}(G))$ be the set of $T \in \mathcal{C}(\mathrm{Temp}_{\mathrm{Ind}}(G))$ such that 
\begin{enumerate}
    \item $L_{\rho}^{-1}T$ 
    is a holomorphic section, and 
    \item \label{seminorm} for any $M\in\mathcal{M}$, any compact set $\mathcal{K} \subset \mathfrak{a}_M^*$, any polynomial $p:\mathfrak{a}^*_{M\CC} \to \CC$ such that $p(\lambda)L(\frac{1}{2},\sigma_\lambda,\rho)$ has no pole in $V_{\mathcal{K}}$ for all $\sigma \in \mathcal{D}(M),$ and any $u,v\in U(\mathrm{Lie}(K)_{\CC})$, the operator
    \begin{align*}
        p(\lambda)I_P^G(\sigma_\lambda)(u)\circ T(M,\sigma_\lambda)\circ I_P^G(\sigma_\lambda)(v),
    \end{align*}
    originally defined for $\lambda \in i\mathfrak{a}_M^*,$
    extends to a bounded operator on $I_P^G(\sigma_\lambda)$ for all $\lambda\in V_{\mathcal{K}}.$ Moreover for any $k\in\mathbb{Z}_{\geq 0}$ we have
    \begin{align*}
    \norm{T}_{u,v,k,\mathcal{K},p}:=\sup_{\sigma\in \mathcal{D}(M)}\sup_{\lambda\in V_{\mathcal{K}}} \norm{p(\lambda)I_P^G(\sigma_\lambda)(u)\circ T(M,\sigma_\lambda)\circ I_P^G(\sigma_\lambda)(v)}_{\mathrm{op}}N(\sigma_\lambda)^k<\infty.
    \end{align*}
\end{enumerate}
Corollary \ref{LLC:Lfactor:pole:cor} states that $p$ as in \eqref{seminorm} exist.

\begin{lem}\label{rhoSchwartz:G2Fmodule} $\mathcal{S}_\rho^{\mathrm{as}}(\mathrm{Temp}_{\mathrm{Ind}}(G))$ is a $G(F)\times G(F)$-invariant subspace of $\mathcal{C}(\mathrm{Temp}_{\mathrm{Ind}}(G))$.
\end{lem}
\begin{proof} As in the proof of Lemma \ref{HP:arch:bound} the bound in (2) implies the analogous bound for all $u,v\in U(\mathfrak{g})$. Thus the proof is the same as that of Lemma \ref{HP:Ginvariant}.
\end{proof}

\begin{defn}\label{Asymp:Schwartz:Space} The \textbf{asymptotic $\rho$-Schwartz space} $\mathcal{S}_\rho^{\mathrm{as}}(G(F))\leq \mathcal{C}(G(F))$ is defined as 
\begin{align*}
    \mathcal{S}_\rho^{\mathrm{as}}(G(F)) := \mathrm{HP}^{-1}(\mathcal{S}_\rho^{\mathrm{as}}(\Temp_\Ind(G))).
\end{align*} 
By Lemma \ref{rhoSchwartz:G2Fmodule} this is a $G(F)\times G(F)$-invariant subspace of $\mathcal{C}(G(F)).$
\end{defn}

We call this the asymptotic Schwartz space for two reasons.  First, it is a subspace of $\mathcal{C}(G(F))$ specified by stipulating the asymptotics of functions after applying  $\mathrm{HP}.$  Second, it is formally similar to the asymptotic Hecke algebra considered in \cite{BK:asymp}.

\begin{lem} \label{lem:as:preserve} The space $\mathcal{S}_\rho^{\mathrm{as}}(G(F))$ is preserved by $\mathcal{F}_{\rho}$.
\end{lem}
\begin{proof}
By commutativity of the diagram \eqref{FT:defn:5} and the fact that $\mathrm{HP}$ is an isomorphism by Theorem \ref{HCPlan:thm1}, it suffices to prove that $\mathcal{S}_\rho^{\mathrm{as}}(\mathrm{Temp}_{\mathrm{Ind}}(G))$ is stable under the operator $\mathcal{F}_{\rho}^{\mathrm{op}}$  of \eqref{FT:defn:4}.

Let $T\in \mathcal{S}_\rho^{\mathrm{as}}(\mathrm{Temp}_{\mathrm{Ind}}(G)).$  By the definition of $\intercal$ in Lemma \ref{HCPlan:transpose},  we have
\begin{align}\label{FopT:1}
    \frac{\mathcal{F}_{\rho}^{\mathrm{op}}(T)(\pi)}{L_\rho(\pi)} = \frac{\gamma_{\rho}(\pi^\vee)T(\pi^\vee)^\intercal}{L_\rho(\pi)} = \varepsilon_{\rho}(\pi^\vee)\frac{T(\pi^\vee)^\intercal}{L_\rho(\pi^\vee)}=\varepsilon_{\rho}(\pi^\vee)\left(\frac{T}{L_\rho}\right)^\intercal(\pi).
\end{align} 
By lemmas \ref{polysec:properties}\ref{transpose:nA}, \ref{holosec:transpose}\ref{transposearch} and \ref{LLC:gamma:rationality:1}, this proves that $L^{-1}_\rho\mathcal{F}_{\rho}^{\mathrm{op}}(T)$ is polynomial in the non-Archimedean case and holomorphic in the Archimedean case.

This completes the proof in the non-Archimedean case.  In the Archimedean case, we must additionally show $\norm{\mathcal{F}_{\rho}^{\mathrm{op}}(T)}_{u,v,k,\mathcal{K},p}<\infty$, where $u,v,k,\mathcal{K},p$ are as in \eqref{seminorm}. Let $p^\vee$ be another polynomial such that $p^\vee(\lambda)L(\tfrac{1}{2},(\sigma_\lambda)^\vee,\rho)$ is nonvanishing and holomorphic in $V_{\mathcal{K}}$. By 
Corollary \ref{LLC:Lfactor:pole:cor} we can choose $p^\vee$ from a finite set depending on $\mathcal{K}.$
Write
\begin{align*}
    &\norm{p(\lambda)I_P^G(\sigma_\lambda)(u)\circ \mathcal{F}_{\rho}^{\mathrm{op}}(T)(M,\sigma_\lambda)\circ I_P^G(\sigma_\lambda)(v)}_{\mathrm{op}}\\
    &=\norm{\dfrac{p(\lambda)\gamma(\tfrac{1}{2},(\sigma_\lambda)^\vee,\rho)}{p^\vee(\lambda)} p^\vee(\lambda)I_P^G(\sigma_\lambda)(u)\circ T(M,(\sigma_\lambda)^\vee)^{\intercal}\circ I_P^G(\sigma_\lambda)(v)}_{\mathrm{op}}.
\end{align*}
In view of Proposition \ref{prop:gamma:bounds}, it follows that $\norm{\mathcal{F}_{\rho}^{\mathrm{op}}(T)}_{u,v,k,\mathcal{K},p}<\infty$.
\end{proof}

\begin{lem}\label{lem:Sas:moduleCcinfty} One has $\mathcal{S}(G(F))\leq \mathcal{S}_\rho^{\mathrm{as}}(G(F))$ and
\begin{align*}
    \mathcal{S}(G(F))\ast \mathcal{S}_\rho^{\mathrm{as}}(G(F))\leq \mathcal{S}_\rho^{\mathrm{as}}(G(F)),\quad
    \mathcal{S}_\rho^{\mathrm{as}}(G(F))\ast \mathcal{S}(G(F))\leq \mathcal{S}_\rho^{\mathrm{as}}(G(F)).
\end{align*}
\end{lem}
\begin{proof} The first assertion follows from Lemma \ref{polysec:cptsupp} and Lemma \ref{holosec:cptsupp}. To prove the left containment, let $\varphi\in \mathcal{S}(G(F))$ and $f\in\mathcal{S}_\rho^{\mathrm{as}}(G(F))$. Then
\begin{align*}
    \frac{\mathrm{HP}(\varphi\ast f)}{L_\rho} = \mathrm{HP}(\varphi)\frac{\mathrm{HP}(f)}{L_\rho}.
\end{align*}
In view of Lemma \ref{polysec:properties}\ref{poly:mult} and Lemma \ref{holosec:composition}, we are left with proving that the seminorms defined in \eqref{seminorm} are bounded in the Archimedean case. This follows from the bound in Lemma \ref{holosec:cptsupp}. The proof of the second containment one is essentially the same.
\end{proof}

\begin{lem}\label{constant:map:asSchwartz:preserve} Let $M\in\mathcal{M}$ and $P\in \mathcal{P}(M)$. The map \eqref{constant:map:spectral} induces a map
\begin{align*}
    (-)^{P}:\mathcal{S}_\rho^{\mathrm{as}}(\mathrm{Temp}_{\mathrm{Ind}}(G))\lto \mathcal{S}_{\rho\vert_{^LM}}^{\mathrm{as}}(\mathrm{Temp}_{\mathrm{Ind}}(M)).
\end{align*}
Consequently, the constant term map $f\mapsto f^P$ induces a map
\begin{align*}
    (-)^{P}:\mathcal{S}_\rho^{\mathrm{as}}(G(F))\lto \mathcal{S}_{\rho\vert_{^LM}}^{\mathrm{as}}(M(F)).
\end{align*}
\end{lem}
\begin{proof} 
By Lemma \ref{constant:map:HP:nice} it suffices to prove the first statement. 
Observe that $L_\rho\vert_{\widetilde{\Temp}_\Ind(M)} = L_{\rho\vert_{^LM}}.$  Using the last statement of Lemma \ref{Ind:evaluation:continuous} in the Archimedean case we deduce that if $T\in \mathcal{S}_\rho^{\mathrm{as}}(\mathrm{Temp}_{\mathrm{Ind}}(G))$ then $L_{\rho\vert_{^LM}}^{-1}T^{P}$ is polynomial (resp.~holomorphic). 
The argument proving Lemma \ref{constant:map:spectral:welldefined} implies the
 required bound for $T^{P}$ in the Archimedean case.
\end{proof}

\begin{lem}\label{Basic:Function:Asymptotic} If $F$ is non-Archimedean then  $b_{\rho} \in \mathcal{S}_\rho^{\mathrm{as}}(G(F))$. \qed
\end{lem}

\subsection{Tempered local zeta integrals}

For a representation $\pi$ of $G(F)$ and $(\varphi,\varphi^\vee) \in \pi^{\mathrm{sm}} \times (\pi^\vee)^{\mathrm{sm}}$, let
$
c_{\varphi,\varphi^\vee}(g):=\langle \pi(g)\varphi,\varphi^\vee\rangle
$
be the corresponding smooth matrix coefficient.  We let
\begin{align}
\mathcal{C}(\pi):=\langle c_{\varphi,\varphi^\vee}:(\varphi,\varphi^\vee) \in \pi^{\mathrm{sm}} \times (\pi^\vee)^{\mathrm{sm}} \rangle \leq C^\infty(G(F))
\end{align}
be the $\CC$-span of these matrix coefficients in $C^\infty(G(F))$. Note that $\{c^\vee:c\in\mathcal{C}(\pi)\} = \mathcal{C}(\pi^\vee)$, where we recall that for a function $c\in C^\infty(G(F))$ we have set $c^\vee(g):=c(g^{-1})$.

\quash{For $\varphi \in \mathcal{C}(\pi),$ $\chi \in \Lambda_G,$ and $\lambda \in \mathfrak{a}_{G\CC}^*$  we let
\begin{align} \label{char:twist} \begin{split}
\varphi_{\chi}(g):=\varphi(g)\chi(g), \quad 
\varphi_{\lambda}(g):=\varphi(g)e^{\langle \lambda,H_G(g)\rangle} \end{split}
\end{align}}

\begin{defn} For a representation $\pi$ of $G(F)$ and $(\lambda,f,c) \in \mathfrak{a}_{G\CC}^* \times  C^\infty(G(F)) \times \mathcal{C}(\pi),$ set
\begin{align*}
    Z(\lambda ,f,c) := \int_{G(F)} e^{\langle \lambda,H_G(g)\rangle}f(g)c(g) dg
\end{align*}
whenever the integral is absolutely convergent\quash{or analytically continued in $\lambda$ from a region of absolutely convergence}.  We refer to $Z(\lambda,f,c)$ as a \textbf{local zeta integral.}
\end{defn}

We point out that for any $f \in \mathcal{C}(G(F)),$ $\pi \in \mathrm{Temp}(G)$, $(\varphi,\varphi^\vee) \in \pi^{\mathrm{sm}} \times (\pi^\vee)^{\mathrm{sm}}$ and $\lambda \in i\mathfrak{a}_G$ one has that
\begin{align}\label{zeta:pairing}
Z(\lambda,f,c_{\varphi,\varphi^\vee})=\langle \pi_{\lambda}(f)\varphi,\varphi^\vee \rangle.
\end{align}
In particular, the integral defining $Z(\lambda,f,c)$ is absolutely convergent.

\begin{lem}\label{prop:loc:zeta:lem} For $(\lambda,f,c) \in i\mathfrak{a}_G^* \times \mathcal{C}(G(F)) \times \mathcal{C}(\pi)$ one obtains that
\begin{align*}
    Z(-\lambda,\mathcal{F}_{\rho}(f),c^\vee) = \gamma(\tfrac{1}{2},\pi_\lambda,\rho)Z(\lambda,f,c).
\end{align*}    
\end{lem}
\begin{proof} In view of \eqref{zeta:pairing}, this follows from Lemma \ref{FT:lem:important:1}.\quash{Then using 
 Lemma \ref{FT:lem:important:1},
for $(\varphi,\varphi^\vee) \in \pi^{\mathrm{sm}} \times (\pi^\vee)^{\mathrm{sm}}$
we have
\begin{align*}
\gamma(\tfrac{1}{2},\pi_\lambda,\rho)\langle I_P^G(\sigma)_{\lambda}(f)\varphi,\varphi^\vee\rangle&=
    \langle I_P^G(\sigma)_{\lambda}(\mathcal{F}_{\rho} (f)^\vee)\varphi,\varphi^\vee\rangle \\&= \int_{G(F)} (\mathcal{F}_{\rho} (f))^\vee(x)\langle \pi_{\lambda}(x)\varphi,\varphi^\vee\rangle dx\\
    &=\int_{G(F)} \mathcal{F}_{\rho} (f)(x)\langle \pi_{\lambda}(x^{-1})\varphi,\varphi^{\vee}\rangle dx\\
    &=Z(-\lambda,\mathcal{F}_{\rho} f(x),c_{\varphi,\varphi^\vee}^\vee).
\end{align*}}
\end{proof}

\begin{prop} \label{prop:loc:zeta} Let $\pi\in \Temp(G)$. For $(f,c)\in \mathcal{S}_\rho^{\mathrm{as}}(G(F))\times \mathcal{C}(\pi)$, the local zeta integral $Z(\lambda,f,c)$ extends to a meromorphic function of $\lambda\in \mathfrak{a}_{G\CC}^*$, and is a holomorphic (resp. polynomial) multiple of $L(\tfrac{1}{2},\pi_{\lambda},\rho)$ when $F$ is Archimedean (resp. non-Archimedean). It satisfies the functional equation
\begin{align}\label{prop:loc:zeta:fe}
    Z(-\lambda,\mathcal{F}_{\rho}(f),c^\vee) = \gamma(\tfrac{1}{2},\pi_\lambda,\rho)Z(\lambda,f,c).
\end{align}
\end{prop}

\noindent In the non-Archimedean case the statement that $Z(\lambda,f,c)$ is a polynomial multiple of $L(\tfrac{1}{2},\pi_\lambda,\rho)$ means there exists $p\in \mathbb{C}[\Lambda_G]$ such that $Z(\lambda,f,c) = p(e^{\langle\lambda,H_G(\cdot)\rangle})L(\tfrac{1}{2},\pi_\lambda,\rho)$.
 
\begin{proof} For $\lambda\in i\mathfrak{a}_G^*$, we rewrite \eqref{zeta:pairing} as 
\begin{align*}
Z(\lambda,f,c_{\varphi,\varphi^\vee}) = L(\tfrac{1}{2},\pi_\lambda,\rho)\left\langle \frac{\pi_\lambda(f)}{L(\tfrac{1}{2},\pi_\lambda,\rho)}\varphi,\varphi^\vee \right\rangle.
\end{align*}
Realizing $\pi$ as a subrepresentation of $I_P^G(\sigma)$ for some $(M,\sigma) \in \widetilde{\Temp}_\Ind(G)$, we see that
\begin{align*}
    Z(\lambda,f,c_{\varphi,\varphi^\vee}) = L_\rho(M,\sigma_\lambda)\left\langle \frac{\mathrm{HP}(M,\sigma_\lambda)(f)}{L_\rho(M,\sigma_\lambda)}\varphi,\varphi^\vee \right\rangle.
\end{align*}
Since $f\in \mathcal{S}^{\mathrm{as}}_{\rho}(G(F))$, the right hand side extends to a meromorphic function in $\lambda\in \mathfrak{a}_{G\CC}^*$. This provides a meromorphic continuation of $Z(\lambda,f,c_{\varphi,\varphi^\vee}).$ Moreover, it shows that $Z(\lambda,f,c_{\varphi,\varphi^\vee})$  is a holomorphic (resp. polynomial) multiple of $L(\tfrac{1}{2},\pi_{\lambda},\rho)$ when $F$ is Archimedean (resp. non-Archimedean). Finally, since both sides of \eqref{prop:loc:zeta:fe} are meromorphic functions in $\lambda$, the equality holds by Lemma \ref{prop:loc:zeta:lem} and the identity principle.
\end{proof}

In fact the local zeta integrals converge for a larger range of $\lambda$. Write $\rho=\rho_1\oplus\cdots\oplus\rho_m$ where each $\rho_i$ is irreducible. By irreducibility and \eqref{LLC:UT}, there exist linear functionals $\mu_i:\mathfrak{a}_{G}^*\to\mathbb{R}$ for all $1 \leq i \leq m$ such that 
\begin{align*}
    L_\rho(M,\sigma_\lambda) = \prod_{i=1}^mL(\tfrac{1}{2} + \mu_i(\lambda),\sigma,\rho_i\vert_{^L{M}})
\end{align*}
for all $(M,\sigma)\in\widetilde{\Temp}_\Ind(G)$. Let
\begin{align*}
    C := \bigcap_{i=1}^m \left\{\lambda\in \mathfrak{a}_{G}^*\mid \mu_i(\lambda)>-\tfrac{1}{2}\right\}
\end{align*}
and form the cylinder $V_C\subseteq \mathfrak{a}_{G\CC}^*$ as in \eqref{cylinder}.

\begin{lem}
Let $f \in \mathcal{S}_{\rho}^{\mathrm{as}}(G(F)).$  Then $fe^{\langle \lambda,H_G(\cdot) \rangle} \in \mathcal{C}(G(F))$ for $\lambda\in V_C$. \quash{ all $\lambda \in \mathfrak{a}_{G\CC}^*$ such that $\mathrm{Re}(\mu_\rho(\lambda))>-\tfrac{1}{2}.$}
\end{lem}

\begin{proof} 
For $\lambda\in V_C$ we have a well-defined operator
\begin{align*}
\mathrm{HP}(f)_{\lambda}(M,\sigma)
:= (L_\rho^{-1}\mathrm{HP}(f))(M,\sigma_\lambda) \cdot L_{\rho}(M,\sigma_{\lambda}) \in\mathcal{HS}(I_P^G(\sigma)).
\end{align*}
Indeed, the quotient is a polynomial section in the non-Archimedean case and a holomorphic section in the Archimedean case by definition of $\mathcal{S}_{\rho}^{\mathrm{as}}(G(F)).$ 
It is clear that $\mathrm{HP}(f)_{\lambda} \in \mathcal{C}(\mathrm{Temp}_{\mathrm{Ind}}(G))$ in the non-Archimedean case.  

We claim the same is true in the Archimedean case.  
We must show 
\begin{align}\label{zeta:lem:1:bound}
        \sup_{\sigma\in \Pi_2(M)} \norm{I_P^G(\sigma)(u) \circ D\mathrm{HP}(f)_{\lambda}(M,\sigma) \circ I_P^G(\sigma)(v)}_{\mathrm{op}} N(\sigma)^k<\infty
\end{align}
where all unexplained symbols are as in \S\ref{sec:HPmap}. Let $\mathcal{K} \subset \mathfrak{a}^*_{M\CC}$ be the closure of a bounded neighborhood of the origin. Choose $\mathcal{K}$ small enough that 
 $L(\frac{1}{2},\sigma_{\lambda'},\rho)$ has no pole in $V_{\mathcal{K}}$ for all $\sigma\in\Pi_2(M)$ and all $\lambda' \in \lambda+V_{\mathcal{K}}.$ Noticing $\lambda + V_\mathcal{K} = V_{\mathrm{Re}(\lambda)+\mathcal{K}}$, by definition of $\mathcal{S}_{\rho}^{\mathrm{as}}(G(F))$ we have
    \begin{align*}
\sup_{\sigma\in\Pi_2(M)/i\mathfrak{a}_M^*}\sup_{\lambda'\in \lambda+V_{\mathcal{K}}} \norm{p(\lambda')I_P^G(\sigma_{\lambda'})(u)\circ T(M,\sigma_{\lambda'})\circ I_P^G(\sigma_{\lambda'})(v)}_{\mathrm{op}} N(\sigma)^k <\infty
    \end{align*}
for any polynomial $p$ on $\mathfrak{a}_{M\CC}^*.$
Restricting the bound to $i\mathfrak{a}_M^*\subseteq V_{\mathcal{K}}$ proves \eqref{zeta:lem:1:bound} when $D$ is constant one (by taking $p(\lambda)=1$). Since we chose $\mathcal{K}$ to be the closure of an open set, bounds for general $D$ follow from Lemma \ref{FT:defn:lem:-1}.

To conclude the proof, we contend that for $(g,\lambda)\in G(F) \times V_C$ one has that 
\begin{align} \label{analyt}
   f(g)e^{\langle \lambda,H_G(g)\rangle} = \mathrm{HP}^{-1}(\mathrm{HP}(f)_{\lambda})(g).
\end{align}
This is clear when $\mathrm{Re}(\lambda)=0$, so by the identity principle it suffices to show that both sides of \eqref{analyt} are holomorphic in $\lambda$.  This is immediate for the left hand side. For the right hand side this follows from the fact that $\mathrm{HP}^{-1}$ is a continuous linear map.  
\end{proof}

\begin{cor} \label{zeta:integral:convergence:cone}
For $(f,\pi) \in \mathcal{S}_{\rho}^{\mathrm{as}}(G(F)) \times \mathrm{Temp}(G)$ and $c \in \mathcal{C}(\pi)$ the local zeta integral $Z(\lambda,f,c)$ converges absolutely for $\lambda\in V_C$. \qed
\end{cor}

\section{The $\rho$-Schwartz space} \label{sec:rho:Sch} 

Throughout this section  $F$ is a non-Archimedean local field.

\begin{defn} \label{defn:ac}
The space $C_{ac}^{\infty}(G(F))$ of \textbf{almost compactly supported smooth functions} is the space of $K\times K$-finite functions $f: G(F) \to \CC$ such that for all compact sets $\Omega \subseteq \mathfrak{a}_G$ the function $g\mapsto f(g)\one_{\Omega}(H_G(g))$ is compactly supported. 
\end{defn}

Elements of the asymptotic $\rho$-Schwartz space $\mathcal{S}^{\mathrm{as}}_{\rho}(G(F))$ need not be of almost compact support:
\begin{example} \label{BP:example}
 Let $G=\GL_2$ and let $\rho:\GL_2 \to \GL_2$ be the identity (i.e.~the standard representation).  Let $I < \GL_2(\OO_F)$ be the Iwahori subgroup. 
Let $\mathrm{St}$ denote the Steinberg representation; then $\mathrm{St}^{I}$ is $1$-dimensional \cite[\S 14.4]{Bushnell:Henniart}. Pick any $\varphi\in\mathrm{St}^I$ with $\langle\varphi,\varphi\rangle=1$.

Let $f \in \mathcal{C}(\GL_2(F))$ be the unique function such that $\mathrm{HP}(f)$ is zero outside of $\mathrm{St} \otimes i\mathfrak{a}_{\GL_2}^*,$ and the restriction of $\mathrm{HP}(f)$ to $\mathrm{St} \otimes i\mathfrak{a}_{\GL_2}^*$ is the orthogonal projection to $\mathrm{St}^I \otimes i\mathfrak{a}_{\GL_2}^*$. Clearly $\HP(f)\in\mathcal{S}_\rho^{\mathrm{as}}(\Temp_\Ind(\GL_2))$, so that $f\in\mathcal{S}_{\rho}^{\mathrm{as}}(\GL_2(F))$. One computes
\begin{align*}
    f(g)=\one_{\OO_F^\times}(\det g)\langle \mathrm{St}(g^{-1})\varphi,\varphi\rangle.
\end{align*}
Then $\mathrm{supp}(f)\subseteq \{g\in\GL_2(F): |\det g|=1\} = \{g\in M_2(F):|\det g|=1\}$. Note that $\mathrm{supp}(f)$ cannot be compact in $M_2(F)$ since this would imply the Steinberg representation is supercuspidal.
\end{example}

Motivated by this example we isolate a proper subspace of $\mathcal{S}_{\rho}^{\mathrm{as}}(G(F)).$  To proceed we must impose some additional assumptions on our tempered representation $\rho:{}^LG \to \GL_{V_{\rho}}(\CC).$ We assume that
\begin{align} \label{rho:dual:center}
    \text{there exists $ d\in X^*(G)$ such that $\rho\circ d^\vee(z) = z\cdot \mathrm{id}_{V_\rho}$ for all $z\in \CC^\times$.}
\end{align}

\begin{defn}
    The \textbf{$\rho$-Schwartz space} is
    $$
    \mathcal{S}_{\rho}(G(F))=
    \mathcal{S}_{\rho}^{\mathrm{as}}(G(F)) \cap C_{ac}^\infty(G(F))
    $$
\end{defn} 

\begin{rem}
The Schwartz space $\mathcal{S}_{\rho}(G(F))$ is still defined if $\rho$ does not satisfy \eqref{rho:dual:center}, but the resulting space has unexpected properties.  For example, if $G=\SL_2$ and $\rho:{}^L\SL_2 \to \GL_V$ is any nontrivial representation then
$
b_{\rho} \not \in C_{ac}^\infty(G(F))=C_c^\infty(\SL_2(F)).$
\end{rem}

As in \cite[\S 5]{Ngo:Hankel} and \cite[\S7]{H:Y}, under assumption \eqref{rho:dual:center} one can attach an affine reductive monoid $M_\rho$ with unit group $G$ to $\rho.$ In \cite{H:Y}, Hsu and Yao prove the following:
\begin{thm} Assume \eqref{LLC:LCT}-\eqref{LLC:para} and \eqref{rho:dual:center}.
\begin{enumerate}
    \item The support of every element of $\mathcal{S}_{\rho}(G(F))$ is contained in a compact subset of $M_{\rho}(F)$ If $G$ is unramified then $b_{\rho} \in \mathcal{S}_{\rho}(G(F))$.
    \item Assume moreover \eqref{LLC:gamma:compat}. Then the space $\mathcal{S}_{\rho}(G(F))$ is preserved by $\mathcal{F}_{\rho}$.
\end{enumerate}
\qed
\end{thm}

\section{Proof of Theorem \ref{main:thm:intro}}
\label{BKN:statement}

In this section we collect our previous work to prove Theorem \ref{main:thm:intro}. We recall the setting. Let $\psi:F \to \CC^\times$ be a nontrivial character and let $\rho:{}^LG\to \GL_{V_{\rho}}(\mathbb{C})$ be a tempered representation. Assume that there is an isomorphism $d:G/G^{\mathrm{der}} \tilde{\to} \GG_m$ such that the dual map $\rho \circ d^\vee :\CC^\times \to \GL_{V_{\rho}}(\CC)$
satisfies $\rho\circ d^\vee(z) = z\cdot \mathrm{id}_{V_\rho}$ for all $z\in \CC^\times$.

In \S\ref{sec:FourieronHCSchwartz} we defined the $\rho$-Fourier transform 
\begin{align*}
    \mathcal{F}_{\rho,\psi}:\mathcal{C}(G(F))&\lto \mathcal{C}(G(F)),\\
    f&\longmapsto (\cdot)^\vee \circ \mathrm{HP}^{-1}\circ \gamma_{\rho,\psi}\cdot\circ \mathrm{HP}(f)
\end{align*}
which depends on both $\rho$ and $\psi$. Unless otherwise specified, we suppress the dependence on $\psi$ from the notation. 
We recall that 
Theorem \ref{thm:main} holds with this definition of the Fourier transform (see \S \ref{sec:FourieronHCSchwartz}).

We now show that any $G(F) \times G(F)$-invariant subspace of $\mathcal{S}_{\rho}^{\mathrm{as}}(G(F))<\mathcal{C}(G(F))$ stable under $\mathcal{F}_{\rho}$ and dense in $L^2(G(F))$ satisfies \eqref{Zetas}. For any tempered representation $\pi$ of $G(F)$, we denote by $\mathcal{C}(\pi)$ the space of smooth matrix coefficients of $\pi$. Given $f \in \mathcal{S}_{\rho}^{\mathrm{as}}(G(F))$ and $c\in \mathcal{C}(\pi)$, we consider the local zeta integral 
\[Z(s,f,c) := \int_{G(F)}f(g)|d(g)|^sc(g)dg\]
whenever the integral converges absolutely. Note that $|d(g)|^{s} = e^{\langle s\lambda, H_G(g)\rangle}$ for some $\lambda\in \mathfrak{a}^*_{G\mathbb{C}},$ and $L(\tfrac{1}{2},\pi_{s\lambda},\rho)=L(\tfrac{1}{2}+s,\pi,\rho)$ by our assumption on $d.$ Since $f\in \mathcal{S}^{\mathrm{as}}_{\rho}(G(F))$, Corollary \ref{zeta:integral:convergence:cone} implies that the zeta integral $Z(s,f,c)$ converges for $\mathrm{Re}(s)$ sufficiently large. Proposition \ref{prop:loc:zeta} and our assumption on $d$ imply that $Z(s,f,c)$ extends meromorphically to all $s$ and that 
$$\frac{Z(s,f,c)}{L(\frac{1}{2}+s,\pi,\rho)}
$$
is holomorphic as a function of $s.$ Moreover, Proposition \ref{prop:loc:zeta} also shows that 
\[Z(-s,\mathcal{F}_{\rho}(f),c^\vee) = \gamma(\tfrac{1}{2}+s,\pi,\rho)Z(s,f,c).\]
Now  $\mathcal{S}_{\rho}^{\mathrm{as}}(G(F))$ contains $ \mathcal{S}(G(F))$ by Lemma \ref{lem:Sas:moduleCcinfty}, hence it is dense in $L^2(G(F)).$  It is clearly $G(F) \times G(F)$-invariant, and it is stable under $\mathcal{F}_{\rho}$ by Lemma \ref{lem:as:preserve}.
Moreover \eqref{BKN:arch:BEV} follows immediately from the definition of $\mathcal{S}_{\rho}^{\mathrm{as}}(G(F)).$  This completes the proof when $F$ is Archimedean.

Assume for the remainder of the proof that $F$ is non-Archimedean.  Then $b_{\rho} \in \mathcal{S}_{\rho}^{\mathrm{as}}(G(F))$ by Lemma \ref{Basic:Function:Asymptotic}.
 By Lemma \ref{Basic:Function:Fixed:Fourier}, it is invariant under the Fourier transform $\mathcal{F}_{\rho}$. Furthermore, the definition of $b_{\rho}$ implies that $Z(s,b_{\rho},c)$ is non-zero only when $\pi$ is unramified, and in this case 
$Z(s,b_{\rho},c^{\circ}) = L(\tfrac{1}{2}+s,\pi,\rho),$
when $c^{\circ}$ is the zonal spherical function. This establishes \eqref{BKN:basic} of Conjecture \ref{conj:Sch:des}, and
completes the proof of Theorem \ref{main:thm:intro}. \qed

\section{Comparison with the work of Godement and Jacquet} 
\label{sec:GJ}
Let $F$ be a local field of characteristic zero. Let $G=\GL_{n}$ and $\rho_{\mathrm{st}}:{}^{L}\GL_n\to\GL_n(\mathbb{C})$ be the standard representation.  It is this case that motivated the conjectures of Braverman, Kazhdan and Ng\^o.  Bearing in mind our unitary normalization, the usual Schwartz space is
$$
\mathcal{S}(M_n(F))|\det|^{n/2} < C^\infty(\GL_n(F)).
$$
Here we view elements of $\mathcal{S}(M_n(F))|\det|^{n/2}$ as functions on $\GL_n(F)$ by restriction.  
The Fourier transform is 
\begin{align*} 
\mathcal{F}_{\psi}^{\mathrm{GJ}}:\mathcal{S}(M_n(F))|\det|^{n/2} &\lto \mathcal{S}(M_n(F))|\det|^{n/2}\\
f &\longmapsto |\det|^{n/2}\mathcal{F}_{\psi}(f|\det|^{-n/2})
\end{align*}
where $\mathcal{F}_{\psi}$ is the linear Fourier transform of \eqref{lin:FT}.  We drop the $\psi$ from notation for simplicity.  It is well-known that the linear Fourier transform extends to a unitary operator $\mathcal{F}:L^2(M_n(F)) \to L^2(M_n(F)).$   Here we define the $L^2$-spaces using the additive Haar measure on $M_n(F).$  Restriction to $\GL_n(F)$ induces an isomorphism
$$
L^2(M_n(F)) \tilde{\lto} L^2(\GL_n(F),|\det|^ndg)
$$
where $dg$ is a Haar measure on $\GL_n(F).$  Thus $\mathcal{F}^{\mathrm{GJ}}$ extends to a unitary operator $L^2(\GL_n(F)) \to L^2(\GL_n(F)).$

\begin{lem} \label{lem:HC:Planch} One has $\mathcal{S}(M_n(F))|\det|^{s+\frac{n-1}{2}} < \mathcal{C}(\GL_n(F))$ for any $\mathrm{Re}(s)>0$. 
\begin{proof} Since $\mathcal{S}(M_n(F))$ is a smooth representation of $\GL_n(F)\times \GL_n(F)$, it suffices to show that for each $\phi\in \mathcal{S}(M_n(F))$, $\mathrm{Re}(s)>0$ and $d\geq 0$ that
\begin{align*}
    p_d(|\det|^{s+\frac{n-1}{2}}\phi) <\infty.
\end{align*}

By Cartan decomposition, one has $\GL_n(F) = KAK$, where
\begin{align*}
    A = \{\mathrm{diag}(a_1,\ldots,a_n):a_i = \varpi^{k_i},\,k_1\geq \cdots \geq k_n\}
\end{align*}
if $F$ is non-Archimedean, and
\begin{align*}
    A = \{\mathrm{diag}(a_1,\ldots,a_n):0<a_1\leq \cdots\leq a_n\}.
\end{align*}
Let $B_n \leq \GL_n$ be the subgroup of upper triangular matrices. By \cite[Proposition 1.5.1]{BP:local:GGP},  for $a =\mathrm{diag}(a_1,\ldots,a_n)\in A$ one has that
\begin{align*}
    \Xi(a) \gg \delta_{B_n}(a)^{1/2} = |a_1|^{\frac{n-1}{2}}\cdots |a_n|^{\frac{1-n}{2}}.
\end{align*}
Thus writing $t:=\mathrm{Re}(s)$ we have
\begin{align*}
    (|\det|^{s+\frac{n-1}{2}}\Xi^{-1})(a) \ll |a_1|^{t}|a_2|^{t+1}\cdots |a_n|^{t+n-1}.
\end{align*}
For any $\varepsilon>0$ one has
\begin{align*}
    \sigma(a)^d \asymp (1+\log\max\{|a_1|^{-1},|a_n|\})^d \ll_{d,\varepsilon} 1+|a_1|^{-d\varepsilon} + |a_n|^{d\varepsilon}.
\end{align*} 
Therefore, by the Cartan decomposition 
\begin{align} \label{sup}
    p_d(|\det|^{s+\frac{n-1}{2}}\phi) &\ll_{d,\varepsilon} \sup_{k,k'\in K}\sup_{a\in A} |a_1|^t|a_2|^{t+1}\cdots |a_n|^{t+n-1}(1+|a_1|^{-d\varepsilon} + |a_n|^{d\varepsilon}) |\phi|(kak').
\end{align}
Choose $\varepsilon>0$ so that $t-d\varepsilon\geq 0$. When $F$ is non-Archimedean, the supremum over $k,k'$ only ranges over a finite set of $k,k'$, so this \eqref{sup} is finite. When $F$ is Archimedean,  the Schwartz bound for $|\phi|$ and the fact that $\norm{kak'}_{M_n(F)}\asymp\norm{a}_{M_n(F)}$ implies that \eqref{sup} is finite.
\end{proof}
\end{lem}  

\begin{prop}\label{GJ:st:Fourier}
One has $\mathcal{F}_{\rho_{\mathrm{st}}}=\mathcal{F}^{\mathrm{GJ}}$ as operators on $L^2(\GL_n(F)).$
\end{prop}

\begin{rem}
Proposition \ref{GJ:st:Fourier} was obtained in \cite[(3.1)]{Stein:SLnC} in the case $F=\CC$.
\end{rem}

\begin{proof}
We have already explained that $\mathcal{F}^{\mathrm{GJ}}$ is unitary, and $\mathcal{F}_{\rho_{\mathrm{st}}}$ is unitary by Lemma \ref{lem:FT:unit}.  
Since both operators are continuous, it suffices to check that they agree on some dense subspace $|\det|^{\frac{n}{2}}S\leq L^2(\GL_n(F)).$ When $F$ is non-Archimedean take $S = \mathcal{S}(M_n(F))$. When $F$ is Archimedean take $S$ to be the space of standard Schwartz functions (derivatives of the Gaussian) defined on \cite[p.~115]{GodementJacquetBook}.  Applying Lemma \ref{lem:HC:Planch} and the Harish-Chandra Plancherel formula 
(Theorem \ref{HCPlan:thm1}), for $f \in |\det|^{\frac{n}{2}}S$ we have
\begin{align*}
    \mathcal{F}^{\mathrm{GJ}}(f)(g) = (\mathcal{F}^{\mathrm{GJ}}(f)^{\vee})^\vee(g)=\int_{\mathrm{Temp}_{\mathrm{Ind}}(G)} \mathrm{tr}\left(\pi(g)\circ \pi(\mathcal{F}^{\mathrm{GJ}}(f)^\vee)\right) d\pi.
\end{align*}
By the Godement-Jacquet functional equation \cite[(9.5)]{GodementJacquetBook} this is 
\begin{align*}
    \int_{\mathrm{Temp}_{\mathrm{Ind}}(G)} \gamma(\tfrac{1}{2},\pi,\psi)\mathrm{tr}\left(\pi(g)\circ \pi(f)\right) d\pi=\mathcal{F}_{\rho_{st}}(f)(g)
\end{align*}
(see \eqref{FT:defn:2}).
\end{proof}

\bibliography{refs.bib}{}

@incollection {Borel:Corvallis,
    AUTHOR = {Borel, A.},
     TITLE = {Automorphic {$L$}-functions},
 BOOKTITLE = {Automorphic forms, representations and {$L$}-functions
              ({P}roc. {S}ympos. {P}ure {M}ath., {O}regon {S}tate {U}niv.,
              {C}orvallis, {O}re., 1977), {P}art 2},
    SERIES = {Proc. Sympos. Pure Math., XXXIII},
     PAGES = {27--61},
 PUBLISHER = {Amer. Math. Soc.},
   ADDRESS = {Providence, R.I.},
      YEAR = {1979},
   MRCLASS = {10D40 (12A67 22E50)},
  MRNUMBER = {MR546608 (81m:10056)},
MRREVIEWER = {Yasuhiro Asoo},
}

@ARTICLE{BBFK,
       author = {{Bezrukavnikov}, R. and {Braverman}, A. and {Finkelberg}, M. and {Kazhdan}, D.},
        title = "{Schwartz spaces, local L-factors and perverse sheaves}",
      journal = {arXiv e-prints},
         year = 2023,
        month = mar,
          eid = {arXiv:2303.00913},
        pages = {arXiv:2303.00913},
          doi = {10.48550/arXiv.2303.00913},
archivePrefix = {arXiv},
       eprint = {2303.00913},
 primaryClass = {math.RT},
       adsurl = {https://ui.adsabs.harvard.edu/abs/2023arXiv230300913B}
}

@article {BK-lifting,
    AUTHOR = {Braverman, A. and Kazhdan, D.},
     TITLE = {{$\gamma$}-functions of representations and lifting},
      NOTE = {With an appendix by V. Vologodsky,
              GAFA 2000 (Tel Aviv, 1999)},
   JOURNAL = {Geom. Funct. Anal.},
  FJOURNAL = {Geometric and Functional Analysis},
      YEAR = {2000},
    NUMBER = {Special Volume, Part I},
     PAGES = {237--278},
      ISSN = {1016-443X},
     CODEN = {GFANFB},
   MRCLASS = {11F70 (11R39 11S37 11S40 22E55)},
  MRNUMBER = {1826255 (2002g:11064)},
MRREVIEWER = {Volker J. Heiermann},
       DOI = {10.1007/978-3-0346-0422-2_9},
       URL = {http://dx.doi.org/10.1007/978-3-0346-0422-2_9},
}

@preamble{
   "\def\polhk#1{\setbox0=\hbox{#1}{\ooalign{\hidewidth
    \lower1.5ex\hbox{`}\hidewidth\crcr\unhbox0}}} "
}

@article {Getz:Liu:BK,
    AUTHOR = {Getz, J. R. and Liu, B.},
     TITLE = {A refined {P}oisson summation formula for certain
              {B}raverman-{K}azhdan spaces},
   JOURNAL = {Sci. China Math.},
  FJOURNAL = {Science China. Mathematics},
    VOLUME = {64},
      YEAR = {2021},
    NUMBER = {6},
     PAGES = {1127--1156},
      ISSN = {1674-7283},
   MRCLASS = {11F70 (11F66)},
  MRNUMBER = {4268887},
MRREVIEWER = {Solomon Friedberg},
       DOI = {10.1007/s11425-018-1616-0},
       URL = {https://doi.org/10.1007/s11425-018-1616-0},
}

@book {Getz:Hahn,
    AUTHOR = {Getz, J.~R. and Hahn, H.},
     TITLE = {An introduction to automorphic representations---with a view
              toward trace formulae},
    SERIES = {Graduate Texts in Mathematics},
    VOLUME = {300},
 PUBLISHER = {Springer, Cham},
      YEAR = {[2024] \copyright 2024},
     PAGES = {xviii+609},
      ISBN = {978-3-031-41151-9; 978-3-031-41153-3},
   MRCLASS = {11Fxx (11-02 22-02 22Exx)},
  MRNUMBER = {4738301},
       DOI = {10.1007/978-3-031-41153-3},
       URL = {https://doi.org/10.1007/978-3-031-41153-3},
}

@book {GodementJacquetBook,
    AUTHOR = {Godement, R. and Jacquet, H.},
     TITLE = {Zeta functions of simple algebras},
    SERIES = {Lecture Notes in Mathematics, Vol. 260},
 PUBLISHER = {Springer-Verlag, Berlin-New York},
      YEAR = {1972},
     PAGES = {ix+188},
   MRCLASS = {12A80 (12A65 12A70 12B35 22E55)},
  MRNUMBER = {0342495},
MRREVIEWER = {L. Corwin},
}

@article {JPSS:Conv,
    AUTHOR = {Jacquet, H. and Piatetskii-Shapiro, I. I. and Shalika, J. A.},
     TITLE = {Rankin-{S}elberg convolutions},
   JOURNAL = {Amer. J. Math.},
  FJOURNAL = {American Journal of Mathematics},
    VOLUME = {105},
      YEAR = {1983},
    NUMBER = {2},
     PAGES = {367--464},
      ISSN = {0002-9327},
   MRCLASS = {11F67 (11F70 11R39 22E55)},
  MRNUMBER = {701565},
MRREVIEWER = {F. Shahidi},
       DOI = {10.2307/2374264},
       URL = {https://doi.org/10.2307/2374264},
}

@book {KnappSS,
    AUTHOR = {Knapp, A.~W.},
     TITLE = {Representation theory of semisimple groups},
    SERIES = {Princeton Landmarks in Mathematics},
      NOTE = {An overview based on examples,
              Reprint of the 1986 original},
 PUBLISHER = {Princeton University Press, Princeton, NJ},
      YEAR = {2001},
     PAGES = {xx+773},
      ISBN = {0-691-09089-0},
   MRCLASS = {22E46 (22-01 22E30)},
  MRNUMBER = {1880691},
}

@article {LafforgueJJM,
    AUTHOR = {Lafforgue, L.},
     TITLE = {Noyaux du transfert automorphe de {L}anglands et formules de
              {P}oisson non lin\'eaires},
   JOURNAL = {Jpn. J. Math.},
  FJOURNAL = {Japanese Journal of Mathematics},
    VOLUME = {9},
      YEAR = {2014},
    NUMBER = {1},
     PAGES = {1--68},
      ISSN = {0289-2316},
   MRCLASS = {11F66 (11F03 11F70)},
  MRNUMBER = {3173438},
MRREVIEWER = {Ravi Raghunathan},
       DOI = {10.1007/s11537-014-1274-y},
       URL = {http://dx.doi.org/10.1007/s11537-014-1274-y},
}

@Article{WWLiSat,
author="Li, W-W.",
title="Basic functions and unramified local {L}-factors for split groups",
journal="Science China Mathematics",
year="2016",
pages="1--36",
issn="1869-1862",
doi="10.1007/s11425-015-0730-4",
url="http://dx.doi.org/10.1007/s11425-015-0730-4"
}

@book {Moreno:Advanced:analytic:number:theory,
    AUTHOR = {Moreno, C. Julio},
     TITLE = {Advanced analytic number theory: {$L$}-functions},
    SERIES = {Mathematical Surveys and Monographs},
    VOLUME = {115},
 PUBLISHER = {American Mathematical Society, Providence, RI},
      YEAR = {2005},
     PAGES = {xx+291},
      ISBN = {0-8218-3641-2},
   MRCLASS = {11Mxx (11F66 11M36 11R42)},
  MRNUMBER = {2135107},
MRREVIEWER = {Philippe G. Michel},
}

@article {Waldspurger:plancherel,
    AUTHOR = {Waldspurger, J.-L.},
     TITLE = {La formule de {P}lancherel pour les groupes {$p$}-adiques
              (d'apr\`es {H}arish-{C}handra)},
   JOURNAL = {J. Inst. Math. Jussieu},
  FJOURNAL = {Journal of the Institute of Mathematics of Jussieu. JIMJ.
              Journal de l'Institut de Math\'{e}matiques de Jussieu},
    VOLUME = {2},
      YEAR = {2003},
    NUMBER = {2},
     PAGES = {235--333},
      ISSN = {1474-7480,1475-3030},
   MRCLASS = {22E35 (22E50)},
  MRNUMBER = {1989693},
MRREVIEWER = {Rebecca\ Herb},
       DOI = {10.1017/S1474748003000082},
       URL = {https://doi.org/10.1017/S1474748003000082},
}

@incollection {NgoSums,
    AUTHOR = {Ng{\^o}, B. C.},
     TITLE = {On a certain sum of automorphic {$L$}-functions},
 BOOKTITLE = {Automorphic forms and related geometry: assessing the legacy
              of {I}. {I}. {P}iatetski-{S}hapiro},
    SERIES = {Contemp. Math.},
    VOLUME = {614},
     PAGES = {337--343},
 PUBLISHER = {Amer. Math. Soc., Providence, RI},
      YEAR = {2014},
   MRCLASS = {11F70},
  MRNUMBER = {3220933},
MRREVIEWER = {Fan Gao},
       DOI = {10.1090/conm/614/12270},
       URL = {http://dx.doi.org/10.1090/conm/614/12270},
}

@Book{Renard,
 Author = {D. {Renard}},
 Title = {{Repr\'esentations des groupes r\'eductifs \(p\)-adiques}},
 FJournal = {{Cours Sp\'ecialis\'es (Paris)}},
 Journal = {{Cours Sp\'ec. (Paris)}},
 ISSN = {1284-6090},
 Volume = {17},
 ISBN = {978-2-85629-278-5/hbk},
 Pages = {vi + 332},
 Year = {2010},
 Publisher = {Paris: Soci\'et\'e Math\'ematique de France},
 Language = {French},
 MSC2010 = {22E50 22-02 11F70},
 Zbl = {1186.22020}
}

@incollection {Tate_NT,
    AUTHOR = {Tate, J.},
     TITLE = {Number theoretic background},
 BOOKTITLE = {Automorphic forms, representations and {$L$}-functions
              ({P}roc. {S}ympos. {P}ure {M}ath., {O}regon {S}tate {U}niv.,
              {C}orvallis, {O}re., 1977), {P}art 2},
    SERIES = {Proc. Sympos. Pure Math., XXXIII},
     PAGES = {3--26},
 PUBLISHER = {Amer. Math. Soc., Providence, R.I.},
      YEAR = {1979},
   MRCLASS = {12A67},
  MRNUMBER = {546607 (80m:12009)},
MRREVIEWER = {A. I. Vinogradov},
}

@book {WallachRG1,
    AUTHOR = {Wallach, N. R.},
     TITLE = {Real reductive groups. {I}},
    SERIES = {Pure and Applied Mathematics},
    VOLUME = {132},
 PUBLISHER = {Academic Press, Inc., Boston, MA},
      YEAR = {1988},
     PAGES = {xx+412},
      ISBN = {0-12-732960-9},
   MRCLASS = {22E46 (17B10 22-02 22E30)},
  MRNUMBER = {929683 (89i:22029)},
MRREVIEWER = {R.o J. Miatello},
}

@book {Wallach:RGII,
    AUTHOR = {Wallach, N. R.},
     TITLE = {Real reductive groups. {II}},
    SERIES = {Pure and Applied Mathematics},
    VOLUME = {132},
 PUBLISHER = {Academic Press, Inc., Boston, MA},
      YEAR = {1992},
     PAGES = {xiv+454},
      ISBN = {0-12-732961-7},
   MRCLASS = {22E46 (22D25 22E30 46L99)},
  MRNUMBER = {1170566},
MRREVIEWER = {Jorge A. Vargas},
}

@book {Warner:semisimple:I,
    AUTHOR = {Warner, G.},
     TITLE = {Harmonic analysis on semi-simple {L}ie groups. {I}},
      NOTE = {Die Grundlehren der mathematischen Wissenschaften, Band 188},
 PUBLISHER = {Springer-Verlag, New York-Heidelberg},
      YEAR = {1972},
     PAGES = {xvi+529},
   MRCLASS = {22E45 (22E30)},
  MRNUMBER = {0498999},
MRREVIEWER = {Mitsuo Sugiura},
}

@article {Zelevinsky:IndII,
    AUTHOR = {Zelevinsky, A. V.},
     TITLE = {Induced representations of reductive {$\mathfrak{p}$}-adic
              groups. {II}. {O}n irreducible representations of {${\rm
              GL}(n)$}},
   JOURNAL = {Ann. Sci. \'Ecole Norm. Sup. (4)},
  FJOURNAL = {Annales Scientifiques de l'\'Ecole Normale Sup\'erieure.
              Quatri\`eme S\'erie},
    VOLUME = {13},
      YEAR = {1980},
    NUMBER = {2},
     PAGES = {165--210},
      ISSN = {0012-9593},
     CODEN = {ASENAH},
   MRCLASS = {22E50},
  MRNUMBER = {584084},
       URL = {http://www.numdam.org/item?id=ASENS_1980_4_13_2_165_0},
}

@article {Ngo:Hankel,
    AUTHOR = {Ng{\^{o}}, B. C.},
     TITLE = {Hankel transform, {L}anglands functoriality and functional
              equation of automorphic {$L$}-functions},
   JOURNAL = {Jpn. J. Math.},
  FJOURNAL = {Japanese Journal of Mathematics},
    VOLUME = {15},
      YEAR = {2020},
    NUMBER = {1},
     PAGES = {121--167},
      ISSN = {0289-2316},
   MRCLASS = {22E35 (11F66)},
  MRNUMBER = {4068833},
       DOI = {10.1007/s11537-019-1650-8},
       URL = {https://doi.org/10.1007/s11537-019-1650-8},
}

@ARTICLE{Luo:Ngo,
       author = {{Luo}, Z. and {Ng{\^o}}, B. C.},
        title = "{Nonabelian Fourier Kernels on $\mathrm{SL}_2$ and $\mathrm{GL}_2$}",
      journal = {arXiv e-prints},
         year = 2024,
        month = sep,
          eid = {arXiv:2409.14696},
        pages = {arXiv:2409.14696},
          doi = {10.48550/arXiv.2409.14696},
archivePrefix = {arXiv},
       eprint = {2409.14696},
 primaryClass = {math.NT},
       adsurl = {https://ui.adsabs.harvard.edu/abs/2024arXiv240914696L}
}

@article {BP:local:GGP,
    AUTHOR = {Beuzart-Plessis, R.},
     TITLE = {A local trace formula for the {G}an-{G}ross-{P}rasad
              conjecture for unitary groups: the {A}rchimedean case},
   JOURNAL = {Ast\'{e}risque},
  FJOURNAL = {Ast\'{e}risque},
    NUMBER = {418},
      YEAR = {2020},
     PAGES = {ix+305},
      ISSN = {0303-1179},
      ISBN = {978-2-85629-919-7},
   MRCLASS = {22E50 (11F72 11F85 20G05)},
  MRNUMBER = {4146145},
MRREVIEWER = {Wen-Wei Li},
       DOI = {10.24033/ast},
       URL = {https://doi.org/10.24033/ast},
}

@Article{AdamsVoganContragredient,
 Author = {Adams, J. and Vogan, D. A.},
 Title = {Contragredient representations and characterizing the local {Langlands} correspondence},
 FJournal = {American Journal of Mathematics},
 Journal = {Am. J. Math.},
 ISSN = {0002-9327},
 Volume = {138},
 Number = {3},
 Pages = {657--682},
 Year = {2016},
 Language = {English},
 DOI = {10.1353/ajm.2016.0024},
 URL = {hdl.handle.net/1721.1/116073},
 zbMATH = {6603603},
 Zbl = {1354.11039}
}

@article {Delorme:PaleyWiener,
    AUTHOR = {Delorme, P.},
     TITLE = {Sur le th\'eor\`eme de {P}aley-{W}iener d'{A}rthur},
   JOURNAL = {Ann. of Math. (2)},
  FJOURNAL = {Annals of Mathematics. Second Series},
    VOLUME = {162},
      YEAR = {2005},
    NUMBER = {2},
     PAGES = {987--1029},
      ISSN = {0003-486X,1939-8980},
   MRCLASS = {22E45},
  MRNUMBER = {2183287},
MRREVIEWER = {Nils\ Byrial\ Andersen},
       DOI = {10.4007/annals.2005.162.987},
       URL = {https://doi.org/10.4007/annals.2005.162.987},
}

@article{Luo:basicfunction,
    AUTHOR = {Luo, Z.~},
     TITLE = {On the {B}raverman-{K}azhdan proposal for local factors:
              spherical case},
   JOURNAL = {Pacific J. Math.},
  FJOURNAL = {Pacific Journal of Mathematics},
    VOLUME = {300},
      YEAR = {2019},
    NUMBER = {2},
     PAGES = {431--471},
      ISSN = {0030-8730,1945-5844},
   MRCLASS = {11F70 (22E35 22E45)},
  MRNUMBER = {3990815},
MRREVIEWER = {Chen\ Wan},
       DOI = {10.2140/pjm.2019.300.431},
       URL = {https://doi.org/10.2140/pjm.2019.300.431},
}

@article{Arthur:HP:realreductive,
  title={Harmonic analysis of the {S}chwartz space on a reductive {L}ie group {I} \& {II}},
  author={Arthur, J.},
  url={https://www.math.toronto.edu/arthur/pdf/20160115093856221.pdf}
}

@book {Voison:ComplexAG,
    AUTHOR = {Voisin, C.},
     TITLE = {Hodge theory and complex algebraic geometry. {I}},
    SERIES = {Cambridge Studies in Advanced Mathematics},
    VOLUME = {76},
   EDITION = {English},
      NOTE = {Translated from the French by Leila Schneps},
 PUBLISHER = {Cambridge University Press, Cambridge},
      YEAR = {2007},
     PAGES = {x+322},
      ISBN = {978-0-521-71801-1},
   MRCLASS = {32J25 (14C30 14D07 32G20)},
  MRNUMBER = {2451566},
}

@article {Casselman:SchwartzSpace,
    AUTHOR = {Casselman, W.},
     TITLE = {Introduction to the {S}chwartz space of {$\Gamma\backslash
              G$}},
   JOURNAL = {Canad. J. Math.},
  FJOURNAL = {Canadian Journal of Mathematics. Journal Canadien de
              Math\'ematiques},
    VOLUME = {41},
      YEAR = {1989},
    NUMBER = {2},
     PAGES = {285--320},
      ISSN = {0008-414X,1496-4279},
   MRCLASS = {22E40 (11F72 22E45 22E50)},
  MRNUMBER = {1001613},
MRREVIEWER = {Takao\ Watanabe},
       DOI = {10.4153/CJM-1989-015-6},
       URL = {https://doi.org/10.4153/CJM-1989-015-6},
}

@article {Grothendieck:holomorphic:I,
    AUTHOR = {Grothendieck, A.},
     TITLE = {Sur certains espaces de fonctions holomorphes. {I}},
   JOURNAL = {J. Reine Angew. Math.},
  FJOURNAL = {Journal f\"ur die Reine und Angewandte Mathematik. [Crelle's
              Journal]},
    VOLUME = {192},
      YEAR = {1953},
     PAGES = {35--64},
      ISSN = {0075-4102,1435-5345},
   MRCLASS = {46.3X},
  MRNUMBER = {58865},
MRREVIEWER = {A.\ E.\ Taylor},
       DOI = {10.1515/crll.1953.192.35},
       URL = {https://doi.org/10.1515/crll.1953.192.35},
}

@article{Pouslen:smoothInd,
    AUTHOR = {Poulsen, N.~S.},
     TITLE = {On {$C\sp{\infty }$}-vectors and intertwining bilinear forms
              for representations of {L}ie groups},
   JOURNAL = {J. Functional Analysis},
  FJOURNAL = {Journal of Functional Analysis},
    VOLUME = {9},
      YEAR = {1972},
     PAGES = {87--120},
      ISSN = {0022-1236},
   MRCLASS = {22E45 (22D30)},
  MRNUMBER = {310137},
MRREVIEWER = {R.\ Goodman},
       DOI = {10.1016/0022-1236(72)90016-x},
       URL = {https://doi.org/10.1016/0022-1236(72)90016-x},
}

@incollection{Shelstad:realgroup,
    AUTHOR = {Shelstad, D.},
     TITLE = {Orbital integrals, endoscopic groups and
              {$L$}-indistinguishability for real groups},
 BOOKTITLE = {Conference on automorphic theory ({D}ijon, 1981)},
    SERIES = {Publ. Math. Univ. Paris VII},
    VOLUME = {15},
     PAGES = {135--219},
 PUBLISHER = {Univ. Paris VII, Paris},
      YEAR = {1983},
   MRCLASS = {22E45 (11F70 11F72)},
  MRNUMBER = {723184},
MRREVIEWER = {Joe\ Repka},
}

@incollection {Langlands:ArchLLC,
    AUTHOR = {Langlands, R. P.},
     TITLE = {On the classification of irreducible representations of real
              algebraic groups},
 BOOKTITLE = {Representation theory and harmonic analysis on semisimple
              {L}ie groups},
    SERIES = {Math. Surveys Monogr.},
    VOLUME = {31},
     PAGES = {101--170},
 PUBLISHER = {Amer. Math. Soc., Providence, RI},
      YEAR = {1989},
      ISBN = {0-8218-1526-1},
   MRCLASS = {22E45 (22E46)},
  MRNUMBER = {1011897},
MRREVIEWER = {D.\ Mili\v ci\'c},
       DOI = {10.1090/surv/031/03},
       URL = {https://doi.org/10.1090/surv/031/03},
}

@incollection {Jacquet:GLn:Corvallis,
    AUTHOR = {Jacquet, H.},
     TITLE = {Principal {$L$}-functions of the linear group},
 BOOKTITLE = {Automorphic forms, representations and {$L$}-functions
              ({P}roc. {S}ympos. {P}ure {M}ath., {O}regon {S}tate {U}niv.,
              {C}orvallis, {O}re., 1977), {P}art 2},
    SERIES = {Proc. Sympos. Pure Math.},
    VOLUME = {XXXIII},
     PAGES = {63--86},
 PUBLISHER = {Amer. Math. Soc., Providence, RI},
      YEAR = {1979},
      ISBN = {0-8218-1437-0},
   MRCLASS = {22E50 (12A82)},
  MRNUMBER = {546609},
MRREVIEWER = {Stephen\ Gelbart},
}

@incollection {Casselman:realtori,
    AUTHOR = {Casselman, W.},
     TITLE = {Computations in real tori},
 BOOKTITLE = {Representation theory of real reductive {L}ie groups},
    SERIES = {Contemp. Math.},
    VOLUME = {472},
     PAGES = {137--151},
 PUBLISHER = {Amer. Math. Soc., Providence, RI},
      YEAR = {2008},
      ISBN = {978-0-8218-4366-6},
   MRCLASS = {22E15 (22E47)},
  MRNUMBER = {2454333},
MRREVIEWER = {Volker\ J.\ Heiermann},
       DOI = {10.1090/conm/472/09238},
       URL = {https://doi.org/10.1090/conm/472/09238},
}

@book {HT:LLC,
    AUTHOR = {Harris, M. and Taylor, R.},
     TITLE = {The geometry and cohomology of some simple {S}himura
              varieties},
    SERIES = {Annals of Mathematics Studies},
    VOLUME = {151},
      NOTE = {With an appendix by Vladimir G. Berkovich},
 PUBLISHER = {Princeton University Press, Princeton, NJ},
      YEAR = {2001},
     PAGES = {viii+276},
      ISBN = {0-691-09090-4},
   MRCLASS = {11G18 (11F70 11S37 14G35 22E45)},
  MRNUMBER = {1876802},
MRREVIEWER = {J. Milne},
}

@article {Henniart:preuve,
    AUTHOR = {Henniart, G.},
     TITLE = {Une preuve simple des conjectures de {L}anglands pour {${\rm
              GL}(n)$} sur un corps {$p$}-adique},
   JOURNAL = {Invent. Math.},
  FJOURNAL = {Inventiones Mathematicae},
    VOLUME = {139},
      YEAR = {2000},
    NUMBER = {2},
     PAGES = {439--455},
      ISSN = {0020-9910},
   MRCLASS = {11F70 (11R39 11S37 22E50 22E55)},
  MRNUMBER = {1738446},
MRREVIEWER = {Dihua Jiang},
       DOI = {10.1007/s002220050012},
       URL = {https://doi.org/10.1007/s002220050012},
}

@ARTICLE{FS,
       author = {{Fargues}, L. and {Scholze}, P.},
        title = "{Geometrization of the local Langlands correspondence}",
      journal = {arXiv e-prints},
         year = 2021,
        month = feb,
          eid = {arXiv:2102.13459},
        pages = {arXiv:2102.13459},
          doi = {10.48550/arXiv.2102.13459},
archivePrefix = {arXiv},
       eprint = {2102.13459},
 primaryClass = {math.RT},
       adsurl = {https://ui.adsabs.harvard.edu/abs/2021arXiv210213459F}
}

@book {Bushnell:Henniart,
    AUTHOR = {Bushnell, C. J. and Henniart, G.},
     TITLE = {The local {L}anglands conjecture for {$\rm GL(2)$}},
    SERIES = {Grundlehren der mathematischen Wissenschaften [Fundamental
              Principles of Mathematical Sciences]},
    VOLUME = {335},
 PUBLISHER = {Springer-Verlag, Berlin},
      YEAR = {2006},
     PAGES = {xii+347},
      ISBN = {978-3-540-31486-8; 3-540-31486-5},
   MRCLASS = {22E50 (11-02 11S37 22-02)},
  MRNUMBER = {2234120},
MRREVIEWER = {Alexandru Ioan Badulescu},
       DOI = {10.1007/3-540-31511-X},
       URL = {https://doi.org/10.1007/3-540-31511-X},
}

@incollection {Buzzard:Gee,
    AUTHOR = {Buzzard, K. and Gee, T.},
     TITLE = {The conjectural connections between automorphic
              representations and {G}alois representations},
 BOOKTITLE = {Automorphic forms and {G}alois representations. {V}ol. 1},
    SERIES = {London Math. Soc. Lecture Note Ser.},
    VOLUME = {414},
     PAGES = {135--187},
 PUBLISHER = {Cambridge Univ. Press, Cambridge},
      YEAR = {2014},
   MRCLASS = {11F33},
  MRNUMBER = {3444225},
MRREVIEWER = {Jeremy A. Rouse},
       DOI = {10.1017/CBO9781107446335.006},
       URL = {https://doi.org/10.1017/CBO9781107446335.006},
}

@article {Kruse,
    AUTHOR = {Kruse, K.},
     TITLE = {Vector-valued holomorphic functions in several variables},
   JOURNAL = {Funct. Approx. Comment. Math.},
  FJOURNAL = {Uniwersytet im. Adama Mickiewicza w Poznaniu. Wydzia\l 
              Matematyki i Informatyki. Functiones et Approximatio
              Commentarii Mathematici},
    VOLUME = {63},
      YEAR = {2020},
    NUMBER = {2},
     PAGES = {247--275},
      ISSN = {0208-6573},
   MRCLASS = {46E40 (32A10 46E10)},
  MRNUMBER = {4184275},
       DOI = {10.7169/facm/1861},
       URL = {https://doi.org/10.7169/facm/1861},
}

@book {DG,
    AUTHOR = {Demazure, M. and Gabriel, P.},
     TITLE = {Introduction to algebraic geometry and algebraic groups},
    SERIES = {North-Holland Mathematics Studies},
    VOLUME = {39},
      NOTE = {Translated from the French by J. Bell},
 PUBLISHER = {North-Holland Publishing Co., Amsterdam-New York},
      YEAR = {1980},
     PAGES = {xiv+357},
      ISBN = {0-444-85443-6},
   MRCLASS = {14-01},
  MRNUMBER = {563524},
}

@article {Delorme:constant,
    AUTHOR = {Delorme, P.},
     TITLE = {Constant term of smooth {$H_\psi$}-spherical functions on a
              reductive {$p$}-adic group},
   JOURNAL = {Trans. Amer. Math. Soc.},
  FJOURNAL = {Transactions of the American Mathematical Society},
    VOLUME = {362},
      YEAR = {2010},
    NUMBER = {2},
     PAGES = {933--955},
      ISSN = {0002-9947},
   MRCLASS = {22E50},
  MRNUMBER = {2551511},
       DOI = {10.1090/S0002-9947-09-04925-3},
       URL = {https://doi.org/10.1090/S0002-9947-09-04925-3},
}

@incollection {BK:asymp,
    AUTHOR = {Braverman, A. and Kazhdan, D.},
     TITLE = {Remarks on the asymptotic {H}ecke algebra},
 BOOKTITLE = {Lie groups, geometry, and representation theory},
    SERIES = {Progr. Math.},
    VOLUME = {326},
     PAGES = {91--108},
 PUBLISHER = {Birkh\"{a}user/Springer, Cham},
      YEAR = {2018},
   MRCLASS = {20C08 (20C11 20G25 22D10 22D20 22E50)},
  MRNUMBER = {3890206},
MRREVIEWER = {Mehmet Haluk \c{S}eng\"{u}n},
       DOI = {10.1007/978-3-030-02191-7\_4},
       URL = {https://doi.org/10.1007/978-3-030-02191-7_4},
}

@article {Heiermann:unip,
    AUTHOR = {Heiermann, Volker},
     TITLE = {Orbites unipotentes et p\^{o}les d'ordre maximal de la fonction
              {$\mu$} de {H}arish-{C}handra},
   JOURNAL = {Canad. J. Math.},
  FJOURNAL = {Canadian Journal of Mathematics. Journal Canadien de
              Math\'{e}matiques},
    VOLUME = {58},
      YEAR = {2006},
    NUMBER = {6},
     PAGES = {1203--1228},
      ISSN = {0008-414X},
   MRCLASS = {11F70 (11F80 22E50)},
  MRNUMBER = {2270924},
MRREVIEWER = {David Goldberg},
       DOI = {10.4153/CJM-2006-043-8},
       URL = {https://doi.org/10.4153/CJM-2006-043-8},
}

@article {Gross:Reeder,
    AUTHOR = {Gross, B.~H. and Reeder, M.},
     TITLE = {Arithmetic invariants of discrete {L}anglands parameters},
   JOURNAL = {Duke Math. J.},
  FJOURNAL = {Duke Mathematical Journal},
    VOLUME = {154},
      YEAR = {2010},
    NUMBER = {3},
     PAGES = {431--508},
      ISSN = {0012-7094},
   MRCLASS = {11S37 (11S15 22E50)},
  MRNUMBER = {2730575},
MRREVIEWER = {Joshua M. Lansky},
       DOI = {10.1215/00127094-2010-043},
       URL = {https://doi.org/10.1215/00127094-2010-043},
}

@book {Schaefer:Wolff,
    AUTHOR = {Schaefer, H. H. and Wolff, M. P.},
     TITLE = {Topological vector spaces},
    SERIES = {Graduate Texts in Mathematics},
    VOLUME = {3},
   EDITION = {Second},
 PUBLISHER = {Springer-Verlag, New York},
      YEAR = {1999},
     PAGES = {xii+346},
      ISBN = {0-387-98726-6},
   MRCLASS = {46-01 (47-01)},
  MRNUMBER = {1741419},
       DOI = {10.1007/978-1-4612-1468-7},
       URL = {https://doi.org/10.1007/978-1-4612-1468-7},
}

@article {BNS,
    AUTHOR = {Bouthier, A. and Ng\^{o}, B. C. and Sakellaridis, Y.},
     TITLE = {On the formal arc space of a reductive monoid},
   JOURNAL = {Amer. J. Math.},
  FJOURNAL = {American Journal of Mathematics},
    VOLUME = {138},
      YEAR = {2016},
    NUMBER = {1},
     PAGES = {81--108},
      ISSN = {0002-9327,1080-6377},
   MRCLASS = {14E18 (11F70)},
  MRNUMBER = {3462881},
MRREVIEWER = {Julien\ Sebag},
       DOI = {10.1353/ajm.2016.0004},
       URL = {https://doi.org/10.1353/ajm.2016.0004},
}

@ARTICLE{H:Y,
       author = {{Hsu}, Chun-Hsien and {Yao}, HaoYun},
        title = "{Schwartz spaces on L-monoids: non-Archimedean}",
      journal = {arXiv e-prints},
         year = 2026,
        month = jul,
          eid = {arXiv:2607.28507},
        pages = {arXiv:2607.28507},
          doi = {10.48550/arXiv.2607.28507},
archivePrefix = {arXiv},
       eprint = {2607.28507},
 primaryClass = {math.NT},
       adsurl = {https://ui.adsabs.harvard.edu/abs/2026arXiv260728507H}
}

@article {Stein:SLnC,
    AUTHOR = {Stein, E. M.},
     TITLE = {Analysis in matrix spaces and some new representations of
              {${\rm SL}(N,\,C)$}},
   JOURNAL = {Ann. of Math. (2)},
  FJOURNAL = {Annals of Mathematics. Second Series},
    VOLUME = {86},
      YEAR = {1967},
     PAGES = {461--490},
      ISSN = {0003-486X},
   MRCLASS = {22.65},
  MRNUMBER = {219670},
MRREVIEWER = {R.\ G.\ Laha},
       DOI = {10.2307/1970611},
       URL = {https://doi.org/10.2307/1970611},
}

@article {Lafforgue:transfer,
    AUTHOR = {Lafforgue, L.},
     TITLE = {Du transfert automorphe de {L}anglands aux formules de
              {P}oisson non lin\'eaires},
   JOURNAL = {Ann. Inst. Fourier (Grenoble)},
  FJOURNAL = {Universit\'e{} de Grenoble. Annales de l'Institut Fourier},
    VOLUME = {66},
      YEAR = {2016},
    NUMBER = {3},
     PAGES = {899--1012},
      ISSN = {0373-0956,1777-5310},
   MRCLASS = {11R58 (11F70 11M06 11R56 22E55 22E57 43A30)},
  MRNUMBER = {3494164},
MRREVIEWER = {Ravi\ Raghunathan},
       DOI = {10.5802/aif.3029},
       URL = {https://doi.org/10.5802/aif.3029},
}
\bibliographystyle{alpha}

\end{document}